\documentclass[smallextended]{svjour3}       % 

\smartqed  % flush right qed marks, e.g. at end of proof

\usepackage{graphicx}
\usepackage{mathptmx}      % use Times fonts if available on your TeX system
\usepackage{glossaries}
\usepackage{amssymb}
\usepackage{mathrsfs}                  
\usepackage{enumitem}
\usepackage{xcolor}
\usepackage{mathtools}

\usepackage{bbm}
\usepackage{bm}
\usepackage[hyperfootnotes=false]{hyperref}
\spnewtheorem{assumption}{Assumption}{\bf}{\it}

\newcommand{\E}{\mathbb{E}}
\newcommand{\R}{\mathbb{R}}
\newcommand{\C}{\mathbb{C}}
\newcommand{\Z}{\mathbb{Z}}
\newcommand{\N}{\mathbb{N}}
\newcommand{\M}{\mathbb{M}}
\newcommand{\tr}{\mathrm{tr}}

\renewcommand{\P}{\mathbb{P}}

\newcommand{\Fc}{\mathcal{F}}

\newcommand{\Hc}{\mathcal{H}}

\renewcommand{\d}{\mathrm{d}}%

\begin{document}

\title{Approximation Bounds for Random Neural Networks and Reservoir Systems}

%\thanks{}


\author{Lukas Gonon \and Lyudmila Grigoryeva \and Juan-Pablo Ortega  }


\institute{L. Gonon \at
              Faculty of Mathematics and Statistics,
              Universit\"at Sankt Gallen,
              Switzerland \\
              \email{lukas.gonon@unisg.ch}           %  \\
           \and
           L. Grigoryeva \at
              Department of Mathematics and Statistics,
              Graduate School of Decision Sciences,
              Universit\"at Konstanz,
              Germany \\ \email{Lyudmila.Grigoryeva@uni-konstanz.de}
            \and J.-P. Ortega \at
              Faculty of Mathematics and Statistics,
              Universit\"at Sankt Gallen,
              Switzerland and CNRS, France  \\
              \email{Juan-Pablo.Ortega@unisg.ch}           %  \\
}

\date{Received: date / Accepted: date}

\maketitle

\begin{abstract}
This work studies approximation based on single-hidden-layer feedforward and recurrent neural networks with randomly generated internal weights. These methods, in which only the last layer of weights and a few hyperparameters are optimized, have been successfully applied in a wide range of static and dynamic learning problems.
Despite the popularity of this approach in empirical tasks, important theoretical questions regarding the relation between the unknown function, the weight distribution, and the approximation rate have remained open. In this work it is proved that, as long as the unknown function, functional, or dynamical system is sufficiently regular, it is possible to draw the internal weights of the random (recurrent) neural network from a generic distribution (not depending on the unknown object) and quantify the error in terms of the number of neurons and the hyperparameters.
In particular, this proves that echo state networks with randomly generated weights are capable of approximating a wide class of dynamical systems arbitrarily well and thus provides the first mathematical explanation for their empirically observed success at learning dynamical systems. 

\end{abstract}

\section{Introduction}

This article studies the approximation of an unknown map $H^* \colon \mathcal{X} \to \R^m$ by a random (recurrent) neural network. More specifically, when $\mathcal{X} = \R^q$ we study approximations of the function $H^*$ by single-hidden-layer feedforward neural networks  $H^{{\bf A},\bm{\zeta}}_{\bf W}({\bf z})= {\bf W} \bm{\sigma}({\bf A} {\bf z} + \bm{\zeta})$ with ${\bf A} \in \mathbb{M}_{N, q}, \bm{\zeta} \in \R^{N}$ randomly drawn (not using any knowledge about $H^*$), $\boldsymbol{\sigma}: \mathbb{R}^N\longrightarrow \mathbb{R}^N$ a given activation function (obtained as the componentwise application of a map $\sigma: \mathbb{R} \longrightarrow \mathbb{R}$) and ${\bf W} \in \mathbb{M}_{m, N}$ a matrix that can be trained in order to approximate $H^*$ as well as possible. Random neural networks of this type have been applied very successfully in a variety of settings, we refer in particular to the seminal works on random feature models \cite{Rahimi2007} and Extreme Learning Machines \cite{Huang2006}. We refer to this case as the \textit{static} situation and will come back to it later on. In contrast, we speak about the \textit{dynamic} situation when $H^*$ takes as inputs sequences, i.e.\ $\mathcal{X} \subset (\R^d)^{\Z_-}$.

A particularly important family of approximants that we study in the dynamic situation are {\it reservoir systems}, that is, $H({\bf z}) = {\bf y}_0$ for ${\bf z} \in \mathcal{X} \subset (\R^d)^{\Z_-}$, where  ${\bf y}_0$ is the solution (which exists and is unique under suitable hypotheses) of the state-space system 
\begin{equation} \label{eq:RCIntro}
\left\{
\begin{aligned}
\mathbf{x}_t & = F(\mathbf{x}_{t-1},{\bf z}_t), \\ 
{\bf y}_t & = h(\mathbf{x}_t), \quad t \in \Z_-,
\end{aligned}
\right.
\end{equation}
where the {\it state} or {\it reservoir} map $F$ is (for the most part) randomly generated and only the static {\it observation} or {\it readout} map $h$ is trained in specific learning tasks. An important particular case of \eqref{eq:RCIntro} are {\it echo state networks} (ESNs) \cite{Matthews:thesis}, \cite{Matthews1993}, \cite{Matthews1994}, \cite{Jaeger04}. These are recurrent neural networks that map the input ${\bf z} \in (\R^d)^{\Z_-}$ to the value $H^{{\bf A},{\bf C},\bm{\zeta}}_{\bf W}({\bf z})={\bf Y}_0\in \mathbb{R} ^m$ determined by
\begin{equation} \label{eq:ESNIntro}
\left\{
\begin{aligned}
 \mathbf{X}_t &  = \bm{\sigma}( {\bf A} \mathbf{X}_{t-1} + {\bf C} {\bf z}_t + \bm{\zeta}), \quad t \in \Z_-,\\
{\bf Y}_t & = {\bf W} \mathbf{X}_t, \quad t \in \Z_-.
\end{aligned}
\right.
\end{equation} 
Here ${\bf A}, {\bf C}, \bm{\zeta}$ are randomly drawn (from a distribution that does not use any knowledge about $H^*$), $\bm{\sigma}$ is a given activation function as above, and ${\bf W}$ is optimized at the time of training in order to approximate $H^*$ as well as possible. This technique has been successful in a wide range of applications (see, for example, \cite{Jaeger04}, \cite{pathak:chaos}, \cite{Pathak:PRL}, \cite{Ott2018}). Based on these empirical results, ESNs with randomly generated ${\bf A}, {\bf C}, \bm{\zeta}$ are thought to be capable of approximating arbitrary dynamical and input/output systems. However, a rigorous mathematical result proving this statement does not exist yet in the literature. It is only in the context of invertible and differentiable dynamical systems on a compact manifold that a result of this type has been recently established. Indeed, the results in \cite{hart:ESNs} show that randomly drawn ESNs like \eqref{eq:ESNIntro} can be trained by optimizing ${\bf W} $ using generic one-dimensional observations of a given invertible and differentiable dynamical system to produce dynamics that are topologically conjugate to that given system.

In this article we place ourselves in the more general setup of input/output systems and provide a first mathematical result that proves the approximation capabilities of ESNs in a discrete-time setting and quantifies them by providing approximation bounds in terms of their architecture parameters.
In more detail, we propose a constructive sampling procedure for ${\bf A}, {\bf C}, \bm{\zeta}$ (depending only on three hyperparameters) so that by training ${\bf W}$, the associated system \eqref{eq:ESNIntro} can be used to approximate any $H^*$ satisfying mild regularity assumptions. The $L^2$-error between $H^*$ and its echo state approximation $H^{{\bf A},{\bf C},\bm{\zeta}}_{\bf W}$ can be bounded explicitly and the approximation result can also be extended to a universality result for general $H^*$ (not satisfying the regularity conditions). For full details we refer to Theorem~\ref{thm:ESNErrorBound} and Corollary~\ref{cor:universality} below.

We complement these results by analyzing a popular modification of \eqref{eq:ESNIntro}, in which the hidden state $\mathbf{X}$ is updated according to $\mathbf{X}_t   = \bm{\sigma}( {\bf A} {\bf W} \mathbf{X}_{t-1} + {\bf C} {\bf z}_t + \bm{\zeta})$. These systems are called {\it  echo state networks with output feedback} (or {\it Jordan recurrent neural networks} with random internal weights) and are also widely used in the literature even though, in this case, a more sophisticated training algorithm is needed (for instance a stochastic gradient-type optimization algorithm combined with backpropagation in time). By applying similar tools as in the case of \eqref{eq:ESNIntro} we provide an approximation result for such systems in situations when the unknown functional is itself given by a sufficiently regular reservoir system of type \eqref{eq:RCIntro}. In this case, only one hyperparameter $N$ appears (proportional to the number of neurons, i.e. the dimension of $\bf X$) and the approximation error is of order  $O(1/\sqrt{N})$. We refer to Theorem~\ref{thm:approx} below for full details.

To prove these results we rely mainly on probabilistic arguments involving concentration inequalities, an importance sampling procedure and techniques from empirical process theory (in particular the Ledoux-Talagrand inequality~\cite{Ledoux2013}). A further crucial ingredient is an integral representation for sufficiently regular functions related to the integral representations appearing in the proofs in \cite{Barron1993}, \cite{Maiorov2000}, \cite{Klusowski2018}. In continuous time, an alternative approach based on randomized signature is presented in \cite{JLpaper} and \cite{RC13}.

We emphasize that the proof of these dynamic statements crucially relies on our novel results for the static case. To understand these better, we briefly elaborate on the literature (we refer to the introduction of \cite{Rahimi2009} for a detailed overview). The seminal work by Barron~\cite{Barron1993} shows that any function $H^* \colon \R^q \to \R$ of a certain regularity can be approximated up to an error of order $O(1/\sqrt{N})$ using a neural network with one hidden layer and $N$ hidden nodes. The hidden weights can be generated randomly, but the distribution from which they need to be drawn depends on $H^*$. Thus, the randomly drawn weights are only used to guarantee the existence of tunable weights. Subsequently, the important contributions by Rahimi and Recht \cite{Rahimi2007}, \cite{Rahimi2008}, \cite{Rahimi2009} analyze random weights generated from a known probability distribution $p$. In their argument the optimal output layer weights (which are tuned) implement an importance sampling procedure.
The function class $\mathfrak{F}_p$ for which error bounds can be derived (see Theorems~3.1 and 3.2 in \cite{Rahimi2009}) and for which an approximation error of order $O(1/\sqrt{N})$ is guaranteed is defined in terms of $p$ and it is shown that $\mathfrak{F}_p$ is dense. However, for a given function $H^*$ it may be challenging to decide whether $H^* \in \mathfrak{F}_p$ (and hence the error bound applies) or not. 
In this paper we show that under mild regularity assumptions on $H^*$ one automatically has $H^* \in \mathfrak{F}_p$ for a wide class of distributions $p$ including the most commonly used case when $p$ is a uniform distribution. This is formulated abstractly in Theorem~\ref{thm:approxstatic} and then specialized to the uniform distribution in Proposition~\ref{prop:BarronFininteDim} and Corollary~\ref{cor:RandomNNApprox}. We also make the dependence of the resulting bounds on the input dimension explicit. This can be used to decide whether approximations by (shallow) random neural networks for classes of functions (parametrized by the input dimension) suffer from the curse of dimensionality or not. We emphasize that although all these results use shallow neural networks which are, from an approximation theory perspective, less flexible than deep neural networks (see, for instance, \cite{Maiorov2000}, \cite{Poggio2017}), here the hidden weights are generated randomly and so the neural network training does not require gradient descent-type optimization techniques.

Finally, let us point out that Theorem~\ref{thm:ESNErrorBound} entails a constructive sampling scheme for the weights that may be readily used by  practitioners and provides a learning procedure in which only  ${\bf W}$ and three hyperparameters need to be optimized. 

The remainder of this paper is organized as follows. In Section~\ref{sec:Preliminaries} we introduce some key concepts on reservoir systems. Section~\ref{sec:integralRepresentation} then proves an integral representation for sufficiently regular functions, which is at the core of the subsequent approximation results. In Section~\ref{sec:Static} we then treat the static case and prove the random neural network approximation results Theorem~\ref{thm:approxstatic}, Proposition~\ref{prop:BarronFininteDim} and Corollary~\ref{cor:RandomNNApprox}. Section~\ref{sec:ESN} is concerned with the dynamic case and contains the echo state network approximation results  Theorem~\ref{thm:ESNErrorBound} and Corollary~\ref{cor:universality}. Finally, in Section~\ref{sec:ESNOutput} we prove the approximation result for echo state networks with output feedback, Theorem~\ref{thm:approx}.

\subsection*{Notation}
{We use the notations $\Z_{-} = \{0,-1,-2,\ldots\}$, $\N = \{0,1,2,\ldots\}$, $\N^+ = \N \setminus \{0\}$.}
Throughout the article $d,m,N,q \in \N^+$ denote positive integers and $M$ is a positive constant. For any $R>0$ we denote by $B_R$ the Euclidean ball of radius $R$ around $0$ in the appropriate dimension (which will always be either mentioned explicitly or obvious from the context). Furthermore, $\lambda_q(B_R)$ denotes the volume of the ball $B_R \subset \R^q$. {Unless mentioned otherwise, $\|\cdot\|$ denotes the Euclidean norm. We denote by $\mathbb{M}_{m,n}$ the set of real $m \times n$ matrices. We fix a probability space $(\Omega,\Fc,\P)$ on which all random elements are defined.}

\section{Preliminaries on the dynamic setting} \label{sec:Preliminaries}
The goal of this section is to present some preliminaries on the dynamic case,  that is, when $\mathcal{X} \subset (\R^d)^{\Z_-}$. In this case, it is customary to refer to maps $H^* \colon \mathcal{X} \to \R^m$ as \textit{functionals}. While this article is mainly concerned with approximating functionals, let us point out that these are in one-to-one correspondence with so-called causal and time-invariant filters, see for instance \cite{RC8,RC7,RC6}. An important class of functionals is given by those satisfying $H^* = \arg \inf_{H \in \Hc} \mathcal{R}(H)$ for some class $\Hc$ of functionals and a risk map $\mathcal{R} \colon \Hc \to [0,\infty)$ that satisfies certain customary properties (see \cite{RC10} and references therein for details). Another important class is given by reservoir functionals that we recall in the next paragraphs.

\subsection{Reservoir systems and associated functionals}
Let $d, N \in \mathbb{N}^+$, $D_d \subset \R^d$, $D_N \subset \R^N$ and $F \colon D_N \times D_d\longrightarrow  D_N$, and for ${\bf z} \in (D_d)^{\Z_-}$ consider the system
\begin{equation} 
\label{eq:RCSystemDet}
\mathbf{x}_t =F(\mathbf{x}_{t-1}, {\bf z}_t), \quad t \in \Z_-.
\end{equation}
We say that \eqref{eq:RCSystemDet} satisfies the {\it echo state property}, if for any  ${\bf z} \in (D_d)^{\Z_-}$ there exists a unique $\mathbf{x} \in (D_N)^{\Z_-}$ such that \eqref{eq:RCSystemDet} holds. As the following Proposition shows, a sufficient condition guaranteeing this property is that $D_N$ is a closed ball and $F$ is contractive in the first argument. 

\begin{proposition}[Proposition~1 in \cite{RC10}] \label{prop:contractionEchoState} Let $R>0$, write $\overline{B_R}=\{{\bf u} \in \R^N \colon \|{\bf u}\|\leq R \}$ and suppose that $F \colon \overline{B_R} \times D_d \to \overline{B_R}$ is continuous. Assume that $F$ is a contraction in the first argument, that is, there exists $0<r<1$  such that for all ${\bf u},{\bf v} \in \overline{B_R}$, ${\bf w} \in D_d$ it holds that
\[ \|F({\bf u},{\bf w})- F({\bf v},{\bf w}) \| \leq r \|{\bf u}-{\bf v}\|. \]
Then the system \eqref{eq:RCSystemDet} has the echo state property. Furthermore, we can associate to it a unique mapping $H_F \colon (D_d)^{\Z_-} \to \R^N$ that is continuous (where $(D_d)^{\Z_-}$ is equipped with the product topology) and satisfies $H_F({\bf z}_{\cdot+t})=\mathbf{x}_t$, for all $t \in \Z_-$ (the symbol ${\bf z}_{\cdot+t} $ stands for the shifted semi-infinite sequence $(\ldots,{\bf z}_{-2+t},{\bf z}_{-1+t}, {\bf z} _t) \in (D_d)^{\Z_-}$).
\end{proposition}

The functional $H_F$ in Proposition~\ref{prop:contractionEchoState} will be referred to as the \textit{reservoir functional} associated to $F$. In many situations one is also interested in considering the input/output system generated by \eqref{eq:RCSystemDet} together with a readout or observation map, that is, 
\begin{equation}\label{eq:RCSystemDetYeqn}
{\bf y}_t = h(\mathbf{x}_t), \quad t \in \Z_-,
\end{equation}
for some $h \colon D_N \to \R^m$. The reservoir functional associated to \eqref{eq:RCSystemDet}-\eqref{eq:RCSystemDetYeqn} is given as $h \circ H_F$. 

In the dynamic case the functionals $H$ that we use in this article to approximate a given (unknown) functional $H^*$ are always of the form $H = h \circ H_F$ for $h$ linear and $F$ suitably constructed.

\section{Integral representations of sufficiently regular functions}\label{sec:integralRepresentation}

A key ingredient in the proofs of the approximation results in this paper {are certain integral representations of sufficiently regular functions. We provide a first result in Proposition~\ref{prop:SmoothRepresentation} below. Variations of this result under weaker conditions will be developed later on in the article. In probabilistic terms, Proposition~\ref{prop:SmoothRepresentation}} shows that for all $R>0$, any sufficiently regular function $f$ can be represented on $B_M$ as the difference of two functions of type ${\bf v} \mapsto c\E[\max({\bf v}\cdot {\bf U} + \zeta,0)]$ for some constant $c>0$, some random variables ${\bf U}$ and $\zeta$ admitting a Lebesgue-density with certain integrability properties and satisfying  $\|{\bf U}\|\leq R$ and $|\zeta| \leq \max(M R,1)$, $\P$-a.s.

The integral representation below is related to the Radon-wavelet integral representation as used in \cite{Maiorov2000} and representations appearing in \cite{Barron1993,Klusowski2018} {and \cite[Theorem~2]{Barron1992}}.

{This integral representation will be crucial to obtain random neural network approximation results with weights sampled from a uniform distribution, see Proposition~\ref{prop:BarronFininteDim} below. We will also formulate similar results for more general sampling distributions and under weaker integrability conditions (see Theorem~\ref{thm:approxstatic} and Corollary~\ref{cor:RandomNNApprox} below).}

\begin{proposition}
\label{prop:SmoothRepresentation} 
Let $\sigma \colon \R \to \R$ be given as $\sigma(x)=\max(x,0)$.
Suppose that $f \colon \R^{q} \to \R$ satisfies for all ${\bf v} \in \R^q$ with $\|{\bf v}\| \leq M$ that
\[ f({\bf v}) = \int_{\R^q} e^{i {\bf v}\cdot {\bf w}} g({\bf w}) \d {\bf w}  \]
for some $g \colon \R^q \to \C$ satisfying 
\begin{equation}\label{eq:smoothnessAss} v^* = \int_{\R^q} \max(1,\|{{\bf w}}\|^{2q+6}) |g({\bf w})|^2 \d {{\bf w}} < \infty. \end{equation}
Then, for any $R>0$ there exists a measurable function $\pi \colon \R^{q+1} \to \R$ such that 
\begin{itemize}
\item[(i)] $\pi(\bm{\omega}) = 0$ for all $\bm{\omega} = ({\bf w},u) \in \R^{q} \times \R$  satisfying $\|{\bf w}\|>R$ or $|u|>\max(M R,1)$, 

\item[(ii)] \[ \int_{\R^{q+1}} \max(1,\|\bm{\omega}\|) |\pi(\bm{\omega})| \d \bm{\omega} < \infty,
\]
\item[(iii)] for all ${\bf v} \in \R^q$ with $\|{\bf v}\| \leq M$, 
\begin{equation} \label{eq:Brepr2} f({\bf v}) = \int_{\R^{q+1}} \pi(\bm{\omega}) \sigma(({\bf v},1) \cdot \bm{\omega}) \d \bm{\omega},
 \end{equation}
\item[(iv)] \[\begin{aligned}  \int_{\R^{q+1}}   \|\bm{\omega}\|^2 \pi(\bm{\omega})^2 \d \bm{\omega}  & \leq 8 (M^3+M+2)
\left( \int_{B_R} \max(1,\|{\bf w}\|^3) |g({\bf w})|^2 \d {\bf w} \right. \\ & \qquad \left. + \int_{\R^{q} \setminus B_R} \frac{\max(1,\|{\bf w}\|^{2q+5})}{R^{2q+2}} |g({\bf w})|^2  \d {\bf w}\right)
\end{aligned}\]

and thus in particular if $R \geq 1$ then \[ \int_{\R^{q+1}} \|\bm{\omega}\|^2 \pi(\bm{\omega})^2 \d \bm{\omega} \leq 8 (M^3+M+2) v^*.  \]
\end{itemize}
\end{proposition}

\begin{remark} A sufficient condition for \eqref{eq:smoothnessAss} to be satisfied is that $f \in L^1(\R^q)$ has an integrable Fourier transform and belongs to the Sobolev space $W^{q+3,2}(\R^{q})$, see for instance \cite[Theorem~6.1]{Folland1995} {or Corollary~\ref{cor:sobolev} below.}
\end{remark}

\begin{remark}
{We now emphasize two points concerning the condition~\eqref{eq:smoothnessAss}. First, this condition~\eqref{eq:smoothnessAss} is stronger than the condition $$\int_{\R^q} \|{{\bf w}}\| |g({\bf w})| \d {{\bf w}} < \infty$$ appearing in the well-known work by Barron \cite{Barron1993} (see e.g. \eqref{eq:auxEq66} below for an argument). However, this stronger condition~\eqref{eq:smoothnessAss}  also allows us to obtain a stronger conclusion. More specifically, whereas \cite{Barron1993} proves that there exist neural network weights ensuring a certain approximation accuracy,  we will see how Proposition~\ref{prop:BarronFininteDim} below provides under condition~\eqref{eq:smoothnessAss} a \emph{constructive procedure for the neural network weights}. }

{Second, we now discuss why the condition~\eqref{eq:smoothnessAss} is necessary. The properties of $\pi$ derived in Proposition~\ref{prop:SmoothRepresentation} are required to guarantee that the neural network weights can be sampled from a uniform distribution in Proposition~\ref{prop:BarronFininteDim} below. This is ensured, on the one hand, by the compact support of $\pi$ (see Proposition~\ref{prop:SmoothRepresentation}(i)), which is achieved by a change of variables in the proof of Proposition~\ref{prop:SmoothRepresentation}. On the other hand, to carry out the importance sampling procedure in the proof of Proposition~\ref{prop:BarronFininteDim}, the square integrability condition on $\pi$ (see Proposition~\ref{prop:SmoothRepresentation}(i)) is needed. To obtain square integrability of $\|{\bf w}\|\pi({\bf w})$ we need condition~\eqref{eq:smoothnessAss} (see \eqref{eq:auxEq63}, \eqref{eq:auxEq65} in the proof below) since the Jacobian determinant appearing in the change of variables mentioned above  makes the term $\|{\bf w}\|^{2q+2}$ appear.}	
\end{remark}

\begin{proof} The proof consists of two steps. In a first step, we use a modification of the argument in \cite{Klusowski2018} to obtain a representation of type \eqref{eq:Brepr2}, but with corresponding $\pi$ not necessarily satisfying (i). Then a suitable change of variables allows to obtain a representation with the desired properties (i)-(iv). 

Beforehand, let us verify that \eqref{eq:smoothnessAss} implies that 
\begin{equation}\label{eq:auxEq66}
\int_{\R^q} |g({\bf w})| \d {\bf w} < \infty \quad  \text{and} \quad \int_{\R^q} \|{\bf w}\|^3 |g({\bf w})| \d {\bf w} < \infty.
\end{equation} Indeed, by first splitting the integral into an integral over ${B_1} \subset \R^q$ and $\R^q \setminus B_1$ and then applying H\"older's inequality one obtains
\[\begin{aligned} \int_{\R^q} (1+\|{\bf w}\|^3)|g({\bf w})| \d {\bf w} & \leq 2 \left(\int_{B_1} |g({\bf w})|^2 \d {\bf w}\right)^{1/2} \lambda_q(B_1)^{1/2} \\ & \qquad + 2 \int_{\R^q \setminus B_1} \|{\bf w}\|^3 |g({\bf w})| \d {\bf w}, 
 \end{aligned} \]
 where the last term can be estimated by applying H\"older's inequality once more to obtain 
\[ \int_{\R^q \setminus B_1} \|{\bf w}\|^3 |g({\bf w})| \d {\bf w} \leq \left(\int_{\R^q \setminus B_1} \|{\bf w}\|^{6+2q} |g({\bf w})|^2 \d {\bf w}\right)^{1/2}  \left(\int_{\R^q \setminus B_1} \|{\bf w}\|^{-2q} \d {\bf w} \right)^{1/2}
\]
and the integrals are finite thanks to the hypothesis \eqref{eq:smoothnessAss}. 

\medskip

\noindent \textit{Step 1:}
Firstly, note that for any $z \in \R$ one may write
\begin{equation}\label{eq:auxEq60} -\int_0^\infty (z-u)^+ e^{i u} + (-z-u)^+ e^{-iu} \d u = e^{i z} - i z - 1, \end{equation}
since for $z > 0$ one has 
\[ \int_0^z (z-u) e^{i u }\d u = - \frac{1}{i} z + \frac{1}{i} \int_0^z e^{i u} \d u = i z - e^{i z} + 1 \]
and for $z < 0$ one calculates
\[\int_0^{-z} (-z-u) e^{-iu} \d u = -\frac{1}{i} z -  \frac{1}{i} \int_0^{-z} e^{-i u} \d u = i z - e^{i z } + 1. \]
Secondly, for any ${\bf v} \in \R^q$ one obtains by Tonelli's theorem and \eqref{eq:auxEq66} that
\[\begin{aligned} \int_{\R^q \times [0,\infty)} & |({\bf v}\cdot{\bf w}-u)^+ e^{i u} + (-{\bf v}\cdot{\bf w}-u)^+ e^{-iu}||g({\bf w})| \d {\bf w} \d u 
\\ & \leq \int_{\R^q} \int_0^{|{\bf v}\cdot{\bf w}|} (|{\bf v}\cdot{\bf w}|-u) |g({\bf w})| \d u \d {\bf w} \\
& \leq \frac{\|{\bf v}\|^2}{2} \int_{\R^q} \|{\bf w}\|^2 |g({\bf w})| \d {\bf w} < \infty.
\end{aligned} \]
Hence one may combine Fubini's theorem, \eqref{eq:auxEq66} and \eqref{eq:auxEq60} to obtain for any ${\bf v} \in \R^q$
\[\begin{aligned} - \int_{\R^q \times [0,\infty)} &  [({\bf v}\cdot{\bf w}-u)^+ e^{i u} + (-{\bf v}\cdot{\bf w}-u)^+ e^{-iu}]g({\bf w}) \d {\bf w} \d u \\ &  = \int_{\R^q} (e^{i {\bf v}\cdot{\bf w}} - i {\bf v}\cdot{\bf w} - 1) g({\bf w}) \d {\bf w} = f({\bf v}) - (\nabla f)(0) \cdot {\bf v} - f(0). \end{aligned} \]
Based on this integral representation of $f$ we will now define $\alpha$ appropriately to obtain 
 \[
f({\bf v}) =    \int_{\R^{q+1}}   \sigma(({\bf v},1)\cdot({\bf w},u)) \alpha({\bf w},u) \d {\bf w} \d u 
 \]
 for all ${\bf v} \in \R^q$ with $\|{\bf v}\| \leq M$. To do this, first note that for all ${\bf v} \in \R^q$ with $\|{\bf v}\| \leq M$ and all $({\bf w},u)\in \R^{q+1}$ with $u \leq - M \|{\bf w}\|$ we have ${\bf v} \cdot {\bf w} +u \leq 0$ and therefore $\sigma(({\bf v},1)\cdot({\bf w},u))=0$. Setting
\[ \alpha_1({\bf w},u) = - [\mathrm{Re}(e^{-i u} g({\bf w})) +  \mathrm{Re}(e^{i u} g(-{\bf w}))] \mathbbm{1}_{(-M \|{\bf w}\|,0]}(u)  \] and changing variables we thus obtain 
\begin{equation} \label{eq:auxEq61} \begin{aligned} f({\bf v}) &  - (\nabla f)(0) \cdot {\bf v} - f(0) =  \int_{\R^{q+1}}   \sigma(({\bf v},1)\cdot({\bf w},u)) \alpha_1({\bf w},u) \d {\bf w} \d u.
\end{aligned} \end{equation}
In addition $f(0), (\nabla f)(0) \in \R$ and therefore one has that $\int_{\R^q} \mathrm{Im}[g({\bf w})]\d {\bf w} = 0$ and $\int_{\R^q}  ({\bf v}\cdot{\bf w})\mathrm{Re}[g({\bf w})]\d {\bf w} = 0$ . This yields
\begin{equation}\label{eq:auxEq62}\begin{aligned} (\nabla f)(0) \cdot {\bf v} & + f(0) \\ &  = \int_{\R^q} {\bf v}\cdot{\bf w} (-\mathrm{Im}[g({\bf w})]) + \mathrm{Re}[g({\bf w})]\d {\bf w}  \\
&  = \int_{\R^q} \int_0^1 ({\bf v}\cdot{\bf w}+u) (\mathrm{Re}[g({\bf w})]-\mathrm{Im}[g({\bf w})]) \d u \d {\bf w}  
\\ & = \int_{\R^q} \int_0^1 [({\bf v}\cdot{\bf w}+u)^+ - (-{\bf v}\cdot{\bf w}-u)^+] (\mathrm{Re}[g({\bf w})]-\mathrm{Im}[g({\bf w})]) \d u \d {\bf w}   . \end{aligned} \end{equation}
Defining $\tilde{g}({\bf w}) = \mathrm{Re}[g({\bf w})]-\mathrm{Im}[g({\bf w})]$ and
\[ \alpha_2({\bf w},u) = \mathbbm{1}_{[0,1]}(u)\tilde{g}({\bf w}) - \mathbbm{1}_{[-1,0]}(u) \tilde{g}(-{\bf w}) \] 
we may rewrite \eqref{eq:auxEq62} as
\[ \begin{aligned}
(\nabla f)(0) \cdot {\bf v} + f(0) & =  \int_{\R^{q+1}}   \sigma(({\bf v},1)\cdot({\bf w},u)) \alpha_2({\bf w},u) \d {\bf w} \d u.
\end{aligned}\]
 Combining this with \eqref{eq:auxEq61} and setting $\alpha = \alpha_1 + \alpha_2$ thus yields 
 \[
f({\bf v}) =    \int_{\R^{q+1}}   \sigma(({\bf v},1)\cdot({\bf w},u)) \alpha({\bf w},u) \d {\bf w} \d u .
 \]

\medskip

\noindent \textit{Step 2:} For $\bm{\omega} = ({\bf w},u) \in \R^{q} \times \R$ define
\[ \pi({\bf w},u) = \mathbbm{1}_{B_R \setminus \{0\}}({\bf w})  \left[\alpha(\bm{\omega}) + \frac{R^{2(q+2)}}{\|{\bf w}\|^{2(q+2)}} \alpha \left(\frac{R^2 \bm{\omega}}{\|{\bf w}\|^2}\right)\right].
\] 
Then clearly $ \pi({\bf w},u) = 0$ if $\|{\bf w}\|>R$. If $|u|>\max(M R,1)$ and $\|{\bf w}\|\leq R$ then it follows that $|u|>M \|{\bf w}\|$ and $|u|R^2/\|{\bf w}\|^2>1$ and hence $\alpha_1({\bf w},u) = \alpha_2({\bf w},u) =\alpha_1(R^2{\bf w}/\|{\bf w}\|^2,R^2 u/\|{\bf w}\|^2) = \alpha_2(R^2{\bf w}/\|{\bf w}\|^2,R^2 u/\|{\bf w}\|^2)=0$. This shows (i).  Next, define the mapping
\[ \varphi \colon B_R \setminus \{0\} \to \R^{q} \setminus \overline{B_R}, \quad \varphi({\bf w}) = \frac{R^2{\bf w}}{\|{\bf w}\|^2} \]
and note that $\varphi$ is a diffeomorphism satisfying
\[|\det(\varphi'({\bf w}))| = R^{2q} \left|\det\left(\mathbbm{1}_{q \times q} \frac{1}{\|{\bf w}\|^2} - 2 \frac{{\bf w} {\bf w}^\tr}{\|{\bf w}\|^4}\right)\right| = \frac{R^{2q}}{\|{\bf w}\|^{2q}}.
\] 
The change of variables formula hence implies for any measurable function $h \colon \R^{q} \to \R$ that
\[\begin{aligned}
\int_{\R^{q} \setminus B_R} h({\bf w}) \d {\bf w} = \int_{B_R} h(\varphi({\bf w})) \frac{R^{2q}\d {\bf w}} {\|{\bf w}\|^{2q}}.
\end{aligned}
\]
Applying this and the substitution $R^2 \tilde{u} = u \|{\bf w}\|^2$ one obtains that
\[\begin{aligned} & \int_{\R^{q+1}}  \max(1,\|\bm{\omega}\|) |\pi(\bm{\omega})| \d \bm{\omega} \\ &  \quad \leq \int_{B_R} \int_\R \frac{R^{2} \max(1,\|({\bf w},\tilde{u})\|)}{\|{\bf w}\|^{2}} \left| \alpha \left(R^2\frac{({\bf w},\tilde{u})}{\|{\bf w}\|^2}\right)\right| \frac{R^{2(q+1)} \d \tilde{u} \d {\bf w}}{\|{\bf w}\|^{2(q+1)}} \\ & \qquad + \int_{B_R\times\R}  (1+\|\bm{\omega}\|^2) |\alpha(\bm{\omega})| \d \bm{\omega} 
\\ & \quad= \int_{\R^{q} \setminus B_R} \int_\R  \max(\frac{\|{\bf w}\|^2}{R^2},\|({\bf w},u)\|) |\alpha ({\bf w},u)|  \d u \d {\bf w}  + \int_{B_R\times\R}  (1+\|\bm{\omega}\|^2) |\alpha(\bm{\omega})| \d \bm{\omega}
\\ & \quad \leq 12 \max(1,R^{-2}) \int_{\R^{q}} \int_0^{\max(1,M\|{\bf w}\|)}  (1+\|{\bf w}\|^2+u^2) |g({\bf w})|  \d u \d {\bf w}
\\ & \quad \leq 12(M^3+M+1)\max(1,R^{-2}) \int_{\R^{q}} (1+\|{\bf w}\|^3) |g({\bf w})|  \d u \d {\bf w} < \infty.
\end{aligned}\]

This shows (ii). To deduce the representation (iii) one may now use Step 1 and apply the same substitution as above to the first term to obtain for any ${\bf v} \in \R^q$ with $\|{\bf v}\| \leq M$ that
\[\begin{aligned}
& f({\bf v}) \\  & \quad =    \int_{\R^{q} \setminus B_R} \int_\R  \sigma(({\bf v},1)\cdot({\bf w},u)) \alpha({\bf w},u) \d u \d {\bf w}  + \int_{B_R \times \R} \sigma(({\bf v},1)\cdot({\bf w},u)) \alpha({\bf w},u) \d {\bf w} \d u
\\ & \quad = \int_{B_R} \int_\R   \sigma(({\bf v},1)\cdot (\varphi({\bf w}),\frac{R^2\tilde{u}}{\|{\bf w}\|^2})) \alpha(\varphi({\bf w}),\frac{R^2\tilde{u}}{\|{\bf w}\|^2}) \frac{R^{2(q+1)}\d \tilde{u} \d {\bf w}}{\|{\bf w}\|^{2(q+1)}} \\ & \qquad + \int_{B_R\times\R} \sigma(({\bf v},1)\cdot \bm{\omega}) \alpha(\bm{\omega}) \d \bm{\omega}
\\ & \quad = \int_{B_R\times \R} \sigma(({\bf v},1)\cdot \bm{\omega}) \pi(\bm{\omega}) \d \bm{\omega}.
\end{aligned} \]
It remains to prove (iv). Applying again the change of variables formula and using $\|\varphi({\bf w})\| = R^2 \|{\bf w}\|^{-1}$ yields
\begin{equation} \label{eq:auxEq63} \begin{aligned} \int_{\R^{q+1}} & \|\bm{\omega}\|^2 \pi(\bm{\omega})^2 \d \bm{\omega} \\ & \leq 2 \int_{B_R \times \R} \|\bm{\omega}\|^2 \alpha(\bm{\omega})^2 \d \bm{\omega}  \\ & \qquad + 2 \int_{B_R} \int_\R  \frac{R^{2q+2}}{\|{\bf w}\|^{2q+2}} \left\|\frac{R^2(\tilde{u},{\bf w})}{\|{\bf w}\|^2}\right\|^2 \alpha \left(\frac{R^2(\tilde{u},{\bf w})}{\|{\bf w}\|^2}\right)^2 \frac{R^{2(q+1)}\d \tilde{u} \d {\bf w}}{\|{\bf w}\|^{2(q+1)}}
\\ & = 2 \int_{B_R \times \R} \|\bm{\omega}\|^2 \alpha(\bm{\omega})^2 \d \bm{\omega} \\ & \qquad + 2 R^{-(2q+2)} \int_{\R^{q} \setminus B_R} \int_\R  \left[u^2 \|{\bf w}\|^{2q+2}+\|{\bf w}\|^{2q+4}\right] \alpha(u,{\bf w})^2 \d u \d {\bf w}. \end{aligned}\end{equation}
To estimate the first term, we note that $ |\alpha(\bm{\omega})|^2 \leq 2|\alpha_1(\bm{\omega})|^2+2|\alpha_2(\bm{\omega})|^2$ and thus

\begin{equation} \label{eq:auxEq64} \begin{aligned}
\int_{B_R \times \R} &  \|\bm{\omega}\|^2 \alpha(\bm{\omega})^2 \d \bm{\omega} \\ & \leq 4 \int_{B_R} \int_{\R} (u^2+\|{\bf w}\|^2) [|g({\bf w})|^2\mathbbm{1}_{[0,M \|{\bf w}\|]}(u) + |\tilde{g}({\bf w})|^2 \mathbbm{1}_{[0,1]}(u)] \d u \d {\bf w}
\\ & \leq 4(M^3+M) \int_{B_R} \|{\bf w}\|^3 |g({\bf w})|^2 \d {\bf w} + 4 \int_{B_R} (1+\|{\bf w}\|^2) |g({\bf w})|^2 \d {\bf w}.
\end{aligned} \end{equation}
Furthermore, one estimates the integral in the second term in \eqref{eq:auxEq63} as  
\begin{equation} \label{eq:auxEq65} \begin{aligned}
 & \frac{1}{4}\int_{\R^{q} \setminus B_R} \int_\R    \left[u^2 \|{\bf w}\|^{2q+2}+\|{\bf w}\|^{2q+4}\right] \alpha(u,{\bf w})^2 \d u \d {\bf w} 
 \\ & \leq  \int_{\R^{q} \setminus B_R} \int_\R  \left[u^2 \|{\bf w}\|^{2q+2}+\|{\bf w}\|^{2q+4}\right]  [|g({\bf w})|^2\mathbbm{1}_{[0,M \|{\bf w}\|]}(u) + |\tilde{g}({\bf w})|^2 \mathbbm{1}_{[0,1]}(u)] \d u \d {\bf w}
 \\& \leq (M^3+M) \int_{\R^{q} \setminus B_R} \|{\bf w}\|^{2q+5} |g({\bf w})|^2  \d {\bf w} + 4 \int_{\R^{q} \setminus B_R} (\|{\bf w}\|^{2q+2}+\|{\bf w}\|^{2q+4}) |g({\bf w})|^2 \d {\bf w}.
 \end{aligned}
\end{equation}
Combining \eqref{eq:auxEq63}, \eqref{eq:auxEq64}, and \eqref{eq:auxEq65} one obtains

\[\begin{aligned} \int_{\R^{q+1}}   \|\bm{\omega}\|^2 \pi(\bm{\omega})^2 \d \bm{\omega} & \leq 8(M^3+M+2)\left( \int_{B_R} \max(1,\|{\bf w}\|^3) |g({\bf w})|^2 \d {\bf w} \right. \\ & \qquad + \left. \int_{\R^{q} \setminus B_R} \frac{\max(1,\|{\bf w}\|^{2q+5})}{R^{2q+2}} |g({\bf w})|^2  \d {\bf w}\right), \end{aligned}\]
%\[ \int_{\R^{q+1}} \|\bm{\omega}\|^2 \pi(\bm{\omega})^2 \d \bm{\omega} \leq 8 (M^3+M+2) v^*,\]
 as claimed. \qed
\end{proof}

\section{Approximation Error Estimates for Random Neural Networks}\label{sec:Static}

In this section we derive random neural network approximation bounds for sufficiently regular functions. We first introduce the setting and prove a result for separable Hilbert spaces $\mathcal{X}$ and general sampling distributions (Theorem~\ref{thm:approxstatic} below). In Section~\ref{sec:Rqcase} we then consider the special case $\mathcal{X}=\R^q$ and derive results for weights sampled from a uniform distribution (see Proposition~\ref{prop:BarronFininteDim} and Corollary~\ref{cor:RandomNNApprox}). The dependence of the approximation bounds on the input dimension is explicit and thus these results may be used to decide when the approximation by random neural networks for classes of functions (parametrized by the input dimension) suffer from the curse of dimensionality.
Finally, in Section~\ref{sec:staticUniversal} we deduce as a corollary of the results in Section~\ref{sec:Rqcase} that neural networks with randomly generated inner weights and in which only the last layer is trained possess universal approximation capabilities. This is a new version of the $L^2$-universal approximation theorem for neural networks from \cite{hornik1991}.

\subsection{Setting and result for separable Hilbert spaces} \label{sec:hilbertcase}
Suppose $\mathcal{X}$ is a separable Hilbert space {with inner product $\langle \cdot, \cdot \rangle$ and associated norm $\|\cdot\|$.}
Let $({\bf A}_1,\zeta_1),\ldots,({\bf A}_N,\zeta_N)$ be i.i.d.\ $\mathcal{X} \times \R$-valued random variables with distribution $\pi$, a probability measure on $\mathcal{B}(\mathcal{X} \times \R) = \mathcal{B}(\mathcal{X}) \otimes \mathcal{B}(\R)$ (see \cite[Lemma~1.2]{Kallenberg2002}). Denote by ${\bf A} \colon \mathcal{X} \to \R^N$ the random linear map with ${\bf A} {\bf z} = (\langle {\bf A}_1, {\bf z} \rangle, \ldots, \langle {\bf A}_N, {\bf z} \rangle )$ and set $\bm{\zeta}=(\zeta_1,\ldots,\zeta_N)$. Then for any $\mathbb{M}_{m , N}$-valued random matrix ${\bf W}$ we may define a random function $H^{{\bf A},\bm{\zeta}}_{\bf W} \colon \mathcal{X} \to \R^m$ by
\begin{equation}
\label{random neural network}
H^{{\bf A},\bm{\zeta}}_{\bf W}({\bf z})= {\bf W}  \bm{\sigma}({\bf A} {\bf z} + \bm{\zeta}), \quad {\bf z} \in \mathcal{X}. 
\end{equation}
Such a function will be called a {\it random neural network} with $N$ hidden nodes and inputs in $\mathcal{X}$. Clearly, if $\mathcal{X} = \R^d$, then this is a classical single-hidden-layer feedforward neural network with inputs in $\R^d$. When $ \boldsymbol{\sigma}: \mathbb{R}^N \longrightarrow \mathbb{R} ^N $ is obtained as the componentwise application of the  rectifier function $\sigma \colon \R \to \R$ given by $\sigma(x):=\max(x,0)$ we say that \eqref{random neural network} is a {\it ReLU} neural network.

We will be interested in using random neural networks to approximate a (unknown) function $H^* \colon \mathcal{X} \to \R^m$. In applications, the procedure is typically as follows: in a first step the network parameters ${\bf A},\bm{\zeta}$ are generated randomly. Then these are considered as fixed and the matrix ${\bf W}$ is trained (given the realizations of ${\bf A},\bm{\zeta}$) in order to approximate $H^*$ as well as possible. With this in mind, in what follows we will be mainly interested in measuring the approximation error between $H^{{\bf A},\bm{\zeta}}_{\bf W}$ and $H^*$ \textit{conditional on} ${\bf A},\bm{\zeta}$ and with respect to the $L^2(\mathcal{X},\mu_{{\bf Z}})$-norm for a probability measure $\mu_{{\bf Z}}$ on $(\mathcal{X},\mathcal{B}(\mathcal{X}))$. Thus, throughout this section, ${\bf Z}$ is an arbitrary $\mathcal{X}$-valued random variable. We denote by $\mu_{{\bf Z}}$ its distribution. The only assumptions we impose is that $\|{\bf Z}\|\leq M$, $\P$-a.s. and that ${\bf Z}$ is independent of $({\bf A}_1,\zeta_1),\ldots,({\bf A}_N,\zeta_N)$.
The following Lemma guarantees in particular that $H_{\bf W}^{{\bf A},\bm{\zeta}}({\bf Z})$ is a random variable, that is, $\Fc$-measurable.

\begin{lemma}
$H^{{\bf A},\bm{\zeta}}_{\bf W}$ is product-measurable, that is, the mapping $(\omega,{\bf z}) \ni \Omega \times \mathcal{X} \mapsto H^{{\bf A}(\omega),\bm{\zeta}(\omega)}_{\bf W(\omega)}({\bf z}) \in \R^{m}$ is $\Fc \otimes \mathcal{B}(\mathcal{X})$-measurable.
\end{lemma}
\begin{proof}
On the one hand, the Cauchy-Schwarz inequality implies that for any ${\bf z} \in \mathcal{X}$ the mapping $\mathcal{X} \ni {\bf v} \mapsto \langle {\bf v},{\bf z} \rangle$ is continuous and thus $\mathcal{B}(\mathcal{X})$-measurable. This shows that $ \langle {\bf A}_i,{\bf z} \rangle$ is a random variable for all $i=1,\ldots,N$. Therefore, for any  ${\bf z} \in \mathcal{X}$ the mapping
\[ \Omega \ni \omega \mapsto H^{{\bf A}(\omega),\bm{\zeta}(\omega)}_{\bf W(\omega)}({\bf z}) = {\bf W}(\omega)\bm{\sigma}({\bf A}(\omega) {\bf z} + \bm{\zeta}(\omega)) \in \R^m  \] 
is $\Fc$-measurable.

On the other hand, for any $\omega \in \Omega$ the linear map ${\bf A}(\omega) \colon \mathcal{X} \to \R^N$ is continuous (again by the Cauchy-Schwarz inequality) and thus also $H^{{\bf A}(\omega),\bm{\zeta}(\omega)}_{\bf W(\omega)} \colon \mathcal{X} \to \R^m$ is continuous. The claimed product-measurability therefore follows for instance from Aliprantis \& Border~\cite[Lemma~4.51]{Aliprantis2006}.\qed
\end{proof}

We now present our random neural network approximation result, {see also Remark~\ref{rmk:leastsquares} below for a discussion}. We use the following notation: for any measure $\nu$ we write $\nu^-$ for the measure $\nu^-(\cdot)=\nu(-\cdot)$ {and for a complex measure $\nu$ we denote by $|\nu|$ its total variation measure, see \cite[Chapter~6]{Rudin:real:analysis}.}

\begin{theorem}
\label{thm:approxstatic} 
Suppose that
 $H^* \colon \mathcal{X} \to \R^m$ can be represented as
 \[ H_j^*({\bf z}) = \int_{\mathcal{X}} e^{i \langle {\bf w}, {\bf z} \rangle } \hat{\mu}_j(\d {\bf w})  \]
 for some complex measures $\hat{\mu}_j$, $j=1,\ldots,m$, on $(\mathcal{X},\mathcal{B}(\mathcal{X}))$ and all ${\bf z} \in \mathcal{X}$ with $\|{\bf z}\| \leq M$. Assume that 
 \begin{equation} \label{eq:BarronCondHilbert} \int_{\mathcal{X}} \max(1,\|{\bf w}\|^2) |\hat{\mu}_j|(\d {\bf w}) < \infty, \end{equation}
$\pi = \pi_{\mathcal{X}} \otimes (\pi_\R(x)\d x)$, $ |\hat{\mu}_j|+|\hat{\mu}_j|^- \ll  \pi_{\mathcal{X}}$ and with $F_\pi(x) = 2\int_{-x}^0 \frac{1}{\pi_\R(u)} \d u$ either (i) or (ii) holds:
\begin{itemize}
\item [(i)] $\pi_\R$ is strictly positive and $F_\pi(x) < \infty$ for all $x \in \R$
\item [(ii)] for some $R>0$, $\pi_{\mathcal{X}}(\{ {\bf w} \in \mathcal{X} \, : \, \|{\bf w}\| > R \})=0$ and $\pi_\R(x) > 0$, $F_\pi(x)<\infty$  for $|x|\leq \max(M R,1)$.
\end{itemize}
Furthermore set $g_j = \frac{\d (|\hat{\mu}_j|+|\hat{\mu}_j|^-)}{\d \pi_{\mathcal{X}}}$ and assume that 
  \begin{equation} \label{eq:VarianceHilbert} \int_{\mathcal{X}} F_{\pi}(M \|{\bf w}\|)\|{\bf w}\|^2 g_j({\bf w})^2 \pi_\mathcal{X}(\d{\bf w}) < \infty, \quad \int_{\mathcal{X}} \max(\|{\bf w}\|^2,1) g_j({\bf w})^2 \pi_\mathcal{X}(\d{\bf w}) < \infty \end{equation}
and let $\sigma \colon \R \to \R$ be the rectifier function given by $\sigma(x):=\max(x,0)$.
 Then there exists ${\bf W}$ (a  $\M_{m,N}$-valued random variable) and $C^*>0$ such that the random ReLU-neural network $H_{\bf W}^{{\bf A},\bm{\zeta}}$ satisfies
 \[\begin{aligned}  \E[\|H_{\bf W}^{{\bf A},\bm{\zeta}}({\bf Z}) - H^*({\bf Z})\|^2]   \leq \frac{C^*}{N} \end{aligned} \]
 and 
  for any $\delta \in (0,1)$, with probability $1-\delta$ the random neural network $H_{\bf W}^{{\bf A},\bm{\zeta}}$ satisfies
\[
\left(\int_{\mathcal{X}} \|H_{\bf W}^{{\bf A},\bm{\zeta}}({\bf z}) - H^*({\bf z})\|^2 \mu_{{\bf Z}}(\d {\bf z}) \right)^{1/2} \leq \frac{\sqrt{C^*}}{\delta \sqrt{N}}.
\]
Moreover, the constant $C^*$ is explicit and given by $C^*=\sum_{j=1}^m C_j^*$ with 
\begin{equation*}\begin{aligned} C_j^* & = M^2 \int_{\mathcal{X}} F_{\pi}(M \|{\bf w}\|)\|{\bf w}\|^2 g_j({\bf w})^2 \pi_\mathcal{X}(\d{\bf w}) \\ & \quad  + 8 M^2 (F_{\pi}(1)-F_{\pi}(-1)) \int_{\mathcal{X}}  \max(\|{\bf w}\|^2,1) g_j({\bf w})^2 \pi_\mathcal{X}(\d{\bf w}). \end{aligned} \end{equation*}

\end{theorem}

{
\begin{remark}\label{rmk:leastsquares} At first glance Theorem~\ref{thm:approxstatic} may appear to be merely an existence statement. However, an optimal ${\bf W}$ can in fact be computed explicitly by solving the least-squares minimization problem 
\begin{equation}\label{eq:leastSquares} \min_{{\bf W}} \E[\|{\bf W}  \bm{\sigma}({\bf A} {\bf Z} + \bm{\zeta}) - H^*({\bf Z})\|^2 | {\bf A}, \bm{\zeta}], \end{equation}
where the minimization is taken with respect to $\M_{m,N}$-valued random variables which are measurable with respect to the sigma-algebra generated by  ${\bf A}, \bm{\zeta}$. We will show in \eqref{eq:Videf} and \eqref{eq:WFNNdef} below  that the matrix ${\bf W}$ constructed in the proof of Theorem~\ref{thm:approxstatic} is measurable with respect to the sigma-algebra generated by  ${\bf A}, \bm{\zeta}$. Consequently, Theorem~\ref{thm:approxstatic} shows that 
\[ \E\left[\min_{{\bf W}} \left\lbrace\E[\|H_{\bf W}^{{\bf A},\bm{\zeta}}({\bf Z}) - H^*({\bf Z})\|^2 | {\bf A}, \bm{\zeta}]\right\rbrace\right]  \leq \frac{C^*}{N}. \]
\end{remark}
}
{
\begin{remark} A first attempt at proving Theorem~\ref{thm:approxstatic} might be to directly work with the solution to the least-squares minimization problem  \eqref{eq:leastSquares}, i.e.\ the explicit minimizer ${\bf W}^*$.  However, evaluating the approximation error
\[
\E[\|{\bf W}^*  \bm{\sigma}({\bf A} {\bf Z} + \bm{\zeta}) - H^*({\bf Z})\|^2]
\]
directly is very challenging due to the dependence between ${\bf W}^*$ and $\bm{\sigma}({\bf A} {\bf Z} + \bm{\zeta})$. This is further complicated by the fact that the explicit expression of ${\bf W}^*$ involves the inverse of the covariance matrix of $\bm{\sigma}({\bf A} {\bf Z} + \bm{\zeta})$ conditional on ${\bf A}, \bm{\zeta} $. Therefore, evaluating the expectation  with respect to ${\bf A}, \bm{\zeta} $ or providing an upper bound for it is for the time being out of reach. 
This is the reason why we do not work with \eqref{eq:leastSquares} in the proof of Theorem~\ref{thm:approxstatic}, but we rather explicitly construct a  ${\bf W}$ for which the approximation error can be bounded more easily. As pointed out in Remark~\ref{rmk:leastsquares}, we thereby obtain also an upper bound for the optimal ${\bf W}$. Whether or not one can also obtain a lower bound 
\[ \E\left[\min_{{\bf W}} \left\lbrace\E[\|H_{\bf W}^{{\bf A},\bm{\zeta}}({\bf Z}) - H^*({\bf Z})\|^2 | {\bf A}, \bm{\zeta}]\right\rbrace\right]  \geq \frac{\tilde{C}}{N} \]
for some $\tilde{C}>0$ is still not clear due to the difficulties mentioned above. 
\end{remark}
}

\begin{proof}
First note that, writing ${\bf W}_j$ for the $j$-th row of ${\bf W}$, one has
\[ \E[\|H_{\bf W}^{{\bf A},\bm{\zeta}}({\bf Z}) - H^*({\bf Z})\|^2] = \sum_{j=1}^m \E[|H_{{\bf W}_j}^{{\bf A},\bm{\zeta}}({\bf Z}) - H_j^*({\bf Z})|^2]. \]
Thus, it is sufficient to prove the claimed result for each component $j$ individually and sum up the resulting constants. Without loss of generality, we may therefore assume $m=1$. To simplify notation we will write $H^*=H_1^*$, $\hat{\mu}=\hat{\mu}_1$, $g=g_1$ and $C^* = C_1^*$.

The proof now proceeds in two steps. In a first step we derive an integral representation for $H^*$ similar to Proposition~\ref{prop:SmoothRepresentation}. In the second step we then choose ${\bf W}$ in such a way that $H_{\bf W}^{{\bf A},\bm{\zeta}}$ is a sample average of $N$ i.i.d.\ random functions with expectation $H^*$ and deduce the claimed error bound based on this.

\textit{Step 1: Integral representation.} Firstly, recall that by \cite[Theorem~6.12]{Rudin:real:analysis} there exists a measurable function $h \colon \mathcal{X} \to \C$ satisfying $|h({\bf w})|=1$ for all ${\bf w} \in \mathcal{X}$ and $\hat{\mu}(\d {\bf w}) = h({\bf w})|\hat{\mu}|(\d {\bf w}) $. 
Next note that proceeding precisely as in the proof of Step~1 in Proposition~\ref{prop:SmoothRepresentation} and using \eqref{eq:BarronCondHilbert} yields for any ${\bf v} \in \mathcal{X}$ that
\begin{equation}\label{eq:auxEq68} \begin{aligned} - \int_{\mathcal{X} \times [0,\infty)} &  [(\langle{\bf v},{\bf w}\rangle-u)^+ e^{i u} + (-\langle {\bf v},{\bf w}\rangle-u)^+ e^{-iu}] \hat{\mu}(\d {\bf w}) \d u \\ &  = \int_{\mathcal{X}} (e^{i \langle {\bf v},{\bf w}\rangle} - i \langle{\bf v},{\bf w}\rangle - 1) \hat{\mu}(\d {\bf w}) = H^*({\bf v}) - \int_{\mathcal{X}} i \langle{\bf v},{\bf w}\rangle  \hat{\mu}(\d {\bf w}) - H^*(0).  \end{aligned} \end{equation}
We claim that the last integral is a real number. To see this, one uses $\mathrm{Im}(H^*(\lambda{ \bf v}))=0$ and $\mathrm{Im}(H^*(0))=0$ to estimate for any $\lambda > 0$ 

\[ \begin{aligned}
\left|\mathrm{Im}\left(\int_{\mathcal{X}} i \langle{\bf v},{\bf w}\rangle  \hat{\mu}(\d {\bf w})\right)\right|& = \frac{1}{\lambda} \left|\mathrm{Im}\left(H^*(\lambda{ \bf v}) - H^*(0) - \int_{\mathcal{X}} i \langle \lambda {\bf v},{\bf w}\rangle  \hat{\mu}(\d {\bf w})\right)\right|
\\ & \leq \frac{1}{\lambda} \left|\int_{\mathcal{X}} (e^{i \langle \lambda {\bf v}, {\bf w}  \rangle } - 1 - i \langle \lambda {\bf v},{\bf w}\rangle) h({\bf w})  |\hat{\mu}|(\d {\bf w})\right|
\\ & \leq \frac{1}{2 \lambda} \int_{\mathcal{X}} |\langle \lambda {\bf v},{\bf w}\rangle|^2  |\hat{\mu}|(\d {\bf w})
\\ & \leq \lambda \frac{\|{\bf v}\|^2}{2} \int_{\mathcal{X}} \|{\bf w}\|^2  |\hat{\mu}|(\d {\bf w})
\end{aligned}
\]
and note that the last expression converges to $0$ as $\lambda \to 0$ due to   \eqref{eq:BarronCondHilbert}. This shows that 

\begin{equation}\label{eq:auxEq78}\begin{aligned} \int_{\mathcal{X}} & i \langle{\bf v},{\bf w}\rangle  \hat{\mu}(\d {\bf w})  + H^*(0) \\ &  =\int_{\mathcal{X}} \left( \langle{\bf v},{\bf w}\rangle (-\mathrm{Im}[h({\bf w})]) + \mathrm{Re}[h({\bf w})] \right) |\hat{\mu}|(\d {\bf w})  \\
&  = \int_{\mathcal{X}} \int_0^1 (\langle{\bf v},{\bf w}\rangle+u) (\mathrm{Re}[h({\bf w})]-\mathrm{Im}[h({\bf w})]) \d u  |\hat{\mu}|(\d {\bf w}) 
\\ & = \int_{\mathcal{X}} \int_0^1 [(\langle{\bf v},{\bf w}\rangle+u)^+ - (-\langle{\bf v},{\bf w}\rangle-u)^+] (\mathrm{Re}[h({\bf w})]-\mathrm{Im}[h({\bf w})]) \d u |\hat{\mu}|(\d {\bf w}), \end{aligned} \end{equation}
which is the analogue to \eqref{eq:auxEq62} in the proof of Proposition~\ref{prop:SmoothRepresentation}.

We now combine the representations \eqref{eq:auxEq68} and \eqref{eq:auxEq78} to arrive at the claimed integral representation. To this end define the function $\bar{h} \colon \mathcal{X} \to \R$ by $\bar{h}({\bf w}) = \mathrm{Re}[h({\bf w})]-\mathrm{Im}[h({\bf w})]$ for ${\bf w} \in \mathcal{X}$ and define the measures $\tilde{\mu}_1(\d {\bf w}, \d u)  = \mathrm{Re}[e^{-i u} h({\bf w})] |\hat{\mu}|(\d {\bf w}) \d u$, $\tilde{\mu}_2(\d {\bf w}, \d u)  = \mathrm{Re}[e^{i u} h(-{\bf w})] |\hat{\mu}|^-(\d {\bf w}) \d u$ on $\mathcal{X} \times \R$. With these notations we may define the measures $\alpha_1$ and $\alpha_2$ on $\mathcal{X} \times \R$ by
\[\begin{aligned} \alpha_1(\d {\bf w}, \d u) & = -  \mathbbm{1}_{(-M \|{\bf w}\|,0]}(u) [\tilde{\mu}_1(\d {\bf w}, \d u)+\tilde{\mu}_2(\d {\bf w}, \d u)]  \\ \alpha_2(\d {\bf w}, \d u) & = \mathbbm{1}_{[0,1]}(u)\bar{h}({\bf w})|\hat{\mu}|(\d {\bf w}) \d u - \mathbbm{1}_{[-1,0]}(u)\bar{h}(-{\bf w})|\hat{\mu}|^-(\d {\bf w}) \d u .\end{aligned}\]
As shown above, the right hand side in \eqref{eq:auxEq68} is real and hence so is the left hand side. Thus, by setting $\alpha = \alpha_1 + \alpha_2$, rearranging \eqref{eq:auxEq68} and using \eqref{eq:auxEq78} one obtains for any ${\bf v} \in \mathcal{X}$ with $\|{\bf v}\| \leq M$ that
\begin{equation}\label{eq:auxEq71} H^*({\bf v}) = \int_{\mathcal{X} \times \R}   \sigma(\langle{\bf v},{\bf w}\rangle+u) \alpha(\d {\bf w}, \d u). 
\end{equation}
Finally,  let $A \in \mathcal{B}(\mathcal{X} \times \R)$ satisfy $\pi(A)=0$ and for $u \in \R$ denote $A_u =\{ {\bf w} \in \mathcal{X} : ({\bf w},u) \in A \text{ for some } u \in \R \} $. If (i) holds, then the assumptions that  $\pi = \pi_{\mathcal{X}} \otimes (\pi_\R(x)\d x)$ and $\pi_\R>0$ imply that $\pi_{\mathcal{X}}(A_u) = 0$ for Lebesgue-a.e.\ $u \in \R$. Consequently,  $ |\hat{\mu}|(A_u)+|\hat{\mu}|^-(A_u) = 0$ and $\alpha(A)=0$. In case (ii) one may proceed similarly to obtain in either case that $\alpha \ll \pi$. Writing
\[ g({\bf w}) = \frac{\d ( |\hat{\mu}|+|\hat{\mu}|^-)}{\d \pi_{\mathcal{X}}}({\bf w}), \quad {\bf w} \in \mathcal{X} \]
one uses $|\bar{h}({\bf w})|\leq \sqrt{2}$ to estimate for any $({\bf w},u) \in \mathcal{X} \times \R$ that
\begin{equation}\begin{aligned}\label{eq:auxEq73} \left|\frac{\d \alpha}{\d \pi }({\bf w},u)\right| \leq (\mathbbm{1}_{(-M \|{\bf w}\|,0]}(u) +\sqrt{2}\mathbbm{1}_{[-1,1]}(u)) \frac{1}{\pi_\R(u)}g({\bf w}).
\end{aligned} \end{equation}

\textit{Step 2: Importance sampling.}

Next, write ${\bf U}_i=({\bf A}_i,
\zeta_i)$,  define the random variables
\begin{equation}\label{eq:Videf} V_i = \frac{\d \alpha}{\d \pi }({\bf U}_i) \end{equation}
and set
\begin{equation} \label{eq:WFNNdef}
{\bf W} = \frac{1}{N} \begin{pmatrix}
  V_1 &
   \cdots    &
   V_N
 \end{pmatrix}.
\end{equation}
By first inserting the definitions and then using independence, conditioning (see for instance \cite[Lemma~2.11]{Kallenberg2002}) and the assumption that $\mathcal{X}$ is separable we obtain 
\begin{equation}\label{eq:auxEq72}\begin{aligned} \E[\|H_{\bf W}^{{\bf A},\bm{\zeta}}({\bf Z}) - H^*({\bf Z})\|^2] & = \E[|{\bf W} \bm{\sigma}({\bf A} {\bf Z} + \bm{\zeta}) - H^*({\bf Z})|^2] 
\\ &  = \E\left[\left.\E\left[\left|\frac{1}{N}\sum_{i=1}^N V_i \sigma(\langle{\bf A}_i ,{\bf z}\rangle + \zeta_i) - H^*({\bf z})\right|^2\right]\right|_{{\bf z}={\bf Z}}\right]. \end{aligned} \end{equation}
However, by construction each of the summands $ V_i \sigma(\langle{\bf A}_i ,{\bf z}\rangle + \zeta_i)$ in \eqref{eq:auxEq72} is a random variable with expectation $H^*({\bf z})$, as one sees by using the representation \eqref{eq:auxEq71} to calculate for each $i=1,\ldots,N$ and any ${\bf z} \in \mathcal{X}$ with $\|{\bf z}\| \leq M$
%\label{eq:auxEq70}
\begin{equation*}\begin{aligned} \E[V_i \sigma(\langle{\bf A}_i ,{\bf z}\rangle + \zeta_i) ] & 
= \int_{\mathcal{X} \times \R} \frac{\d \alpha}{\d \pi}({\bf w},u) \sigma(\langle {\bf w} ,{\bf z}\rangle + u)  \pi(\d{\bf w},\d u)  
\\
& = H^*({\bf z}).\end{aligned} \end{equation*}
Using independence one thus obtains
\begin{equation} \label{eq:auxEq74} \begin{aligned}
\E\left[\left|\frac{1}{N}\sum_{i=1}^N V_i \sigma(\langle{\bf A}_i ,{\bf z}\rangle + \zeta_i) - H^*({\bf z})\right|^2\right] & = \mathrm{Var}\left(\frac{1}{N}\sum_{i=1}^N V_i \sigma(\langle{\bf A}_i ,{\bf z}\rangle + \zeta_i)\right)
\\ & = \frac{1}{N} \mathrm{Var}\left(V_1 \sigma(\langle{\bf A}_1 ,{\bf z}\rangle + \zeta_1)\right)
\\ & \leq \frac{1}{N} \E\left[V_1^2 \sigma(\langle{\bf A}_1 ,{\bf z}\rangle + \zeta_1)^2\right].
\end{aligned}\end{equation}
To estimate the last expectation, one notes that \eqref{eq:auxEq73} and \eqref{eq:VarianceHilbert} yield for any ${\bf z} \in \mathcal{X}$ with $\|{\bf z}\| \leq M$
\begin{equation}\begin{aligned} \label{eq:Cstar} \E& \left[ V_1^2 \sigma(\langle{\bf A}_1 ,{\bf z}\rangle + \zeta_1)^2\right]  \\ &  = \int_{\mathcal{X} \times \R} \left(\frac{\d \alpha}{\d {\pi}}({\bf w},u)\right)^2 \sigma(\langle {\bf w} ,{\bf z}\rangle + u)^2  \pi(\d{\bf w},\d u)
\\ & \leq 2 \int_{\mathcal{X} \times \R} (\mathbbm{1}_{(-M \|{\bf w}\|,0]}(u) +2\mathbbm{1}_{[-1,1]}(u)) \left(\frac{g({\bf w})}{\pi_\R(u)}\right)^2 \sigma(\langle {\bf w} ,{\bf z}\rangle + u)^2  \pi(\d{\bf w},\d u)
\\ & \leq 2 \int_{\mathcal{X}} \int_\R \left[ \mathbbm{1}_{(-M \|{\bf w}\|,0]}(u)|\langle {\bf w} ,{\bf z}\rangle|^2 +4\mathbbm{1}_{[-1,1]}(u)(|\langle {\bf w} ,{\bf z}\rangle|^2+1) \right] \frac{g({\bf w})^2}{\pi_\R(u)}  \pi_\mathcal{X}(\d{\bf w}) \d u
\\ & \leq M^2 \int_{\mathcal{X}} F_{\pi}(M \|{\bf w}\|)\|{\bf w}\|^2 g({\bf w})^2 \pi_\mathcal{X}(\d{\bf w}) \\ & \quad \quad + 8 M^2 (F_{\pi}(1)-F_{\pi}(-1)) \int_{\mathcal{X}}  \max(\|{\bf w}\|^2,1) g({\bf w})^2 \pi_\mathcal{X}(\d{\bf w})
\\ & = C^* < \infty.
\end{aligned}
\end{equation}
Combining \eqref{eq:auxEq72}, \eqref{eq:auxEq74} and \eqref{eq:Cstar} thus yields
\[
\E \left[\left(\int_{\mathcal{X}} \|H_{\bf W}^{{\bf A},\bm{\zeta}}({\bf z}) - H^*({\bf z})\|^2 \mu_{{\bf Z}}(\d {\bf z})\right)^{1/2}\right] \leq \E[\|H_{\bf W}^{{\bf A},\bm{\zeta}}({\bf Z}) - H^*({\bf Z})\|^2]^{1/2} \leq \frac{\sqrt{C^*}}{\sqrt{N}}.
\]
Thus, for any given $\delta \in (0,1)$ one may set $\eta = \frac{\sqrt{C^*}}{\delta \sqrt{N}}$ and apply Markov's inequality to obtain
\[ 
\P\left( \left(\int_{\mathcal{X}} \|H_{\bf W}^{{\bf A},\bm{\zeta}}({\bf z}) - H^*({\bf z})\|^2 \mu_{{\bf Z}}(\d {\bf z}) \right)^{1/2} > \eta \right) \leq \frac{1}{\eta} \frac{\sqrt{C^*}}{\sqrt{N}} = \delta.
\]\qed
\end{proof}

\subsection{Results in the finite-dimensional case}
\label{sec:Rqcase}

Let us now specialize to the case $\mathcal{X} = \R^q$. We work in the setting and notation as introduced in Section~\ref{sec:hilbertcase} and, in particular, consider random neural networks
\begin{equation}\label{eq:rNNDef} H^{{\bf A},\bm{\zeta}}_{\bf W}({\bf z})= {\bf W}  \bm{\sigma}({\bf A} {\bf z} + \bm{\zeta}), \quad {\bf z} \in \R^q. \end{equation}
 Thus, Theorem~\ref{thm:approxstatic} provides a random neural network approximation result for a wide range of sampling distributions $\pi$ for the weights. However, these assumptions may not allow us to sample the weights from a uniform distribution, unless the Fourier representation of $H^*$ is compactly supported.  In this section we prove that this case can still be covered by applying the representation from Proposition~\ref{prop:SmoothRepresentation}. To simplify the statements we choose $m=1$ here, but all the results can be directly generalized to $m \in \N^+$. {In line with Remark~\ref{rmk:leastsquares} the ``existence'' statement in the next proposition also directly yields approximation error bounds for the random neural network with readout ${\bf W}$ trained by least-squares minimization.}

\begin{proposition}\label{prop:BarronFininteDim}
Suppose
 $H^* \colon \R^q \to \R$ can be represented as
 \[ H^*({\bf z}) = \int_{\R^q} e^{i \langle {\bf w}, {\bf z} \rangle } g({\bf w}) \d {\bf w} \]
 for some complex-valued function $g$ on $\R^q$ and all ${\bf z} \in \R^q$ with $\|{\bf z}\| \leq M$. Assume that 
 \begin{equation} \label{eq:BarronCondRd} \int_{\R^q} \max(1,\|{\bf w}\|^{2q+6}) |g({\bf w})|^2 \d {\bf w} < \infty. \end{equation}
 Let $R>0$, suppose the rows of {the $\M_{N,N}$-valued random matrix} ${\bf A}$ are {i.i.d. random variables with} uniform distribution on $B_R \subset \R^q$, {suppose the entries of the $\R^N$-valued random vector $\bm{\zeta}$ are i.i.d. random variables} uniformly distributed on $[-\max(M R,1),\max(M R,1)]$, {assume that ${\bf A}$ and $\bm{\zeta}$ are independent and let} $\sigma \colon \R \to \R$ be given as $\sigma(x)=\max(x,0)$.
 Then, there exists ${\bf W}$ (a  $\M_{1,N}$-valued random variable) and $C^*>0$ such that 
 \begin{equation}\begin{aligned} \label{eq:MSENN}  \E[\|H_{\bf W}^{{\bf A},\bm{\zeta}}({\bf Z}) - H^*({\bf Z})\|^2]   \leq \frac{C^*}{N} \end{aligned} \end{equation}
 and 
  for any $\delta \in (0,1)$, with probability $1-\delta$ the random neural network $H_{\bf W}^{{\bf A},\bm{\zeta}}$ satisfies
\[
\left(\int_{\R^q} \|H_{\bf W}^{{\bf A},\bm{\zeta}}({\bf z}) - H^*({\bf z})\|^2 \mu_{{\bf Z}}(\d {\bf z}) \right)^{1/2} \leq \frac{\sqrt{C^*}}{\delta \sqrt{N}}.
\]
Moreover, the constant $C^*$ is explicit (see \eqref{eq:CstarDef2} below).
\end{proposition}

\begin{proof} Firstly, the function $H^*$ satisfies the hypotheses of Proposition~\ref{prop:SmoothRepresentation}. Thus, there exists an integrable function $\pi^* \colon \R^{q+1} \to \R$ such that for ${\bf z} \in \R^q$ with $\|{\bf z}\| \leq M$ the function $H^*$ can be represented as
\[H^*({\bf z}) = \int_{\R^{q+1}} \sigma({\bf z} \cdot {\bf w} + u) \pi^*({\bf w},u) \d {\bf w} \d u \]
and $\pi^*({\bf w},u) = 0$ for all $({\bf w},u) \in \R^q \times \R$ satisfying $\|{\bf w}\|>R$ or $|u|>\max(M R,1)$. Moreover, 
\begin{equation}\begin{aligned} \label{eq:auxEq76} \int_{\R^{q+1}}  \|\bm{\omega}\|^2 \pi^*(\bm{\omega})^2 \d \bm{\omega}  & \leq 8 (M^3+M+2)
\left( \int_{B_R} \max(1,\|{\bf w}\|^3) |g({\bf w})|^2 \d {\bf w} \right. \\ & \qquad + \left. \int_{\R^{q} \setminus B_R} \frac{\max(1,\|{\bf w}\|^{2q+5})}{R^{2q+2}} |g({\bf w})|^2  \d {\bf w}\right).
\end{aligned}\end{equation}
Recall that by assumption $\pi = \pi_{\mathcal{X}} \otimes \pi_\R$, where $\pi_{\mathcal{X}}$ is the uniform distribution on $B_R$ and $\pi_\R$ is the uniform distribution on $[-\max(M R,1),\max(M R,1)]$. Hence, setting $\alpha = \pi^*(\bm{\omega})\d \bm{\omega}$ one has that   \eqref{eq:auxEq71} holds, $\alpha \ll \pi$ and $\frac{\d \alpha}{\d \pi} = 2\max(M R,1) \mathrm{Vol}_q(B_R) \pi^* $. Thus, one may now mimic Step~2 in the proof of Theorem~\ref{thm:approxstatic}, i.e. \eqref{eq:Videf}-\eqref{eq:auxEq74}, to obtain
\begin{equation}\label{eq:auxEq75}\begin{aligned} \E[\|H_{\bf W}^{{\bf A},\bm{\zeta}}({\bf Z}) - H^*({\bf Z})\|^2] & \leq \frac{1}{N} \E\left[V_1^2 \sigma(\langle{\bf A}_1 ,{\bf z}\rangle + \zeta_1)^2\right]. \end{aligned} \end{equation} 
Furthermore, for any ${\bf z} \in \mathcal{X}$ with $\|{\bf z}\| \leq M$
\begin{equation}\begin{aligned}\label{eq:auxEq77} \E& \left[ V_1^2 \sigma(\langle{\bf A}_1 ,{\bf z}\rangle + \zeta_1)^2\right]  \\ &  = \int_{\R^q \times \R} \left(\frac{\d \alpha}{\d {\pi}}({\bf w},u)\right)^2 \sigma(\langle {\bf w} ,{\bf z}\rangle + u)^2  \pi(\d{\bf w},\d u)
\\ & = 2\max(M R,1) \mathrm{Vol}_q(B_R) \int_{\R^q \times \R} (\pi^*({\bf w},u))^2 \sigma(\langle {\bf w} ,{\bf z}\rangle + u)^2  \d{\bf w}\d u
\\ & \leq 2\max(M R,1) \mathrm{Vol}_q(B_R) (M+1)^2 \int_{\R^{q+1}}\|\bm{\omega}\|^2 \pi^*(\bm{\omega})^2 \d \bm{\omega}.
\end{aligned}
\end{equation}
Combining \eqref{eq:auxEq76}, \eqref{eq:auxEq75} and \eqref{eq:auxEq77} thus yields \eqref{eq:MSENN}, as desired, with 
\begin{equation}\begin{aligned}\label{eq:CstarDef2}
C^* = & 16\max(M R,1) \mathrm{Vol}_q(B_R) (M+1)^2 (M^3+M+2)
\\ & \quad \cdot
\left( \int_{B_R} \max(1,\|{\bf w}\|^3) |g({\bf w})|^2 \d {\bf w} + \int_{\R^{q} \setminus B_R} \frac{\max(1,\|{\bf w}\|^{2q+5})}{R^{2q+2}} |g({\bf w})|^2  \d {\bf w}\right).
\end{aligned}
\end{equation}
The high-probability statement then follows from \eqref{eq:MSENN} precisely as in the proof of Theorem~\ref{thm:approxstatic}.\qed
\end{proof}
In the next result we present an alternative error estimate, for which the integrability condition on $g$ does not depend on the input dimension $q$ (compare \eqref{eq:BarronCondRd} to \eqref{eq:BarronCondRd2}). The estimate can be deduced from the error estimate in Proposition~\ref{prop:BarronFininteDim} by truncating $g$ and estimating the difference between the truncation and the original $H^*$. Recall that ${\bf Z}$ is a $\R^q$-valued random variable satisfying $\|{\bf Z}\|\leq M$, $\P$-a.s. {We emphasize that the ``existence'' statement in the following corollary also  yields approximation error bounds for the random neural network with readout ${\bf W}$ trained by least-squares minimization, see Remark~\ref{rmk:leastsquares}.}

\begin{corollary} \label{cor:RandomNNApprox}
Suppose
 $H^* \colon \R^q \to \R$ can be represented as
 \begin{equation}\label{eq:FourierRepresentation} H^*({\bf z}) = \int_{\R^q} e^{i \langle {\bf w}, {\bf z} \rangle } g({\bf w}) \d {\bf w} \end{equation}
 for some complex-valued function $g \in L^1(\R^q)$ and all ${\bf z} \in \R^q$ with $\|{\bf z}\| \leq M$. Assume that 
 \begin{equation} \label{eq:BarronCondRd2} C_g^* = \left(\int_{\R^q} \max(1,\|{\bf w}\|^{3}) |g({\bf w})|^2 \d {\bf w}\right)^{1/2} < \infty. \end{equation}
 Let $R>0$, suppose the rows of {the $\M_{N,N}$-valued random matrix} ${\bf A}$ are {i.i.d. random variables with} uniform distribution on $B_R \subset \R^q$, {suppose the entries of the $\R^N$-valued random vector $\bm{\zeta}$ are i.i.d. random variables} uniformly distributed on $[-\max(M R,1),\max(M R,1)]$, {assume that ${\bf A}$ and $\bm{\zeta}$ are independent and let} $\sigma \colon \R \to \R$ be given as $\sigma(x)=\max(x,0)$.
 Then there exists ${\bf W}$ (a  $\M_{1,N}$-valued random variable) such that 
 \begin{equation}\begin{aligned} \label{eq:MSENN2}  \E[\|H_{\bf W}^{{\bf A},\bm{\zeta}}({\bf Z}) - H^*({\bf Z})\|^2]^{1/2}   \leq \frac{\sqrt{C^*_R}}{\sqrt{N}} + \int_{\R^q \setminus B_R} |g({\bf w})| \d {\bf w}, \end{aligned} \end{equation}
 where 
 \begin{equation}\label{eq:CstardefCor}
 C^*_R = 16\max(M R,1) \mathrm{Vol}_q(B_R) (M+1)^2 (M^3+M+2)
 \int_{B_R} \max(1,\|{\bf w}\|^3) |g({\bf w})|^2 \d {\bf w}.
 \end{equation}
 
 In particular, writing $c_{M,q}^2 =  16\max(M,1) \mathrm{Vol}_q(B_1) (M+1)^2 (M^3+M+2)$, it follows that:
 \begin{itemize}
 \item[(i)] if $I_k = \int_{\R^q } \|{\bf w} \|^k |g({\bf w})| \d {\bf w} < \infty$ for some $k\in \N^+$, then $R=N^{\frac{1}{2k+q+1}}$ yields
 \begin{equation}\begin{aligned} \label{eq:MSENN2PolMoment}  \E[\|H_{\bf W}^{{\bf A},\bm{\zeta}}({\bf Z}) - H^*({\bf Z})\|^2]^{1/2}   \leq \frac{1}{N^{[2+\frac{(q+1)}{k}]^{-1}}} \left[c_{M,q} C_g^* + I_k\right],  \end{aligned} \end{equation}
 \item[(ii)] if $I_k = \int_{\R^q } \exp(C\|{\bf w} \|^k) |g({\bf w})| \d {\bf w} < \infty$ for some $k\in \N^+$ and $C>0$, then $R=(\frac{\log(\sqrt{N})}{C})^{1/k}$ yields
 \begin{equation}\begin{aligned} \label{eq:MSENN2ExpMoment}  \E[\|H_{\bf W}^{{\bf A},\bm{\zeta}}({\bf Z}) - H^*({\bf Z})\|^2]^{1/2}   \leq \frac{[\log(\sqrt{N})]^{(q+1)/(2k)}}{\sqrt{N}} \left[C^{-\frac{q+1}{2k}} c_{M,q} C_g^* + I_k\right]. \end{aligned} \end{equation}
 \end{itemize} 
\end{corollary}

\begin{remark} 
A sufficient condition for \eqref{eq:FourierRepresentation}-\eqref{eq:BarronCondRd2} to be satisfied is that $H^* \in L^1(\R^q)$ has an integrable Fourier transform and belongs to the Sobolev space $W^{2,2}(\R^{q})$, see for instance \cite[Theorem~6.1]{Folland1995}. The integrability conditions formulated in parts (i) and (ii) are related to additional smoothness properties of $H^*$, where a higher degree of smoothness means that the Fourier transform of $H^*$ decays more quickly and consequently, the expressions $I_k$ in  (i) or (ii) are finite for larger $k \in \N^+$. This results in a faster rate of convergence in the bounds \eqref{eq:MSENN2PolMoment} and \eqref{eq:MSENN2ExpMoment}. For instance, if the condition in part (i) is satisfied, then the error in \eqref{eq:MSENN2PolMoment} is of order $O(N^{-[2+\frac{(q+1)}{k}]^{-1}})$ which is close to $O(1/\sqrt{N})$ when $(q+1)/k$ is small. Thus, as in classical works (see for instance \cite{Mhaskar1996}) the approximation rate depends on the ratio of the input dimension and the smoothness of the function to be approximated. {A similar result for functions in $W^{k,2}(\R^{q})$ for $k>\frac{q}{2}+1$ is formulated in Corollary~\ref{cor:sobolev} below}. 
\end{remark}

\begin{proof}%[Proof of Corollary~\ref{cor:RandomNNApprox}]
Define $\bar{g}({\bf w}) = \mathbbm{1}_{B_R}({\bf w}) g({\bf w})$ and 
 \[ \bar{H}^*({\bf z}) = \int_{\R^q} e^{i \langle {\bf w}, {\bf z} \rangle } \bar{g}({\bf w}) \d {\bf w}. \]
Then 
\[ \int_{\R^q} \max(1,\|{\bf w}\|^{2q+6}) |\bar{g}({\bf w})|^2 \d {\bf w} \leq \max(1,R^{2q+3}) \int_{\R^q} \max(1,\|{\bf w}\|^{3}) |g({\bf w})|^2 \d {\bf w} < \infty \]
and so Proposition~\ref{prop:BarronFininteDim} (applied to $\bar{H}^*$) shows that there exists ${\bf W}$ such that 
\begin{equation}\label{eq:auxEq79}
\E[\|H_{\bf W}^{{\bf A},\bm{\zeta}}({\bf Z}) - \bar{H}^*({\bf Z})\|^2]   \leq \frac{C^*_R}{N}
\end{equation}
with $C^*_R$ given in \eqref{eq:CstardefCor}.
Furthermore, the triangle inequality yields 
\[ \begin{aligned}
\E[\|H^*({\bf Z}) - \bar{H}^*({\bf Z})\|^2]^{1/2} & = \E\left[\left|\int_{\R^q \setminus B_R} e^{i \langle {\bf w}, {\bf Z} \rangle } g({\bf w}) \d {\bf w}\right|^2\right]^{1/2}
\\ & \leq \int_{\R^q \setminus B_R} |g({\bf w})| \d {\bf w}.
\end{aligned} \]
Combining this with \eqref{eq:auxEq79} and the triangle inequality then yields \eqref{eq:MSENN2}. Finally, let us show that the assumptions in (i) and (ii) guarantee a certain decay of the last term in \eqref{eq:MSENN2}. 

\medskip

\noindent (i) Suppose $I_k = \int_{\R^q } \|{\bf w} \|^k |g({\bf w})| \d {\bf w} < \infty$ for some $k\in \N^+$. Then
\[\int_{\R^q \setminus B_R} |g({\bf w})| \d {\bf w} \leq \int_{\R^q \setminus B_R} \left(\frac{ \|{\bf w} \|}{R}\right)^k |g({\bf w})| \d {\bf w} \leq \frac{I_k}{R^k}.
\]
Thus, the right hand side in \eqref{eq:MSENN2} is bounded by  
\[ \frac{\sqrt{C^*_R}}{\sqrt{N}} + \frac{I_k}{R^k} \leq \frac{R^{\frac{q+1}{2}}c_{M,q} C_g^* }{\sqrt{N}} + \frac{I_k}{R^k}, 
\]
which becomes the right hand side of \eqref{eq:MSENN2PolMoment} if we take $R=N^\alpha$ and choose $\alpha$ to make both terms of the same order, i.e.\ $\alpha \frac{q+1}{2} - \frac{1}{2} = -\alpha k$. 

\medskip

\noindent (ii) Suppose $I_k = \int_{\R^q } \exp(C\|{\bf w} \|^k) |g({\bf w})| \d {\bf w} < \infty$ for some $k\in \N^+$ and $C>0$. Then 
\[\int_{\R^q \setminus B_R} |g({\bf w})| \d {\bf w} \leq \int_{\R^q \setminus B_R} \exp(C[\|{\bf w} \|^k - R^k]) |g({\bf w})| \d {\bf w} \leq \frac{I_k}{\exp(C R^k)}.
\]
Thus, taking $R=(\frac{\log(\sqrt{N})}{C})^{1/k}$, the right hand side in \eqref{eq:MSENN2} is bounded by  
\[ \frac{\sqrt{C^*_R}}{\sqrt{N}} + \frac{I_k}{\exp(C R^k)} \leq \frac{R^{\frac{q+1}{2}} c_{M,q} C_g^* }{\sqrt{N}} + \frac{I_k}{\sqrt{N}}.
\]\qed
\end{proof}

{
Recall that $W^{k,2}(\R^q)$ denotes for $k \in \N^+$  the Sobolev space consisting of all functions $u \colon \R^q \to \R$ whose mixed partial derivatives $D^{\alpha} u$ of order $\alpha \in \N^q$ with $\alpha_1+\cdots+\alpha_q \leq k$ satisfy $D^{\alpha} u \in L^2(\R^q)$. For $u \in L^1(\R^q)$ we denote by $\widehat{u}(\bm{\xi})=\int_{\R^q}  e^{-i \langle \bm{\xi}, {\bf z} \rangle } u({\bf z}) \d {\bf z}$, $\bm{\xi} \in \R^q$, the Fourier transform of $u$. By \cite[Theorem~6.1]{Folland1995} the space $W^{k,2}(\R^q)$ consists of precisely those $u \in L^2(\R^q)$ for which the norm 
\[
\| u \|_{k} := \left(\int_{\R^q} |\widehat{u}(\bm{\xi})|^2 (1+\|\bm{\xi}\|^2)^k  \d \bm{\xi}\right)^{1/2}
\] 
is finite.}

{ 
The next corollary specializes Corollary~\ref{cor:RandomNNApprox} to functions in $W^{k,2}(\R^q)$ for sufficiently large $k$. Analogous results could be derived for the Sobolev spaces $W^{k,p}(\R^q)$ for $p>1$ and for generalized Sobolev spaces (see \cite[Section~6.2]{Bergh1976}) with $p=1$ even without restrictions on $k$.
\begin{corollary}\label{cor:sobolev}
Let $k\in \N$ with $k \geq \frac{q}{2}+1 + \varepsilon$ for some $\varepsilon>0$ and suppose $H^* \in W^{k,2}(\R^q) \cap L^1(\R^q)$.
Let $R=N^{1/(2k-2\varepsilon+1)}$, suppose the rows of {the $\M_{N,N}$-valued random matrix} ${\bf A}$ are {i.i.d. random variables with} uniform distribution on $B_R \subset \R^q$, {suppose the entries of the $\R^N$-valued random vector $\bm{\zeta}$ are i.i.d. random variables} uniformly distributed on $[-\max(M R,1),\max(M R,1)]$, {assume that ${\bf A}$ and $\bm{\zeta}$ are independent and let} $\sigma \colon \R \to \R$ be given as $\sigma(x)=\max(x,0)$.
Then, there exists ${\bf W}$ (a  $\M_{1,N}$-valued random variable) and a constant $C>0$ (depending on $q$ and $M$, but independent of $H^*$, $N$) such that 
\begin{equation}\begin{aligned} \label{eq:NNMSE}  \E[\|H_{\bf W}^{{\bf A},\bm{\zeta}}({\bf Z}) - H^*({\bf Z})\|^2]^{1/2}   \leq C \| H^* \|_{k} N^{-1/\alpha}   \end{aligned} \end{equation}
with $\alpha = 2+\frac{(q+1)}{k-q/2-\varepsilon}$. $C$ is explicitly given in \eqref{eq:cdef}.
\end{corollary}
}

{
\begin{remark}
In the neural network $H_{\bf W}^{{\bf A},\bm{\zeta}}$ only the weights ${\bf W}$ are trainable, whereas ${\bf A},\bm{\zeta}$ are generated randomly. Therefore, it is clear that the approximation capabilities of these networks are smaller than those of neural networks in which all parameters can be trained. This intuition is confirmed when comparing the  error rate $1/\alpha \leq 1/2$ in  \eqref{eq:NNMSE} to the rate  $k/q$ obtained in \cite{Mhaskar1996}, \cite{Maiorov2000}, since $k/q \geq 1/2$ for $k\geq q/2$.  
However, the advantage of the result in Corollary~\ref{cor:sobolev} is that training the random neural network $H_{\bf W}^{{\bf A},\bm{\zeta}}$ is straightforward: it only requires to solve the (convex) optimization problem over ${\bf W}$, which is mathematically very well-understood. In contrast, in the case of fully trainable neural networks one typically uses stochastic gradient descent type algorithms for parameter optimization, for which a rigorous mathematical error analysis for general shallow neural networks is challenging.
\end{remark}
}

{
\begin{proof} 
%Firstly, let $\chi \in C_c^\infty(\R^q)$ be a compactly supported, smooth function with $\chi({\bf z}) = 1$ on  $\|{\bf z}\| \leq M$. Then $\widehat{\chi}$ is a Schwartz function and hence $\widehat{H^* \chi} = \widehat{\chi} * \widehat{H^*}$ by \cite[Theorem~9.16(ii)]{Amann2009} and $\widehat{\chi}$ is bounded and integrable. In addition, for $\|{\bf z}\| \leq M$ it follows from the properties of $\chi$ and Fourier inversion that
%\[
%H^*({\bf z}) = \frac{1}{(2 \pi)^q} \int_{\R^q} e^{i \langle {\bf w}, {\bf z} \rangle } \widehat{H^* \chi}({\bf w}) \d {\bf w}
%\]
%which establishes representation \eqref{eq:FourierRepresentation} with $g = (2 \pi)^{-q} \widehat{\chi} * \widehat{H^*} \in L^1(\R^q)$.  
Firstly, using the Cauchy-Schwartz inequality and the assumptions on $H^*$ we obtain that $\widehat{H^*} \in L^1(\R^q)$, see \eqref{eq:auxEq103}. Hence, the Fourier inversion theorem yields the representation \eqref{eq:FourierRepresentation} with $g = (2 \pi)^{-q}  \widehat{H^*} \in L^1(\R^q)$ and all ${\bf z} \in \R^q$. Thus, the constant $C_g^*$ in \eqref{eq:BarronCondRd2} is bounded by
\begin{equation*}  C_g^* \leq (2 \pi)^{-q}  \| H^* \|_{2}. \end{equation*}
Furthermore, 
\begin{equation} \label{eq:auxEq103} \begin{aligned}
\left(\int_{\R^q } \|{\bf w} \|^{k-s} |g({\bf w})| \d {\bf w} \right)^2  \leq \int_{\R^q } (1+\|{\bf w} \|^2)^{-s}  \d {\bf w}  \int_{\R^q } (1+\|{\bf w} \|^2)^{k} |  \widehat{H^*} ({\bf w})|^2 \d {\bf w},
\end{aligned}
\end{equation}
which is finite for $s>q/2$ (see e.g. \cite[p.193]{Folland1995}). Choosing $s = q/2+\varepsilon$ yields $k-q/2-\varepsilon\geq1 $ and we may therefore apply Corollary~\ref{cor:RandomNNApprox} to obtain that there exists ${\bf W}$ (a  $\M_{1,N}$-valued random variable) such that 
\begin{equation*}\begin{aligned}  \E[\|H_{\bf W}^{{\bf A},\bm{\zeta}}({\bf Z}) - H^*({\bf Z})\|^2]^{1/2} &  \leq N^{-[2+\frac{(q+1)}{k-s}]^{-1}} \left[c_{M,q} C_g^* + I_{k-s}\right]
\\ & \leq N^{-[2+\frac{(q+1)}{k-q/2-\varepsilon}]^{-1}} C \| H^* \|_{k}   \end{aligned} \end{equation*}
with $c_{M,q}^2 =  16\max(M,1) \mathrm{Vol}_q(B_1) (M+1)^2 (M^3+M+2)$ and 
\begin{equation} \label{eq:cdef} C = c_{M,q} (2 \pi)^{-q} + \left(\int_{\R^q } (1+\|{\bf w} \|^2)^{-s}  \d {\bf w} \right)^{1/2}.\end{equation} This completes the proof. 
\qed
\end{proof}
}

Finally, we prove a further consequence of Theorem~\ref{thm:approxstatic}. The result in Proposition~\ref{prop:lowerDimApprox} below allows for a larger class of functions $H^*$ (possibly defined in terms of essentially lower-dimensional functions, for instance as a sum of univariate functions) and shows in particular how the sampling scheme in the previous results can be modified in order to cover this more general case; while the rows of ${\bf A}$ were sampled from the uniform distribution on the ball $B_R \subset \R^q$ in Corollary~\ref{cor:RandomNNApprox} above, in Proposition~\ref{prop:lowerDimApprox} the matrix ${\bf A}$ is in general a sparse random matrix with entries drawn from lower dimensional balls $B_R^k \subset \R^k$, $k =1,\ldots,q$. 

\begin{proposition}\label{prop:lowerDimApprox}
Suppose
 $H^* \colon \R^q \to \R$ can be represented as
 \[ H^*({\bf z}) = \int_{\R^q} e^{i \langle {\bf w}, {\bf z} \rangle } \hat{\mu}(\d {\bf w}) \]
 for some complex measure $\hat{\mu}$ on $(\R^q,\mathcal{B}(\R^q))$ and all ${\bf z} \in \R^q$ with $\|{\bf z}\| \leq M$. Assume that 
 %\label{eq:BarronCondRd4}
 \begin{equation*}  \int_{\R^q} \max(1,\|{\bf w}\|^{2}) |\hat{\mu}|(\d {\bf w}) < \infty. \end{equation*}
 Suppose $K_1,\ldots,K_N$ are i.i.d.\ random variables with values in $\{1,\ldots,q\}$ and for $i=1,\ldots,N$, conditional on $K_i = k$ the $i$-th row ${\bf A}_i$ of ${\bf A}$ is sampled as follows:
 \begin{itemize}
 \item select (uniformly randomly on $\{1,\ldots,q\}$) $k$ non-zero entries
 \item draw these entries from the uniform distribution on $B_R \subset \R^{k}$ 
 \item set the remaining $N-k$ entries to $0$
 \end{itemize} and $\zeta_i$ is sampled uniformly on $[-\max(M R,1),\max(M R,1)]$. For $k=1,\ldots,K$ denote by $\lambda_1$ the Lebesgue-measure on $\R$, let $p_k = \P(K_1=k)$ and assume that 
 \begin{equation} \label{eq:absCont}
 \hat{\mu} \ll \sum_{k=1}^q p_k \sum_{\substack{\mu_1,\ldots,\mu_q \in \{\delta_0,\lambda_1 \} \\ \#\{j \,:\, \mu_j = \lambda_1 \} = k }} \mu_1 \otimes \cdots \otimes \mu_q.
 \end{equation}
Let $\sigma \colon \R \to \R$ be given as $\sigma(x)=\max(x,0)$. Then
 \begin{itemize}
 \item[(i)] $\mathbbm{1}_{B_R}|\hat{\mu}|+\mathbbm{1}_{B_R}|\hat{\mu}|^- \ll \pi_{\mathcal{X}}$, where $\pi_{\mathcal{X}}$ denotes the distribution of ${\bf A}_i$,
 \item[(ii)] if $g = \frac{\d (\mathbbm{1}_{B_R}|\hat{\mu}|+\mathbbm{1}_{B_R}|\hat{\mu}|^-)}{ \d \pi_{\mathcal{X}}}$  satisfies 
 \begin{equation}\label{eq:auxEq82} \int_{B_R}  \max(\|{\bf w}\|^3,1) g({\bf w})^2 \pi_\mathcal{X}(\d{\bf w}) < \infty,\end{equation}  then 
   there exists ${\bf W}$ (a  $\M_{1,N}$-valued random variable) such that 
  % \label{eq:MSENN4} 
  \begin{equation*}\begin{aligned}  \E[\|H_{\bf W}^{{\bf A},\bm{\zeta}}({\bf Z}) - H^*({\bf Z})\|^2]^{1/2}   \leq \frac{\sqrt{C}}{\sqrt{N}} + \int_{\R^q \setminus B_R} |\hat{\mu}|(\d {\bf w}), \end{aligned} \end{equation*}
  where 
  %\label{eq:CstardefCor2}
  \begin{equation*}
  C = 8 M^2 \max(MR,1) \max(M,4)  \int_{B_R}  \max(\|{\bf w}\|^3,1) g({\bf w})^2 \pi_\mathcal{X}(\d{\bf w}).
   \end{equation*}
 \end{itemize}

\end{proposition}

\begin{proof}
To prove (i), suppose $B \in \mathcal{B}(\R^q)$ satisfies $\pi_{\mathcal{X}}(B)=0$. Let ${\bf U}_k \sim \pi_k$, where $\pi_k$ denotes the uniform distribution on $B_R^k = B_R \subset \R^k$. By construction, for all $k=1,\ldots,q$ with $p_k>0$ and all $j_1,\ldots,j_k \in \{1,\ldots,q\}$ we have (with $T_{j_1,\ldots,j_k} $ denoting the map that embeds $ \R^k$ in $ \R^q$ by inserting $0$ at each component $j \notin \{j_1,\ldots,j_k\}$) that 
\[\begin{aligned} 0 & = \P({\bf A}_1 \in B | K_1=k, A_{1,j_1} \neq 0,\ldots,A_{1,j_k} \neq 0)  \\ & = \P(T_{j_1,\ldots,j_k}({\bf U}_k) \in B) =  \pi_k(T_{j_1,\ldots,j_k}^{-1}(B)).  \end{aligned} \]
Using that $\pi_k$ has a strictly positive Lebesgue density on $B_R^k \subset \R^k$, this implies that $T_{j_1,\ldots,j_k}^{-1}(B) \cap B_R^k$ is a Lebesgue-nullset in $\R^k$. 
Therefore for $\mu_{j_1,\ldots,j_k}$ = $\mu_1 \otimes \cdots \otimes \mu_q$ with $\mu_{j_i} = \lambda_1$ and $\mu_j = \delta_0$ for $j \notin \{j_1,\ldots,j_q\}$ it follows that 
\[ \mu_{j_1,\ldots,j_k}(B \cap B_R^q) = \int_{\R^k} \mathbbm{1}_{B}(T_{j_1,\ldots,j_k}(w_1,\ldots,w_k)) \mathbbm{1}_{ B_R^k}(w_1,\ldots,w_k) \d w_1 \cdots \d w_k = 0.  \]
This shows that $B \cap B_R^q$ is a nullset for each of the measures on the right hand side of \eqref{eq:absCont} and so, by \eqref{eq:absCont}, also for $\hat{\mu}$ (and consequently for $|\hat{\mu}|$ and $|\hat{\mu}|^-$).

To show (ii) note that $\P$-a.s. $\|{\bf A}_i\| \leq R$ and so we may
apply Theorem~\ref{thm:approxstatic} to $\mathcal{X} = \R^q$ and the function $\bar{H}^*({\bf z}) = \int_{\R^q} e^{i \langle {\bf w}, {\bf z} \rangle } \mathbbm{1}_{B_R}({\bf w})\hat{\mu}(\d {\bf w})$. By assumption on $\zeta_i$, the function $F_\pi$ appearing in  Theorem~\ref{thm:approxstatic} is given for $|x|\leq \max(MR,1)$ as $F_\pi(x) = 2 \int_{-x}^0 2 \max(MR,1) \d u =4 x \max(MR,1)$ and so $C^* = C_1^*$ in Theorem~\ref{thm:approxstatic} becomes
\[\begin{aligned} C^* & = 4 M^2 \max(MR,1) \left(M \int_{B_R} \|{\bf w}\|^3 g({\bf w})^2 \pi_\mathcal{X}(\d{\bf w}) \right. \\ & \qquad + \left. 4 \int_{B_R}  \max(\|{\bf w}\|^2,1) g({\bf w})^2 \pi_\mathcal{X}(\d{\bf w}) \right)
\\ & \leq 8 M^2 \max(MR,1) \max(M,4)  \int_{B_R}  \max(\|{\bf w}\|^3,1) g({\bf w})^2 \pi_\mathcal{X}(\d{\bf w}) \end{aligned}\]
and \eqref{eq:VarianceHilbert} is indeed satisfied by \eqref{eq:auxEq82}. 
The statement then follows precisely as in the proof of Corollary~\ref{cor:RandomNNApprox} by estimating the difference $|\bar{H}^*({\bf z})- H^*({\bf z})| \leq \int_{\R^q \setminus B_R} |\hat{\mu}|(\d {\bf w}) $ for ${\bf z} \in B_M \subset \R^q$ and applying the triangle inequality. \qed
\end{proof}

\subsection{Universal approximation by random ReLU networks}
\label{sec:staticUniversal}
In this subsection we present a further corollary, which proves that feedforward neural networks with randomly generated inner weights are universal approximators in $L^2(\R^q,\mu)$ for any probability measure $\mu$ on $(\R^q,\mathcal{B}(\R^q))$. 

To formulate the result let us first introduce the scheme according to which the weights are sampled. For any $\rho>1$, $R >0$ consider the following scheme to randomly generate weights:
\begin{itemize}
\item[(i)] Let ${\bf A}_1,{\bf A}_2,\ldots$ be i.i.d. random vectors drawn from the uniform distribution on the ball $B_R \subset \R^q$,
\item[(ii)] let $\zeta_1,\zeta_2,\ldots$ be i.i.d.\ uniformly distributed on $[-\rho,\rho]$, independent of $\{ {\bf A}_i \}_{i \in \N^+}$. 
\end{itemize}
{Note that the only parameters that need to be trained for the neural networks in Corollary~\ref{cor:2} are the outer weights $W_1,\ldots,W_N$ (once $N$ is fixed and the inner weights ${\bf A}_1,{\bf A}_2,\ldots$, $\zeta_1,\zeta_2,\ldots$ are sampled randomly). These  outer weights can be trained using least-squares minimization.}

\begin{corollary}\label{cor:2} Let $\mu$ be a probability measure on $\R^q$, $G \in L^2(\R^q,\mu)$ and let $\sigma \colon \R \to \R$ be given as $\sigma(x)=\max(x,0)$. Then for any $\varepsilon>0$, $\delta \in (0,1)$ there exist $N \in \N^+$, $R>0$, $\rho>1$ and {real valued random variables $W_1,\ldots,W_N$ (``outer weights'')} such that the random feedforward neural network (with {``inner weights'' $({\bf A}_1,\zeta_1),({\bf A}_2,\zeta_2),\ldots$} sampled as in (i)-(ii)) {specified by}
\[
G_N({\bf z}) = \sum_{i=1}^N W_i \sigma({\bf A}_i \cdot {\bf z} + \zeta_i), \quad {\bf z} \in \R^q
\]
approximates $G$ in $L^2(\R^q,\mu)$ up to precision $\varepsilon$ with probability $1-\delta$, that is, 
\[
 \int_{\R^q} |G({\bf z})-G_N({\bf z})|^2 \mu(\d {\bf z}) < \varepsilon^2.
\]
\end{corollary}

\begin{proof}
Firstly, by using \cite[Lemma~1.33]{Kallenberg2002} and the fact that {the set of compactly supported infinitely often differentiable functions} $C_c^\infty(\R^q)$ is dense in {the space of continuous functions with compact support} $C_c(\R^q)$ in the supremum norm we find $H^* \in C_c^\infty(\R^q)$ satisfying
\begin{equation}\label{eq:auxEq101}
\left[\int_{\R^q} |H^*({\bf z})-G({\bf z})|^2 \mu(\d {\bf z})\right] ^{1/2}< \frac{\varepsilon \sqrt{\delta}}{2}.
\end{equation}
Denoting by $\widehat{H^*}({\bf w})=\int_{\R^q}  e^{-i \langle {\bf w}, {\bf z} \rangle } H^*({\bf z}) \d {\bf z}$ the Fourier transform of $H^*$ and setting $g = (2\pi)^{-q} \widehat{H^*}$, it follows that $H^*$ can be represented as \eqref{eq:FourierRepresentation} for all ${\bf z} \in \R^q$, that $g \in L^1(\R^q)$ and \eqref{eq:BarronCondRd2} holds. 
Choose $M>0$ large enough to guarantee that the support of $H^*$ is contained in $B_M$, denote by $\tilde{{\bf Z}}$ a random variable with distribution $\mu$ and set ${\bf Z}=\tilde{{\bf Z}}\mathbbm{1}_{B_M}(\tilde{{\bf Z}})+{\bf z}_0 \mathbbm{1}_{\R^q \setminus B_M}(\tilde{{\bf Z}})$ for an arbitrary ${\bf z}_0 \in \overline{B_M} \setminus B_M$. Then $\|{\bf Z}\|\leq M$ and $H^*({\bf Z}) = H^*(\tilde{{\bf Z}})$ and all the assumptions of Corollary~\ref{cor:RandomNNApprox} are satisfied. We now select the hyperparameters as follows: choose $R>0$ large enough to guarantee $\int_{\R^q \setminus B_R} |g({\bf w})| \d {\bf w} < \frac{\varepsilon \sqrt{\delta}}{4}$ and then take $N \in \N^+$ to guarantee $\frac{\sqrt{C^*_R}}{\sqrt{N}} < \frac{\varepsilon \sqrt{\delta}}{4}$ (with $C^*_R $ given in \eqref{eq:CstardefCor}). Furthermore, let $\rho = \max(M R,1)$. Inserting these estimates in the right hand side of \eqref{eq:MSENN2} and applying Corollary~\ref{cor:RandomNNApprox} shows that there exists  ${\bf W}=(W_1 \, \cdots \, W_N)$ (a  $\M_{1,N}$-valued random variable) such that 
\[
\E[|H_{\bf W}^{{\bf A},\bm{\zeta}}({\bf Z}) - H^*({\bf Z})|^2]^{1/2} < \frac{\varepsilon \sqrt{\delta}}{2}.
\]
Combining this with \eqref{eq:auxEq101}, $H^{{\bf A},\bm{\zeta}}_{\bf W}({\bf z})= {\bf W} \bm{\sigma}({\bf A} {\bf z} + \bm{\zeta}) = G_N({\bf z})$, $H^*({\bf Z}) = H^*(\tilde{{\bf Z}})$ and the triangle inequality yields
\[
\E\left[\int_{\R^q} |G({\bf z})-G_N({\bf z})|^2 \mu(\d {\bf z})\right]^{1/2} < \varepsilon \sqrt{\delta}. \]
Applying Markov's inequality then shows that 
\[ 
\P\left( \left(\int_{\R^q} |G({\bf z})-G_N({\bf z})|^2 \mu(\d {\bf z}) \right)^{1/2} > \varepsilon \right) \leq \frac{1}{\varepsilon^2} \E\left[\int_{\R^q} |G({\bf z})-G_N({\bf z})|^2 \mu(\d {\bf z})\right] < \delta,
\]
as claimed.\qed
\end{proof}

\section{Approximation Error Estimates For Echo State Networks}\label{sec:ESN}
In the results formulated above in Section~\ref{sec:Static} we were concerned with the static situation and approximations based on random neural networks. We now turn to the dynamic case. Thus, we consider $D_d \subset \R^d$ and inputs given by semi-infinite sequences in $\mathcal{X}=(D_d)^{\Z_-}$. The unknown mapping that needs to be approximated is denoted by $H^* \colon (D_d)^{\Z_-} \to \R^m$ and is called a functional (see also Section~\ref{sec:Preliminaries} for further preliminaries on the dynamic situation). In applications, $H^*$ is typically approximated by reservoir functionals. Recall that a reservoir functional is a mapping $H^{RC}$ defined as the input-to-solution map   $ \mathcal{X} \ni {\bf z} \mapsto {\bf y}_0 \in \R^m$ of the state space system \eqref{eq:RCSystemDet}-\eqref{eq:RCSystemDetYeqn}.  The goal of this section is to derive bounds for the error that arises when approximating the functional $H^*$ by such reservoir functionals. We will be focusing on two of the most prominent families of reservoir systems, namely linear systems with neural network readouts (Section~\ref{sec:dynLin}) as well as echo state networks (Section~\ref{sec:dynESN}). Beforehand, in Section~\ref{sec:dynRegularity} we  introduce the setting in more detail, describe the regularity assumption that is imposed on $H^*$ in both cases and characterize a general class of examples in which it is satisfied. As a corollary of the approximation error bounds derived in Section~\ref{sec:dynESN} we prove in Section~\ref{sec:dynUniversality} that echo state networks with randomly generated recurrent weights are universal approximators.
This proves, in particular, that echo state networks with randomly generated weights are capable of approximating a large class of input/output systems arbitrarily well and, in conjunction with the error estimates in Theorem~\ref{thm:ESNErrorBound}, thus provides the first mathematical explanation for the empirically observed success of echo state networks in the learning of that kind of systems.

\subsection{Setting and regular functionals}
\label{sec:dynRegularity}

In order to approximate the unknown functional $H^* \colon (D_d)^{\Z_-} \to \R^m$, in applications the procedure is typically as follows. In a first step, the reservoir map $F$ in \eqref{eq:RCSystemDet} is fixed (often generated randomly). Then the readout function $h$ in \eqref{eq:RCSystemDetYeqn} is trained by minimizing a prefixed loss function in order to approximate $H^*$ as well as possible. 
In what follows we will be  interested in quantifying the error committed when using an approximating reservoir functional for $H^*$ \textit{conditional on} the random elements used to generate it and with respect to the $L^2((D_d)^{\Z_-},\mu_{{\bf Z}})$-norm for a probability measure $\mu_{{\bf Z}}$ on the space of inputs $((D_d)^{\Z_-},\mathcal{B}((D_d)^{\Z_-})$. More specifically, throughout this section, ${\bf Z}$ is a $(D_d)^{\Z_-}$-valued random variable, that is, a discrete-time stochastic process, we denote by $\mu_{{\bf Z}}$ its law on $(D_d)^{\Z_-}$ and
we assume that ${\bf 0} \in D_d \subset B_M \subset \R^d$.  To simplify the statements we choose $m=1$ here, but all the results can be directly generalized to $m \in \N^+$. 

The functionals $H^*$, for which the approximation bounds in Section~\ref{sec:dynLin} and Section~\ref{sec:dynESN} can be derived, are required to satisfy certain regularity assumptions. These will be stated in Assumption~\ref{ass:regularityFctl} below. Beforehand, we introduce a Lipschitz-continuity condition which quantifies how quickly $H^*$ forgets  past inputs and is thus linked to its \textit{memory}, see also \cite{RC9} for a thorough discussion.

\begin{definition} 
%\label{ass:HLipschitz}
Consider a sequence $w \in (0,\infty)^{\Z_-}$ with $\sum_{j \in \Z_-} |j|w_j < \infty$.
We say that $H^*$ is \textit{$w$-Lipschitz continuous}, if there exists $L > 0$ such that 
\begin{equation} \label{eq:ZFMPProperty}
|H^*({\bf u})-H^*({\bf v}) | \leq  L \|{\bf u}-{\bf v}\|_{1,w} 
\end{equation} 
for all ${\bf u}= ({\bf u}_t)_{t \in \Z_-} \in (D_d)^{\Z_-}$, ${\bf v}= ({\bf v}_t)_{t \in \Z_-} \in (D_d)^{\Z_-}$, where 
\[ \|{\bf u}-{\bf v}\|_{1,w} := \sum_{i=0}^{\infty} w_{-i} \|{\bf u}_{-i} - {\bf v}_{-i}\|.  \]
\end{definition}

\begin{assumption}\label{ass:regularityFctl}
Suppose that $H^* \colon (D_d)^{\Z_-} \to \R$ is $w$-Lipschitz continuous for some $w \in (0,\infty)^{\Z_-}$ with $\sum_{j \in \Z_-} |j|w_j < \infty$ and assume that for any $T \in \N^+$: 
\begin{description}
\item[(i)] The restriction of $H^*$ to sequences of length $T$, which is given by the function $H^*_T \colon (D_d)^{T+1} \to \R$ defined by $H^*_T({\bf z}_0,\ldots,{\bf z}_{-T}) := H^*(\ldots,0,{\bf z}_{-T},\ldots,{\bf z}_{0})$, can be represented as  
\[ H_T^*({\bf u}) = \int_{\R^q} e^{i \langle {\bf w}, {\bf u} \rangle } g_T({\bf w}) \d {\bf w} \]
for a $\C$-valued function $g_T \in L^1(\R^q)$ and all ${\bf u} = ({\bf z}_0,\ldots,{\bf z}_{-T}) \in (D_d)^{T+1} \subset \R^q$, with $q:=d(T+1)$.
\item[(ii)] 
 \begin{equation} \label{eq:BarronCondRd3} \int_{\R^q} \max(1,\|{\bf w}\|^{3}) |g_T({\bf w})|^2 \d {\bf w} < \infty. \end{equation}
\end{description}

\end{assumption}

 We now provide a general class of examples that satisfy Assumption~\ref{ass:regularityFctl}. This class includes, for example, state affine systems, linear systems with polynomial readouts, and trigonometric state affine systems as long as the matrix coefficients in these systems fulfill certain conditions  that guarantee that the condition (i) in the next proposition is satisfied. We refer to \cite{RC10,RC20,RC8,RC7,RC6,RC9} for a detailed discussion of these systems.

\begin{proposition} Let $\rho >0$ and suppose $H^*$ is the reservoir functional associated to the reservoir system \eqref{eq:RCSystemDet}-\eqref{eq:RCSystemDetYeqn} determined by the restriction to $\overline{B_\rho} \times D_d$ and $\overline{B_\rho} $ of  the maps $F \colon \R^{N^*} \times \R^d \to \R^{N^*}$ and $h \colon \R^{N^*} \to \R$, respectively, and that satisfy the following hypotheses. Firstly,  $F(\overline{B_\rho} \times D_d) \subset \overline{B_\rho}$ and, additionally,  there exist  $r \in (0,1)$, $L_F, L_h >0$, such that
\begin{itemize}
\item[(i)] for any ${\bf z} \in D_d$, $\left. F \right|_{\overline{B_\rho} \times D_d}(\cdot,{\bf z})$ is an $r$-contraction,
\item[(ii)] for any ${\bf x} \in \overline{B_\rho}$, $\left. F \right|_{\overline{B_\rho} \times D_d}({\bf x},\cdot)$ is $L_F$-Lipschitz, 
\item[(iii)] $F$ and $h$ are both infinitely differentiable. 
\end{itemize}
Then $H^*= h(H_{\left. F \right|_{\overline{B_\rho} \times D_d}})$ satisfies Assumption~\ref{ass:regularityFctl}. 
\end{proposition}
\begin{proof}
Firstly, (iii) and the mean value theorem imply that $\left. h \right|_{\overline{B_\rho}} \colon \overline{B_\rho} \to \R$ is Lipschitz continuous. In what follows we denote by $L_h$ the best Lipschitz constant of $\left. h \right|_{\overline{B_\rho}}$.
Secondly, note that Proposition~\ref{prop:contractionEchoState} guarantees that $H_{\left. F \right|_{\overline{B_\rho} \times D_d}}$ is indeed well-defined. For notational simplicity write $H_F = H_{\left. F \right|_{\overline{B_\rho} \times D_d}}$. Then for any ${\bf u}, {\bf v} \in (D_d)^{\Z_-}$
\[\begin{aligned}
 & \|H_F({\bf u})-  H_F({\bf v}) \| \\ &  = \|F(H_F({\bf u}_{\cdot-1}),{\bf u}_0) - F(H_F({\bf v}_{\cdot-1}),{\bf v}_0) \| 
\\ & \leq \|F(H_F({\bf u}_{\cdot-1}),{\bf u}_0) - F(H_F({\bf v}_{\cdot-1}),{\bf u}_0) \| + \|F(H_F({\bf v}_{\cdot-1}),{\bf u}_0) - F(H_F({\bf v}_{\cdot-1}),{\bf v}_0) \|  \\ & \leq r \|H_F({\bf u}_{\cdot-1}) - H_F({\bf v}_{\cdot-1})\| + L_F \|{\bf u}_{0}-{\bf v}_{0}\|,
\end{aligned} \]
where we used the echo state property in the first step, then the triangle inequality and finally hypotheses (i)-(ii).
Iterating this estimate we obtain 
\[
|H^*({\bf u})-H^*({\bf v}) | \leq L_h L_F \sum_{k=0}^\infty r^{k} \|{\bf u}_{-k}-{\bf v}_{-k}\| = L \|{\bf u}-{\bf v}\|_{1,w}
\]
for $L = L_h L_F$ and $w_{-j} = r^j$, $j \in \N$. This proves that $H^*$ is $w$-Lipschitz continuous.

Let $T \in \N^+$. By the echo state property we can write $H^*_T$ as 
\begin{equation}\label{eq:auxEq89}
H^*_T({\bf z}_0,\ldots,{\bf z}_{-T}) = h \circ F(\cdot,{\bf z}_0) \circ \ldots \circ F(H^*(\ldots,0,0),{\bf z}_{-T})
\end{equation}
for $({\bf z}_0,\ldots,{\bf z}_{-T}) \in (D_d)^{T+1}$. The expression on the right hand side of \eqref{eq:auxEq89} can be used to extend $H^*_T$ to $(\R^d)^{T+1} = \R^q$
and hypothesis~(iii) implies that $H^*_T$ is infinitely often differentiable. Let $\chi \colon \R \to \R$ be a compactly supported $C^\infty$ function that satisfies $\chi(x)=1$ for $x \in [-M^2,M^2]$. Define $G \colon (\R^d)^{T+1} \to \R$ by 
\[G({\bf u}_0,\ldots,{\bf u}_{T}) = H^*_T({\bf u}_0,\ldots,{\bf u}_{T}) \chi(\|{\bf u}_0\|^2) \cdots \chi(\|{\bf u}_T\|^2).\]
Then for $({\bf z}_0,\ldots,{\bf z}_{-T}) \in (D_d)^{T+1}$ one has $\|{\bf z}_{-i}\|\leq M$ and thus $\chi(\|{\bf z}_{-i}\|^2) = 1$ for $i=0,\ldots,T$. Consequently, $G = H^*_T$ on $(D_d)^{T+1}$. Therefore, the claim will follow if we prove that $G$ can be represented as
\begin{equation}\label{eq:auxEq90} G({\bf u}) = \int_{\R^q} e^{i \langle {\bf w}, {\bf u} \rangle } g_T({\bf w}) \d {\bf w} \end{equation}
for some  $g_T \in L^1(\R^q)$ satisfying \eqref{eq:BarronCondRd3} and for all ${\bf u} \in \R^q$. However, $G$ is a smooth function with compact support and therefore a Schwartz function. Thus, its Fourier transform $\hat{G}({\bf w}) = \int_{\R^q} e^{-i \langle {\bf w}, {\bf u} \rangle } G({\bf u}) \d {\bf u} $ is also a Schwartz function. The Fourier inversion theorem thus yields \eqref{eq:auxEq90} with $g_T = (2 \pi)^{-q} \hat{G}$ and the integrability conditions $g_T \in L^1(\R^q)$ and \eqref{eq:BarronCondRd3} hold because $g_T$ is a Schwartz function.\qed
\end{proof}

\subsection{Approximation based on Linear Reservoir Systems with Random Neural Network Readouts}\label{sec:dynLin}
In this section we study approximations of the unknown functional  $H^*$ based on reservoir functionals  $H^{RC}$ determined by (random) linear reservoir systems with  \textit{random} neural network readouts. More precisely, for $q, N \in \N^+$ let ${\bf S} \in \M_q$, ${\bf c} \in \M_{q,d}$  and let  ${\bf A}$ and $\bm{\zeta}$ be $\M_{N,q}$ and $\M_{N,1}$-valued random matrices and vectors, respectively. For any readout matrix ${\bf W} \in \M_{1,N}$ consider the reservoir system given by
\begin{equation}\label{eq:RCRandomNNReadout}
\left\{
\begin{aligned}
\mathbf{X}_t & = {\bf S} \mathbf{X}_{t-1} + {\bf c} \mathbf{Z}_t, \quad t \in \Z_-, \\ 
Y_t & = {\bf W} \bm{\sigma}({\bf A} \mathbf{X}_t + \bm{\zeta}), \quad t \in \Z_-.
\end{aligned}
\right.
\end{equation}
Clearly, when the associated system with deterministic inputs  ${\bf z} \in (D_d)^{\Z_-}$ (which is a linear system with random neural network readout, see \eqref{eq:rNNDef}) given by
\begin{align}\label{eq:LinearSystem}
\mathbf{x}_t & = {\bf S} \mathbf{x}_{t-1} + {\bf c} {\bf z}_t, \quad t \in \Z_-, \\ \label{eq:NNreadout}
y_t & = H_{\bf W}^{{\bf A}, \bm{\zeta}}(\mathbf{x}_t), \quad t \in \Z_-,
\end{align}
has the echo state property, then the solution to \eqref{eq:RCRandomNNReadout} can be obtained by evaluating the filter associated to \eqref{eq:LinearSystem}-\eqref{eq:NNreadout} at the stochastic input ${\bf Z}$. 

\begin{remark}
For notational simplicity we take ${\bf S}, {\bf c}$ deterministic here. However, Proposition~\ref{prop:linearreservoir} directly extends to randomly drawn ${\bf S}, {\bf c}$ satisfying $\P$-a.s. the hypotheses of Proposition~\ref{prop:linearreservoir}. The expectation in \eqref{eq:MSENN5} is then conditional on ${\bf S}, {\bf c}$. 
\end{remark}

\begin{proposition}\label{prop:linearreservoir} Let $N,T \in \N^+$, $R,M_T>0$ and $q=d(T+1)$.
Suppose the rows of ${\bf A}$ are sampled from the uniform distribution on $B_R \subset \R^{q}$ and the entries of $\bm{\zeta}$ are uniformly distributed on $[-\max(M_T R,1),\max(M_T R,1)]$, let $\sigma \colon \R \to \R$ be given as $\sigma(x)=\max(x,0)$, assume that 
\eqref{eq:LinearSystem} satisfies the echo state property, the matrix  
\[ {\bf K} =   \begin{pmatrix} {\bf c} & {\bf S} {\bf c} & \cdots & {\bf S}^T {\bf c} \end{pmatrix}  \]
is invertible, $\|\mathbf{X}_0\| \leq M_T$ and  ${\bf K}^{-1} \mathbf{X}_0 \in (D_d)^{T+1}$.
Then for any $H^* \colon (D_d)^{\Z_-} \to \R$ satisfying Assumption~\ref{ass:regularityFctl}
there exists ${\bf W}$ (a  $\M_{1,N}$-valued random variable) such that 
 \begin{equation} \label{eq:MSENN5} \begin{aligned}
 \E[|Y_0 - H^*({\bf Z})|^2]^{1/2}   & \leq \frac{\sqrt{C_{T,R}}}{\sqrt{N}} + |\det({\bf K})| \int_{\R^q \setminus B_R} |g_T({\bf K}^{\top} {\bf w})| \d {\bf w} + L M  \left(\sum_{i=T+1}^\infty w_{-i} \right) \\ & \quad + L \left(\sum_{i=0}^T w_{-i}^2 \right)^{1/2}   \|{\bf K}^{-1} {\bf S}^{T+1} \mathbf{X}_{-T-1} \| \end{aligned}
 \end{equation}
 where $C_{T,R}$ is given in \eqref{eq:CTRdef}.
 
\end{proposition}

\begin{remark}  The bound in Proposition~\ref{prop:linearreservoir} shows, in particular, that for suitable choices of ${\bf S}$ (for instance as given in Remark~\ref{rmk:shift} below) the approximation error can be made arbitrarily small. Indeed, if $\varepsilon > 0$ is given, $T$ is large enough and ${\bf S}^{T+1} = 0$, then  the last term in \eqref{eq:MSENN5} vanishes and the third term satisfies $ L M  \sum_{i=T+1}^\infty w_{-i} < \frac{\varepsilon}{3}$, since the weighting sequence $w$ is summable.  Next, one chooses $R >0$ to make $|\det({\bf K})| \int_{\R^q \setminus B_R} |g_T({\bf K}^{\top} {\bf w})| \d {\bf w} < \frac{\varepsilon}{3}$ (this is possible, since $g$ is integrable) and finally (with $R, T$ now fixed) $N$ so that 
$\frac{\sqrt{C_{T,R}}}{\sqrt{N}} < \frac{\varepsilon}{3}$.
Altogether, one obtains that
 \[ \E[|Y_0 - H^*({\bf Z})|^2]^{1/2} < \varepsilon. \] 

\end{remark}

\begin{proof}
Firstly, the hypothesis that $H^*$ is $w$-Lipschitz continuous  yields (see \eqref{eq:ZFMPProperty}) for any ${\bf z} \in (D_d)^{\Z_-}$ 
\begin{equation}\label{eq:auxEq80}
|H^*_T({\bf z}_0,\ldots,{\bf z}_{-T}) - H^*({\bf z})| \leq L \left(\sum_{i=T+1}^\infty w_{-i} \|{\bf z}_{-i}\|\right) \leq L M  \left(\sum_{i=T+1}^\infty w_{-i} \right).
\end{equation}
Secondly, using once more the $w$-Lipschitz property \eqref{eq:ZFMPProperty} and H\"older's inequality show for any ${\bf u}= ({\bf u}_t)_{t=0,\ldots,T}, {\bf v}= ({\bf v}_t)_{t=0,\ldots,T} \in (D_d)^{T+1}$ that
\[ |H^*_T({\bf u}) - H^*_T({\bf v})| \leq L \left(\sum_{i=0}^T w_{-i} \|{\bf u}_{i} -{\bf v}_i \|\right) \leq L \left(\sum_{i=0}^T w_{-i}^2 \right)^{1/2}  \left(\sum_{i=0}^T \|{\bf u}_{i} -{\bf v}_i \|^2 \right)^{1/2} \]
and therefore 
\begin{equation}\label{eq:auxEq83}
|H^*_T({\bf Z}_0,\ldots,{\bf Z}_{-T}) - H^*_T({\bf K}^{-1} \mathbf{X}_0)| \leq L \left(\sum_{i=0}^T w_{-i}^2 \right)^{1/2}  \|({\bf Z}_{0},\ldots,{\bf Z}_{-T}) - {\bf K}^{-1} \mathbf{X}_0 \|.
\end{equation}
Iterating \eqref{eq:LinearSystem} yields the representation
\[ \mathbf{X}_0 = \sum_{i=0}^T {\bf S}^i {\bf c} {\bf Z}_{-i} + {\bf S}^{T+1} \mathbf{X}_{-T-1} = {\bf K} \begin{pmatrix}
{\bf Z}_0 \\ \vdots \\ {\bf Z}_{-T}
\end{pmatrix} +  {\bf S}^{T+1} \mathbf{X}_{-T-1}, \]
which we insert in \eqref{eq:auxEq83} to obtain 
\begin{equation}\label{eq:auxEq84} \begin{aligned}
|H^*_T({\bf Z}_0,\ldots,{\bf Z}_{-T}) - H^*_T({\bf K}^{-1} \mathbf{X}_0)| & \leq L \left(\sum_{i=0}^T w_{-i}^2 \right)^{1/2}  \|{\bf K}^{-1}  {\bf S}^{T+1} \mathbf{X}_{-T-1} \|.
\end{aligned} \end{equation}
Thirdly, consider the function $G \colon B_{M_T} \to \R$ defined for ${\bf v} \in B_{M_T} \subset \R^q$ by 
\[G({\bf v}) = |\det({\bf K})| \int_{\R^q} e^{i \langle {\bf w}, {\bf v} \rangle } g_T({\bf K}^{\top} {\bf w}) \d {\bf w},\]
which is indeed well-defined because $g_T$ is integrable.
Then the change of variables formula and  Assumption~\ref{ass:regularityFctl}  yield 
\[ \begin{aligned}
H^*_T({\bf K}^{-1} \mathbf{X}_0) &  = \int_{\R^q} e^{i \langle {\bf K}^{-\top} {\bf w},  \mathbf{X}_0 \rangle } g_T({\bf w}) \d {\bf w} \\ & = |\det({\bf K})| \int_{\R^q} e^{i \langle {\bf w},  \mathbf{X}_0 \rangle } g_T({\bf K}^{\top} {\bf w}) \d {\bf w} = G( \mathbf{X}_0).
\end{aligned} \]
Therefore, the function $G$ satisfies the hypotheses of Corollary~\ref{cor:RandomNNApprox} (integrability again follows by the change of variables formula) and so by Corollary~\ref{cor:RandomNNApprox} there exists ${\bf W}$ (a  $\M_{1,N}$-valued random variable) such that 
 \begin{equation}\begin{aligned} \label{eq:MSENN3}  \E[|H_{\bf W}^{{\bf A},\bm{\zeta}}(\mathbf{X}_0) - [H^*_T \circ {\bf K}^{-1}](\mathbf{X}_0)|^2]^{1/2}   \leq \frac{\sqrt{C_{T,R}}}{\sqrt{N}} + |\det({\bf K})| \int_{\R^q \setminus B_R} |g_T({\bf K}^{\top} {\bf w})| \d {\bf w}, \end{aligned} \end{equation}
 where 
 \begin{equation}\label{eq:CTRdef} \begin{aligned}
 C_{T,R} & = 16\max(M_T R,1) \mathrm{Vol}_q(B_R) (M_T+1)^2 ([M_T]^3+M_T+2)
 \\ & |\det({\bf K})|^2 \int_{B_R} \max(1,\|{\bf w}\|^3) |g_T({\bf K}^{\top} {\bf w})|^2 \d {\bf w}. \end{aligned}
 \end{equation}
By using the triangle inequality and inserting the bounds obtained in \eqref{eq:auxEq80}, \eqref{eq:auxEq84} and \eqref{eq:MSENN3} one thus obtains the approximation bound \eqref{eq:MSENN5}, as claimed.
 \qed
\end{proof}

\begin{remark}\label{rmk:shift}
An important special case is 
\begin{equation} \label{eq:shiftMatrix}
{\bf S}= \rho \left(
\begin{array}{cc}
 \boldsymbol{0}_{d,dT}&\boldsymbol{0}_{d,d}\\
  \boldsymbol{I}_{dT}&\boldsymbol{0}_{d,d}
\end{array}
\right) \quad \mbox{and} \quad
{\bf c}= \left(
\begin{array}{c}
\boldsymbol{I}_{d}\\
\boldsymbol{0}_{dT,d}\\
\end{array}
\right)
\end{equation}
for $\rho \in (0,1]$.
In this case one calculates ${\bf S}^{T+1}=0$ and for $k=1,\ldots,T$
\[ {\bf S}^k {\bf c} = \rho^k \left(
\begin{array}{c}
\boldsymbol{0}_{dk,d}\\  \boldsymbol{I}_{d} \\
\boldsymbol{0}_{d(T-k),d}\\
\end{array}
\right).
\]
Thus, e.g. for $\rho=1$ one obtains ${\bf K} = \boldsymbol{I}_{d(T+1)}$ and so in particular ${\bf K}$ is invertible and $\|{\bf K}^{-1}\|=1$. 
In addition, the system \eqref{eq:LinearSystem} satisfies the echo state property and the solution is given by ${\bf x} _t= \left({\bf z} _t^{\top}, \rho {\bf z}_{t-1}^{\top},\ldots , \rho^T {\bf z}_{t-T}^{\top}\right)^\top$, $t \in \mathbb{Z}_{-} $. 
\end{remark}

\subsection{Approximation based on Echo State Networks}\label{sec:dynESN}

In this section we use an echo state network with randomly generated parameters as an approximation to the unknown target functional $H^*$. More precisely, for $\bar{N} \in \N^+$ let ${\bf A}$, ${\bf C}$ and  $\bm{\zeta}$ be $\M_{\bar{N}}$, $\M_{\bar{N},d}$ and $\M_{\bar{N},1}$-valued random matrices/vectors, respectively, and for any readout matrix ${\bf W} \in \M_{1,\bar{N}}$ consider the reservoir system given by
\begin{equation*} %\label{eq:ESN}
\left\{
\begin{aligned}
 \mathbf{x}_t &  = \bm{\sigma}( {\bf A} \mathbf{x}_{t-1} + {\bf C} {\bf z}_t + \bm{\zeta}), \quad t \in \Z_-,\\
y_t & = {\bf W} \mathbf{x}_t, \quad t \in \Z_-
\end{aligned}
\right.
\end{equation*} 
for ${\bf z} \in (D_d)^{\Z_-}$. Such a system is called an echo state network. If this RC system has the echo state property (see Section~\ref{sec:Preliminaries}), then the reservoir functional $H^{{\bf A},{\bf C},\bm{\zeta}}_{\bf W}({\bf z})= y_0$ (that is, the input-to-solution map $(D_d)^{\Z_-} \ni {\bf z} \mapsto y_0$) is well-defined and measurable. Evaluating $H^{{\bf A},{\bf C},\bm{\zeta}}_{\bf W}$ at the stochastic input signal ${\bf Z}$ then amounts to solving the associated system with stochastic input
\begin{equation} \label{eq:ESNViewpoint2}
\left\{
\begin{aligned}
 \mathbf{X}_t &  = \bm{\sigma}( {\bf A} \mathbf{X}_{t-1} + {\bf C} {\bf Z}_t + \bm{\zeta}), \quad t \in \Z_-,\\
Y_t & = {\bf W} \mathbf{X}_t, \quad t \in \Z_-.
\end{aligned}
\right.
\end{equation} 

The next result shows that it is possible to generate ${\bf A}$, ${\bf C}$ and  $\bm{\zeta}$ from a generic distribution (not depending on $H^*$) and use this generic echo state network to approximate $H^*$ arbitrarily well. Thus, ${\bf X}$ is universal and to approximate $H^*$  only the readout matrix ${\bf W} \in \M_{1,\bar{N}}$ needs to be trained, a task which amounts to a linear regression. 

\begin{theorem} \label{thm:ESNErrorBound}
 Let $\sigma \colon \R \to \R$ be given as $\sigma(x)=\max(x,0)$. Let $T, N \in \N^+$, $R>0$, assume that  $\|({\bf Z}_0,\ldots,{\bf Z}_{-T})\|_{\R^{d(T+1)}}\leq M_T$ and generate ${\bf A}, {\bf C}, \bm{\zeta}$ according to the following procedure:
\begin{itemize}
\item[(i)] draw $N$ i.i.d. samples ${\bf A}_1,\ldots,{\bf A}_N$ from the uniform distribution on $B_R \subset \R^{d(T+1)}$ and $N$ i.i.d. samples $\zeta_1,\ldots,\zeta_N$ (also independent of $\{{\bf A}_i\}_{i=1,\ldots,N}$) from the  uniform distribution on $[-\max(M_T R,1),\max(M_T R,1)]$ ,
\item[(ii)] let ${\bf S}$, ${\bf c}$ be the shift matrices defined in \eqref{eq:shiftMatrix} with $\rho=1$ and set 
\begin{equation*} \begin{aligned} {\bf a} = \begin{pmatrix}
{\bf A}_1^\top \\ \vdots \\ {\bf A}_N^\top \end{pmatrix},   \bar{{\bf A}} & = \begin{pmatrix} {\bf S} & \boldsymbol{0}_{q,N} \\ {\bf a} {\bf S} & \boldsymbol{0}_{N,N} \end{pmatrix}, \bar{{\bf C}}  = \begin{pmatrix} {\bf c}  \\ {\bf a} {\bf c} \end{pmatrix}, \bar{\bm{\zeta}} = \begin{pmatrix} \boldsymbol{0}_{q} \\ \zeta_1 \\ \vdots \\ \zeta_N \end{pmatrix}, \\  {\bf A} & = \begin{pmatrix} \bar{{\bf A}} & -\bar{{\bf A}} \\ -\bar{{\bf A}} & \bar{{\bf A}} \end{pmatrix}, {\bf C}  = \begin{pmatrix}  \bar{{\bf C}}  \\ - \bar{{\bf C}} \end{pmatrix}, \bm{\zeta} = \begin{pmatrix} \bar{\bm{\zeta}} \\ -\bar{\bm{\zeta}} \end{pmatrix}. \end{aligned} \end{equation*}
\end{itemize}
Then for any $H^* \colon (D_d)^{\Z_-} \to \R$ satisfying Assumption~\ref{ass:regularityFctl} 
there exists a readout ${\bf W}$ (a  $\M_{1,2(N+d(T+1))}$-valued random variable) such that the system \eqref{eq:ESNViewpoint2} satisfies the echo state property and
 \begin{equation} \label{eq:MSENN6} \begin{aligned}
 \E[|Y_0 - H^*({\bf Z})|^2]^{1/2}   & \leq \frac{\sqrt{C_{T,R}}}{\sqrt{N}} +  \int_{\R^q \setminus B_R} |g_T({\bf u})| \d {\bf u} + L M  \left(\sum_{i=T+1}^\infty w_{-i} \right)  \end{aligned}
 \end{equation}
 with 
   \begin{equation}\label{eq:CTRdef1} \begin{aligned}
   C_{T,R} & = 16\max(M_T R,1) \mathrm{Vol}_q(B_R) (M_T+1)^2 ([M_T]^3+M_T+2) \\ & \qquad \cdot \int_{B_R} \max(1,\|{\bf u}\|^3) |g_T({\bf u})|^2 \d {\bf u}. \end{aligned}
   \end{equation}
\end{theorem}

{
\begin{remark} \label{rmk:genericX}
The process ${\bf X}$ in \eqref{eq:ESNViewpoint2} is not related in any way to the unknown functional $H^*$. ${\bf X}$ is generic and can be viewed as a ``reservoir'' that efficiently stores the information about the history of the input process 
${\bf Z}$. Theorem~\ref{thm:ESNErrorBound} shows that for any ``sufficiently regular'' functional $H^*$ one can approximate $H^*({\bf Z})$ by ${\bf W} {\bf X}_0$ for an appropriately chosen ${\bf W}$, i.e.\ by applying a linear mapping to ${\bf X}_0$. This phenomenon is analogous to the situation encountered in continuous-time stochastic processes satisfying certain stochastic differential equations, which can be approximated by applying a linear functional to the signature of the driving path, see e.g.\ \cite[Chapter~5]{Kloeden1992}, \cite[Chapter~18]{Friz2010}.  See also \cite{JLpaper} and \cite{RC13}.
\end{remark}
}

\begin{remark}\label{rmk:Scmatrices}
For simplicity (and to give a fully constructive sampling procedure) we have chosen here for ${\bf S}$, ${\bf c}$ the shift matrices defined in \eqref{eq:shiftMatrix} with $\rho=1$. However, Theorem~\ref{thm:ESNErrorBound} can be directly generalized to $\rho \in (0,1)$ and arbitrary  ${\bf S}$, ${\bf c}$ satisfying the hypotheses stated in Proposition~\ref{prop:linearreservoir}. The bound \eqref{eq:MSENN6} is then replaced by the bound \eqref{eq:MSENN5} and the constant $C_{T,R}$ given in \eqref{eq:CTRdef1} is replaced by \eqref{eq:CTRdef}.
\end{remark}

\begin{remark} By using Markov's inequality the bound \eqref{eq:MSENN6} immediately yields a high-probability bound on the approximation error conditional on the reservoir parameters: for any $\delta \in (0,1)$ it holds with probability $1-\delta$ that the (random) echo state network $H^{{\bf A},{\bf C},\bm{\zeta}}_{\bf W}$ satisfies
\[
\left(\int_{(D_d)^{\Z_-}} |H^{{\bf A},{\bf C},\bm{\zeta}}_{\bf W}({\bf z}) - H^*({\bf z})|^2 \mu_{{\bf Z}}(\d {\bf z}) \right)^{1/2} \leq \frac{\phi(T,R,N)}{\delta},
\]
where $\phi(T,R,N)$ is the right hand side in \eqref{eq:MSENN6}.
\end{remark}

\begin{proof}
Firstly, Proposition~\ref{prop:linearreservoir} and Remark~\ref{rmk:shift} show that for any $H^*$ satisfying Assumption~\ref{ass:regularityFctl} there exists ${\bf w}$ (a  $\M_{1,N}$-valued random variable) such that the bound \eqref{eq:MSENN6} holds with $Y_0=Y_0^{\text{Lin}}$ satisfying
\begin{equation*}%\label{eq:auxEq85}
\left\{
\begin{aligned}
\mathbf{X}_t^{\text{Lin}} & = {\bf S} \mathbf{X}_{t-1}^{\text{Lin}} + {\bf c} \mathbf{Z}_t, \quad t \in \Z_-, \\ 
Y_t^{\text{Lin}} & = {\bf w} \bm{\sigma}({\bf a} \mathbf{X}_t^{\text{Lin}} + {\bf b}), \quad t \in \Z_-,
\end{aligned}
\right.
\end{equation*}
and ${\bf b}^\top = \begin{pmatrix}
\zeta_1 & \cdots & \zeta_N \end{pmatrix}$. Now set
$\bar{{\bf W}} = \begin{pmatrix} \boldsymbol{0}_{1,q} & {\bf w} \end{pmatrix}$ and ${\bf W} = \begin{pmatrix} \bar{{\bf W}} & \boldsymbol{0}_{1,q+N} \end{pmatrix}$.
We first show that \eqref{eq:ESNViewpoint2} has a solution.  To do this we define $\bar{\mathbf{X}}_t = \begin{pmatrix} \mathbf{X}_t^{\text{Lin}} \\ {\bf a} \mathbf{X}_t^{\text{Lin}} + {\bf b} \end{pmatrix}$ and claim that $\mathbf{X}_t = \begin{pmatrix} \bm{\sigma}(\bar{\mathbf{X}}_t) \\ \bm{\sigma}(-\bar{\mathbf{X}}_t) \end{pmatrix}$ is a solution to the first equation in \eqref{eq:ESNViewpoint2}. Indeed, we first calculate
\[ \bar{{\bf A}} \bar{\mathbf{X}}_{t-1} + \bar{{\bf C}} {\bf Z}_t  +  \bar{\bm{\zeta}} = \begin{pmatrix} {\bf S} \mathbf{X}_{t-1}^{\text{Lin}} + {\bf c} {\bf Z}_t \\ {\bf a} {\bf S} \mathbf{X}_{t-1}^{\text{Lin}} + {\bf a} {\bf c} {\bf Z}_t + {\bf b} \end{pmatrix} = \begin{pmatrix} \mathbf{X}_t^{\text{Lin}} \\ {\bf a} \mathbf{X}_t^{\text{Lin}} + {\bf b}  \end{pmatrix} = \bar{\mathbf{X}}_t \]
and then insert this to obtain
\[ \begin{aligned}
\bm{\sigma}( {\bf A} \mathbf{X}_{t-1} + {\bf C} {\bf Z}_t + \bm{\zeta}) & = \bm{\sigma}( \begin{pmatrix} \bar{{\bf A}}  \\ -\bar{{\bf A}}\end{pmatrix}(\bm{\sigma}(\bar{\mathbf{X}}_{t-1})-\bm{\sigma}(-\bar{\mathbf{X}}_{t-1})) + \begin{pmatrix}  \bar{{\bf C}}  \\ - \bar{{\bf C}} \end{pmatrix} {\bf Z}_t + \begin{pmatrix} \bar{\bm{\zeta}} \\ -\bar{\bm{\zeta}} \end{pmatrix}) 
\\ & = \bm{\sigma} \begin{pmatrix}  \bar{\mathbf{X}}_{t}  \\ -\bar{\mathbf{X}}_{t} \end{pmatrix} = \mathbf{X}_t,
\end{aligned} \]
as claimed. In addition, 
\begin{equation*}
Y_t = {\bf W} \mathbf{X}_t = \bar{{\bf W}}\bm{\sigma}(\bar{\mathbf{X}}_t) = {\bf w} \bm{\sigma}({\bf a} \mathbf{X}_t^{\text{Lin}} + {\bf b}) = Y_t^{\text{Lin}}
\end{equation*}
and so we have constructed a solution to \eqref{eq:ESNViewpoint2} and proved that \eqref{eq:MSENN6} holds. It remains to be  proved that the system  \eqref{eq:ESNViewpoint2} satisfies the echo state property. To do so, consider an arbitrary solution $({\bf U},\tilde{Y})$ to \eqref{eq:ESNViewpoint2}, i.e.\ $({\bf U},\tilde{Y})$ satisfying
\begin{equation*} %\label{eq:ESNViewpoint3}
\left\{
\begin{aligned}
 \mathbf{U}_t &  = \bm{\sigma}( {\bf A} \mathbf{U}_{t-1} + {\bf C} {\bf Z}_t + \bm{\zeta}), \quad t \in \Z_-,\\
\tilde{Y}_t & = {\bf W} \mathbf{U}_t, \quad t \in \Z_-.
\end{aligned}
\right.
\end{equation*} 
Partitioning $\mathbf{U}_t = \begin{pmatrix} \mathbf{U}_t^{[1]} \\ \mathbf{U}_t^{[2]} \end{pmatrix}$ (with $\mathbf{U}_t^{[i]}$ valued in $\R^{d(T+1)+N}$) and setting 
$\bar{\mathbf{U}}_t = \mathbf{U}_t^{[1]} - \mathbf{U}_t^{[2]}$
one calculates
\begin{equation}\label{eq:auxEq88} \begin{aligned}
\mathbf{U}_t & = \bm{\sigma}( \begin{pmatrix} \bar{{\bf A}}  \\ -\bar{{\bf A}}\end{pmatrix}(\mathbf{U}_{t-1}^{[1]} - \mathbf{U}_{t-1}^{[2]}) + \begin{pmatrix}  \bar{{\bf C}}  \\ - \bar{{\bf C}} \end{pmatrix} {\bf Z}_t + \begin{pmatrix} \bar{\bm{\zeta}} \\ -\bar{\bm{\zeta}} \end{pmatrix}) \\ &  = \bm{\sigma} \begin{pmatrix} \bar{{\bf A}} \bar{\mathbf{U}}_{t-1} + \bar{{\bf C}} {\bf Z}_t + \bar{\bm{\zeta}} \\ -(\bar{{\bf A}} \bar{\mathbf{U}}_{t-1} + \bar{{\bf C}} {\bf Z}_t + \bar{\bm{\zeta}})\end{pmatrix}
\end{aligned} \end{equation}
and therefore 
\begin{equation}\label{eq:auxEq86}
\bar{\mathbf{U}}_t = \bm{\sigma}(\bar{{\bf A}} \bar{\mathbf{U}}_{t-1} + \bar{{\bf C}} {\bf Z}_t + \bar{\bm{\zeta}}) - \bm{\sigma}(-(\bar{{\bf A}} \bar{\mathbf{U}}_{t-1} + \bar{{\bf C}} {\bf Z}_t + \bar{\bm{\zeta}})) = \bar{{\bf A}} \bar{\mathbf{U}}_{t-1} + \bar{{\bf C}} {\bf Z}_t + \bar{\bm{\zeta}}.
\end{equation}
By further partitioning $\bar{\mathbf{U}}_t = \begin{pmatrix} \bar{\mathbf{U}}_t^{[1]} \\ \bar{\mathbf{U}}_t^{[2]} \end{pmatrix}$ (with $\bar{\mathbf{U}}_t^{[1]}$ valued in $\R^{d(T+1)}$ and $\bar{\mathbf{U}}_t^{[2]}$ valued in $\R^{N}$) one obtains from \eqref{eq:auxEq86} that
\begin{equation} \label{eq:auxEq87}
\begin{pmatrix} \bar{\mathbf{U}}_t^{[1]} \\ \bar{\mathbf{U}}_t^{[2]} \end{pmatrix} = \begin{pmatrix} {\bf S} \bar{\mathbf{U}}_{t-1}^{[1]} + {\bf c} {\bf Z}_t\\ {\bf a} {\bf S} \bar{\mathbf{U}}_{t-1}^{[1]} + {\bf a} {\bf c} {\bf Z}_t + {\bf b} \end{pmatrix}.
\end{equation}
However, the linear system \eqref{eq:LinearSystem} satisfies the echo state property and so $\bar{\mathbf{U}}_t^{[1]} = \mathbf{X}_t^{\text{Lin}}$. Inserting this in \eqref{eq:auxEq87} shows that   $\bar{\mathbf{U}}_t^{[2]} = {\bf a} \mathbf{X}_t^{\text{Lin}} + {\bf b}$. This proves that $\bar{\mathbf{U}}_t = \bar{\mathbf{X}}_t$. Using this in the second step and inserting \eqref{eq:auxEq86} into \eqref{eq:auxEq88} in the first step shows that 
\[ \mathbf{U}_t = 
\begin{pmatrix} \bm{\sigma}(\bar{\mathbf{U}}_t) \\ \bm{\sigma}(-\bar{\mathbf{U}}_t) \end{pmatrix} = \begin{pmatrix} \bm{\sigma}(\bar{\mathbf{X}}_t) \\ \bm{\sigma}(-\bar{\mathbf{X}}_t) \end{pmatrix} = \mathbf{X}_t
\]
and hence also $\tilde{Y}=Y$, as claimed.\qed
\end{proof}

{
\begin{remark}
As explained in Remark~\ref{rmk:genericX} the state process ${\bf X}$ can be viewed as a ``reservoir'' that stores the history of the input process ${\bf Z}$. Choosing ${\bf X}$ as an echo state network, i.e.\ evolving according to the dynamics specified in \eqref{eq:ESNViewpoint2}, is the most commonly used choice in practical applications in reservoir computing, see for instance \cite{Jaeger04}, \cite{Pathak:PRL}. From a purely  mathematical point of view it could also be interesting to look for other choices of update functions $G$ so that for ${\bf X}_t = G({\bf X}_{t-1},{\bf Z}_t)$ a similar result to Theorem 2 can be proved. However, proving such a result would require different techniques than those used in the proof of Theorem 2 (which, due to its reliance on Corollary~\ref{cor:RandomNNApprox} via Proposition~\ref{prop:linearreservoir}, is specific to the neural network choice made here) and $G$ can not be chosen arbitrarily. For instance, if we choose $\bm{\sigma}(x)=x$ in \eqref{eq:ESNViewpoint2}, then ${\bf W} {\bf X}_0$ is a linear functional of ${\bf Z}$, which can not be used to approximate the (in general non-linear) functional $H^*$.
\end{remark}
}

\begin{remark}
{Let us be more specific about how echo state networks are used in applications. In many situations, the goal is to learn an unknown input/output system from data. For example, in \cite{Jaeger04}, \cite{Pathak:PRL} the considered task is to predict the evolution of chaotic dynamical systems based on observational data. In general such problems can be phrased using a target process  ${\bf Y} = ({\bf Y}_t)_{t \in \mathbb{Z}}$ and an observation process ${\bf Z} = ({\bf Z}_t)_{t \in \mathbb{Z}}$. The goal is to predict ${\bf Y}_{t}$ based on $({\bf Z}_{s})_{s \leq t}$. For instance, the target process is ${\bf Y}_{t} = H^*(({\bf Z}_{s})_{s \leq t})$ or ${\bf Y}_t = {\bf Z}_{t+h}$ for some $h>0$ (which corresponds to learning the functional $H^*(({\bf Z}_{s})_{s \leq t}) = \E[{\bf Z}_{t+h}|({\bf Z}_{s})_{s \leq t}]$). To achieve this goal, echo state networks as introduced in \eqref{eq:ESNViewpoint2} are used. First, the parameters ${\bf A}$, ${\bf C}$ and  $\bm{\zeta}$ are generated according to some given distribution (for instance, all entries are drawn from a normal distribution). Then the readout matrix ${\bf W}$ is trained by a linear regression using past data, i.e. by solving 
\[ {\bf W}^* = \arg \min_{{\bf W}} \frac{1}{T}\sum_{k=1}^T \|{\bf W} {\bf X}_{t-k} - {\bf Y}_{t-k} \|^2   \]
and ${\bf W}^* {\bf X}_t$ is then the prediction of ${\bf Y}_t$. This is in practice repeated for different random samples ${\bf A}$, ${\bf C}$, and  $\bm{\zeta}$ and an optimization over some hyperparameters is carried out. This procedure has been successful at learning input/output systems in a wide range of applications, in the sense that echo state networks have been able to achieve a low mean squared prediction error $\|{\bf W}^* {\bf X}_t-{\bf Y}_t\|$ in comparison to other methods. In view of Remark~\ref{rmk:leastsquares},  Theorem~\ref{thm:ESNErrorBound} directly provides  error bounds for this procedure in the case  ${\bf Y}_{t} = H^*(({\bf Z}_{s})_{s \leq t})$. In the case when $ {\bf Y}_t$ is a general random vector not necessarily measurable with respect to the sigma-algebra generated by $({\bf Z}_{s})_{s \leq t}$ (for instance, if ${\bf Y}_t = {\bf Z}_{t+h}$) then the \emph{approximation error bounds} in Theorem~\ref{thm:ESNErrorBound} can be combined with the \emph{generalization error bounds} in \cite{RC10} to obtain an error analysis for echo state network-based learning also in this case. 
}
\end{remark}

In order to use the bound in Theorem~\ref{thm:ESNErrorBound} in practice one can now  prescribe an approximation accuracy $\varepsilon >0$ and subsequently select the hyperparameters $R$, $T$, $N$ so that the right hand side of \eqref{eq:MSENN6} is smaller than $\varepsilon$. {The next result provides a special case of Theorem~\ref{thm:ESNErrorBound} when $H_T^*$ is in the Sobolev space $W^{k,2}(\R^q)$.}
{
	\begin{corollary}
		%\label{cor:sobolev2}
Let $\sigma \colon \R \to \R$ be given as $\sigma(x)=\max(x,0)$ and let $w \in (0,\infty)^{\Z_-}$ with $\sum_{j \in \Z_-} |j|w_j < \infty$. Let $T, N \in \N^+$,  let $q=d(T+1)$, let $k\in \N$ with $k \geq \frac{q}{2}+1 + \varepsilon$ for some $\varepsilon>0$ and let $R=N^{1/(2k-2\varepsilon+1)}$. Assume that  $\|({\bf Z}_0,\ldots,{\bf Z}_{-T})\|_{\R^{d(T+1)}}\leq M_T$ and generate ${\bf A}, {\bf C}, \bm{\zeta}$ according to the procedure described in (i)-(ii) in Theorem~\ref{thm:ESNErrorBound}. 
Then for any $H^* \colon (D_d)^{\Z_-} \to \R$ that is  $w$-Lipschitz continuous with Lipschitz constant $L$ and satisfies $H^*_T \in W^{k,2}(\R^q) \cap L^1(\R^q)$ 
there exists a readout ${\bf W}$ (a  $\M_{1,2(N+d(T+1))}$-valued random variable) such that the system \eqref{eq:ESNViewpoint2} satisfies the echo state property and
\begin{equation*} 
%\label{eq:MSENN8} 
\begin{aligned}
\E[|Y_0 - H^*({\bf Z})|^2]^{1/2}   & \leq C \| H^*_T \|_{k} N^{-1/\alpha}  + L M  \left(\sum_{i=T+1}^\infty w_{-i} \right)  \end{aligned}
\end{equation*}
with $\alpha = 2+\frac{(q+1)}{k-q/2-\varepsilon}$ and  
\begin{equation*}\begin{aligned}
C & = [16\max(M_T,1) \mathrm{Vol}_q(B_1) (M_T+1)^2 ([M_T]^3+M_T+2)]^{1/2} (2 \pi)^{-q} \\ & \qquad + \left(\int_{\R^q } (1+\|{\bf w} \|^2)^{-q/2-\varepsilon}  \d {\bf w} \right)^{1/2}. \end{aligned}
\end{equation*}
\end{corollary}
\begin{proof}
The corollary is a consequence of Theorem~\ref{thm:ESNErrorBound} and Corollary~\ref{cor:sobolev}. More specifically, to deduce the desired result from Theorem~\ref{thm:ESNErrorBound} it suffices to prove that the first two error terms in \eqref{eq:MSENN6} are bounded by $C \| H^*_T \|_{k} N^{-1/\alpha}$. To this end, note that these error terms arise when applying Corollary~\ref{cor:RandomNNApprox} in \eqref{eq:MSENN3}. Our hypotheses allow us to apply Corollary~\ref{cor:sobolev} instead of Corollary~\ref{cor:RandomNNApprox}, which directly yields the desired expression for the upper bound and the constant. 
%If  $I_{k-q/2-\varepsilon} = \int_{\R^q } \|{\bf w} \|^{k-q/2-\varepsilon} |g_T({\bf w})| \d {\bf w} < \infty$, then with our choice of $R$ the bound \eqref{eq:MSENN3} from Corollary~\ref{cor:RandomNNApprox} can be reduced to
% \[
% \frac{\sqrt{C_{T,R}}}{\sqrt{N}} +  \int_{\R^q \setminus B_R} |g_T({\bf u})| \d {\bf u} \leq \frac{1}{N^{[2+\frac{(q+1)}{k-q/2-\varepsilon}]^{-1}}} \left[c_{M,q} C_g^* + I_{k-q/2-\varepsilon} \right]
% \]
%by Corollary~\ref{cor:RandomNNApprox}(i) with $c_{M,q}$ and $ C_{g}^*$ as in Corollary~\ref{cor:RandomNNApprox} for $g:=g_T$. On the other hand, the hypothesis $H^*_T \in W^{k,2}(\R^q) \cap L^1(\R^q)$ 
\qed
\end{proof}
}

We now provide an example in which, for each $N$, good choices of the hyperparameters $T$ and $R$ can be given explicitly as a function of $N$ and thus also the bound  \eqref{eq:MSENN6} depends only on $N$. 

\begin{example} 
%\label{ex:Exponential}
Let $d=1$, $D_d=[-M,M]$, $\lambda \in (0,1)$ and consider the functional $H^*(z)=\exp(-\frac{1}{2}\sum_{i=0}^\infty \lambda^i (z_{-i})^2 )$.  Then $H^*$ satisfies the hypotheses of Theorem~\ref{thm:ESNErrorBound} and we may choose $R$, $T$ appropriately to obtain for any $N\in \N^+$ 
\[\E[|Y_0 - H^*({Z})|^2]^{1/2}  \leq \frac{p(N)}{N^\gamma} \] 
for some slowly growing function $p$ (a power of logarithms of $N$) and some $\gamma > 0$. We carefully prove this in the next Lemma.
\end{example}

\begin{lemma} Let $\beta>\alpha >0$ satisfy $1>\frac{\alpha}{2}(1-\log(2)+\log(\beta/\alpha))$. Then for any $N \in \N^+$ the ESN approximation constructed in Theorem~\ref{thm:ESNErrorBound} with
$T+1 = \alpha \log(\sqrt{N})$, $R=\beta \log(\sqrt{N})$, satisfies 
\[\E[|Y_0 - H^*({Z})|^2]^{1/2}  \leq \frac{p(N)}{N^\gamma} \] 
with $p \colon (0,\infty) \to \R$ and $\gamma>0$ given in \eqref{eq:auxEq95} and \eqref{eq:auxEq96}, respectively.
\end{lemma}
\begin{proof} Firstly, using that $f_e \colon [0,\infty) \to [0,\infty)$, $f_e(x)=\exp(-x/2)$ is $1/2$-Lipschitz, one estimates
\[
|H^*(u)-H^*(v)| \leq \frac{1}{2} |\sum_{i=0}^\infty \lambda^i [(u_{-i})^2 -(v_{-i})^2] | \leq M \sum_{i=0}^\infty \lambda^i |(u_{-i}) -(v_{-i}) |
\]
and so $H^*$ is $w$-Lipschitz for $w=(\lambda^k)_{k \in \N}$. 
Secondly, let $\Sigma = \mathrm{diag}(1,\lambda,\ldots,\lambda^T)$. Noting that $H^*_T$ is the characteristic function of a $\mathcal{N}(0,\Sigma)$-distributed random variable one has for any ${\bf u} =(z_0,\ldots,z_{-T})$  
\[ H^*_T({\bf u}) = \exp\left(-\frac{1}{2}\sum_{i=0}^T \lambda^i (z_{-i})^2 \right) = \int_{\R^{T+1}} e^{i \langle {\bf w}, {\bf u} \rangle } g_T({\bf w}) \d {\bf w}
\]
where $g_T$ is the density of a $\mathcal{N}(0,\Sigma)$-distribution. In particular, $g_T$ is integrable and \eqref{eq:BarronCondRd3} is satisfied.
Choosing $\rho=\sqrt{\lambda}$ in the shift matrix \eqref{eq:shiftMatrix} we note that ${\bf K} = \Sigma^{1/2}$ is invertible. By Theorem~\ref{thm:ESNErrorBound} and Remark~\ref{rmk:Scmatrices} it follows that the approximation bound
\eqref{eq:MSENN5} holds with $C_{T,R}$ given in \eqref{eq:CTRdef}. The last term in the bound \eqref{eq:MSENN5} is $0$, since ${\bf S}^{T+1} = 0$. For our choice $T+1=\alpha \log(\sqrt{N})$ the second to last term in the bound \eqref{eq:MSENN5} equals 
\begin{equation}\label{eq:auxEq93} LM\sum_{i=T+1}^\infty w_{-i} = \lambda^{T+1} LM/(1-\lambda) = \frac{1}{\sqrt{N}^{\alpha \log(1/\lambda)}}LM/(1-\lambda).\end{equation}
 Denoting by ${\bf V}$ a $\mathcal{N}(0,\boldsymbol{I}_{T+1})$-distributed random variable, the second term in the right hand side of \eqref{eq:MSENN5} can be written as
\begin{equation*}
%\label{eq:auxEq91}
 |\det({\bf K})| \int_{\R^{T+1} \setminus B_R} |g_T({\bf K}^{\top} {\bf w})| \d {\bf w} = (2\pi)^{-(T+1)/2} \int_{\R^{T+1} \setminus B_R} e^{-\frac{\|{\bf w}\|^2}{2}} \d {\bf w} = \P(\|{\bf V}\|>R).
\end{equation*}
Recall that $\|{\bf V}\|^2$ has a chi-square distribution with $T+1$ degrees of freedom. Using this and the fact that $R^2>T+1$ (because $\beta>\alpha$) one estimates 
\begin{equation}\label{eq:auxEq92} \begin{aligned} \P(\|{\bf V}\|>R) & \leq \P(\|{\bf V}\|^2 >R^2) \leq \left(\frac{R^2}{T+1} e^{1-R^2/(T+1)}\right)^{(T+1)/2} \\ & = \frac{1}{\sqrt{N}^{\beta/2 - \alpha/2 - \alpha \log(\beta/\alpha)/2}}. \end{aligned}  \end{equation}
Finally, one calculates
\begin{equation} \label{eq:auxEq102} \begin{aligned}
|\det({\bf K})|^2 \int_{B_R} \max(1,\|{\bf w}\|^3) |g_T({\bf K}^{\top} {\bf w})|^2 \d {\bf w} & \leq R^3 (2\pi)^{-(T+1)} \int_{\R^{T+1}} e^{-\|{\bf w}\|^2} \d {\bf w} \\ & = R^3 (2\pi)^{-(T+1)/2} 2^{-(T+1)/2}.
\end{aligned} \end{equation}
Recall the following standard estimate for the volume of the ball $B_R \subset \R^q$:
\begin{equation}\label{eq:VolUnitBall} \mathrm{Vol}_{q}(B_R) \leq \frac{1}{\sqrt{q\pi}} \left[ \frac{2 \pi e}{q}\right]^{q/2}R^q. \end{equation}
Inserting \eqref{eq:auxEq102}, $M_T \leq \sqrt{(T+1)}M$ and \eqref{eq:VolUnitBall} in \eqref{eq:CTRdef} yields (for $M_T > 1$, $R>1$)
 \begin{equation}\label{eq:auxEq94} \begin{aligned}
 C_{T,R} & \leq \frac{2^8}{\pi} M^7 (T+1)^{3} R^4 \left[\frac{ e R^2}{2(T+1)}\right]^{(T+1)/2}.  \end{aligned}
 \end{equation}
We may now put together all the terms that we estimated separately: inserting \eqref{eq:auxEq93}, \eqref{eq:auxEq92} and \eqref{eq:auxEq94} in the approximation bound \eqref{eq:MSENN5} yields
\[\E[|Y_0 - H^*({Z})|^2]^{1/2}  \leq \frac{p(N)}{\sqrt{N}}, \] 
where
\begin{equation}\label{eq:auxEq95} \gamma = \frac{1}{2} \min\{\alpha \log(\lambda^{-1}),\frac{\beta}{2} - \frac{\alpha}{2}(1+ \log(\beta/\alpha)),1-\frac{\alpha}{2}(1-\log(2)+\log(\beta/\alpha))\}, \end{equation}
\begin{equation}\label{eq:auxEq96}
p(N) = \frac{2^8}{\pi} M^7 \alpha^3 \beta^4 (\log(\sqrt{N}))^{7} +1+\frac{LM}{1-\lambda}.
\end{equation}
Note that the second term in \eqref{eq:auxEq95} is positive, since $1+\log(x)\leq x$ for $x > 0$ and since $\alpha < \beta$. The last term in \eqref{eq:auxEq95} is positive by assumption on $\alpha, \beta$ and so indeed $\gamma>0$.\qed
\end{proof}

\subsection{Universal approximation by echo state networks}
\label{sec:dynUniversality}

As a corollary of the echo state network approximation error bounds in Theorem~\ref{thm:ESNErrorBound}, we also obtain a constructive ESN universality result, see Corollary~\ref{cor:universality} below. This complements the ESN universality result in \cite[Theorem~III.10]{RC8}. The key novelty of Corollary~\ref{cor:universality} is that a constructive approximation procedure (up to tuning the hyperparameters $N,T,R$ and carrying out a regression to estimate ${\bf W}$) is given, whereas \cite[Theorem~III.10]{RC8} is an existence result. Note also that the setting is slightly different (the activation function here is ReLU and the inputs are uniformly bounded). 

\begin{corollary}\label{cor:universality} Let $H^* \colon (D_d)^{\Z_-} \to \R$ measurable satisfy that $\mathbb{E}[|H^*(\mathbf{Z})|^2] < \infty$ and let $\sigma \colon \R \to \R$ be given as $\sigma(x)=\max(x,0)$. Then for any $\varepsilon > 0$, $\delta \in (0,1)$ there exists $N, T \in \mathbb{N}^+$, $R>0$ and a readout ${\bf W}$ {(a  $\M_{1,2(N+d(T+1))}$-valued random variable)} such that the system \eqref{eq:ESNViewpoint2} (with ${\bf A}, {\bf C}, \bm{\zeta}$ generated according to (i)-(ii) in Theorem~\ref{thm:ESNErrorBound} for $M_T = M\sqrt{T}$) satisfies the echo state property and (denoting by $H_{\bf W}^{{\bf A},{\bf C},\bm{\zeta}}$ the associated random ESN functional) the approximation error satisfies with probability $1-\delta$ that
\begin{equation*}\begin{aligned} \left(\int_{(\R^d)^{\Z_-}} |H_{\bf W}^{{\bf A},{\bf C},\bm{\zeta}}({\bf z}) - H^*({\bf z})|^2 \mu_{{\bf Z}}(\d {\bf z}) \right)^{1/2} & = \E[ |H_{\bf W}^{{\bf A},{\bf C},\bm{\zeta}}({\bf Z})   - H^*({\bf Z})|^2|{\bf A}, {\bf C}, \bm{\zeta}]^{\frac{1}{2}} \\ & < \varepsilon. \end{aligned} \end{equation*} 
\end{corollary}

\begin{proof} Firstly, by standard properties of the conditional expectation (see for instance \cite[Lemma~A.1]{RC8}) we may find $T^* \in \N^+$ satisfying 
\begin{equation}
\label{eq:auxEq97}
\E[|H^*({\bf Z})-\E[H^*({\bf Z})|\Fc_{-T^*}]|^2]^{1/2} < \frac{\varepsilon \sqrt{\delta}}{3},
\end{equation}
where $\Fc_{-T^*} := \sigma({\bf Z}_0,\ldots,{\bf Z}_{-T^*})$.
Let $q:=d(T^*+1)$. By definition, $\E[H^*({\bf Z})|\Fc_{-T^*}]$ is $\Fc_{-T^*}$-measurable and so there exists a measurable function $H^{(1)}\colon \R^q \to \R$ such that $\E[H^*({\bf Z})|\Fc_{-T^*}] = H^{(1)}({\bf Z}_0,\ldots,{\bf Z}_{-T^*})$ and $\E[|H^{(1)}({\bf Z}_0,\ldots,{\bf Z}_{-T^*})|^2]< \infty$ (see, e.g., \cite[Lemma~1.13]{Kallenberg2002}). By combining \cite[Lemma~1.33]{Kallenberg2002} and the fact that $C_c^\infty(\R^q)$ is dense in $C_c(\R^q)$ in the supremum norm we find $H^{(2)} \in C_c^\infty(\R^q)$ satisfying \begin{equation}\label{eq:auxEq98}
\E[|H^{(1)}({\bf Z}_0,\ldots,{\bf Z}_{-T^*})-H^{(2)}({\bf Z}_0,\ldots,{\bf Z}_{-T^*})|^2]^{1/2}< \frac{\varepsilon \sqrt{\delta}}{3}.
\end{equation}
We claim that $H^{(2)}$ satisfies Assumption~\ref{ass:regularityFctl}. Indeed, $H^{(2)}$ is Lipschitz continuous on $\R^q$ and thus also $w$-Lipschitz with $w=(\mathbbm{1}_{\{t\leq T^*\}})_{t \in \N}$. In addition, for any $T \in \N^+$ one has that $H^{(2)}_T$ is a Schwartz function and so also the Fourier transform of $H^{(2)}_T$ is a Schwartz function and the Fourier inversion theorem with \eqref{eq:BarronCondRd3} indeed hold. Now set $T=T^*+1$ and choose $R$ so that the second to last term in the right hand side of \eqref{eq:MSENN5} is smaller than $\varepsilon\sqrt{\delta}/6$ and then choose $N$ such that $\frac{\sqrt{C_{T,R}}}{\sqrt{N}} < \varepsilon\sqrt{\delta}/6$.
Applying Theorem~\ref{thm:ESNErrorBound} then yields 
 \begin{equation} \label{eq:auxEq99}
 \E[|Y_0 - H^{(2)}({\bf Z}_0,\ldots,{\bf Z}_{-T^*})|^2]^{1/2}   < \frac{\varepsilon\sqrt{\delta}}{3}.
 \end{equation}
 Applying the triangle inequality and using \eqref{eq:auxEq97}, \eqref{eq:auxEq98}, \eqref{eq:auxEq99} we then obtain
  \begin{equation*} %\label{eq:auxEq100}
  \E[|Y_0 - H^*({\bf Z})|^2]^{1/2}   < \varepsilon\sqrt{\delta}.
  \end{equation*}  
Thus, Markov's inequality gives
\[ \begin{aligned}
\P & \left( \left(\int_{(\R^d)^{\Z_-}} |H_{\bf W}^{{\bf A},{\bf C},\bm{\zeta}}({\bf z}) - H^*({\bf z})|^2 \mu_{{\bf Z}}(\d {\bf z}) \right)^{1/2} > \varepsilon \right) \\ &  \leq \frac{1}{\varepsilon^2} \E[\E[|Y_0 - H^*({\bf Z})|^2|{\bf A}, {\bf C}, \bm{\zeta}]]  < \delta,
\end{aligned}\]
as claimed.\qed
\end{proof}

\section{Approximation Error Estimates For Echo State Networks with Output Feedback}\label{sec:ESNOutput}

In this section we continue our study of the dynamic situation, but we now focus on approximations based on a slightly different type of reservoir computing systems: echo state networks with output feedback, that is, systems given for  ${\bf z} \in (D_d)^{\Z_-}$ and $t \in \Z_-$ by 
\begin{equation} \begin{aligned} \label{eq:ESNrandCoeff2}
\mathbf{x}_t &  = \bm{\sigma}( {\bf A} \mathbf{y}_{t-1} + {\bf C} {\bf z}_t + \bm{\zeta}), \\
{\bf y}_t & = {\bf W} \mathbf{x}_{t}.
\end{aligned}
\end{equation} 
These systems are a popular modification of the echo state networks considered in Section~\ref{sec:ESN}. They are also referred to as Jordan recurrent neural networks (with random internal weights) and are widely used in the literature.

 The advantage of these systems is that they can be used to directly approximate the reservoir function in case the functional $H^* \colon (D_d)^{\Z_-} \to \R^m$ is itself induced by a reservoir system. More precisely, consider $H^*$ defined via $ (D_d)^{\Z_-} \ni {\bf z} \mapsto H^*({\bf z}) = {\bf y}_0^* \in \R^m$ with ${\bf y}^*$ determined by
 \begin{equation}\label{eq:Hstarsystem2}
 \left\{
 \begin{aligned}
 \mathbf{x}_t^* & = F^*(\mathbf{x}_{t-1}^*,{\bf z}_t), \\ 
 {\bf y}_t^* & = h^*(\mathbf{x}_t^*), \quad t \in \Z_-.
 \end{aligned}
 \right.
 \end{equation}
 For functionals $H^*$ of this type the system \eqref{eq:ESNrandCoeff2} can be used to directly approximate the state updating function, that is, the function $F^*$ in \eqref{eq:Hstarsystem2}.
  The disadvantage of the system \eqref{eq:ESNrandCoeff2} is that the training procedure is more involved, since the readout ${\bf W}$ is fed back into the state equation of the echo state network in \eqref{eq:ESNrandCoeff2}. 
Nevertheless, these systems are used frequently in reservoir computing applications and so we also provide a detailed approximation analysis here. 

This section is structured as follows. In Theorem~\ref{thm:approx} in Section~\ref{sec:OFResult} we present our approximation result for functionals induced by sufficiently regular reservoir systems. Remarkably, in this case only one hyperparameter $N$ appears (proportional to the number of neurons, i.e. the dimension of $\mathbf{x}$ in \eqref{eq:ESNrandCoeff2}) and the approximation error is of order  $O(1/\sqrt{N})$. Theorem~\ref{thm:approx} follows from our more general approximation result Theorem~\ref{thm:approx2} below and Proposition~\ref{prop:SmoothRepresentation}. Beforehand we introduce the setting and regularity assumptions in Section~\ref{sec:OFsetting}.

\subsection{Setting and regular reservoir functionals} \label{sec:OFsetting}
As in Section~\ref{sec:ESN} we study systems \eqref{eq:ESNrandCoeff2} in which first ${\bf A},{\bf C},\bm{\zeta}$ are generated randomly (and then considered fixed) and subsequently ${\bf W}$ is trained in order to approximate $H^*$ as well as possible. We now specify the involved objects in more detail. Firstly, note that in practice instead of the infinite history system \eqref{eq:ESNrandCoeff2} in fact one always uses a system that satisfies \eqref{eq:ESNrandCoeff2} for $t \geq -T$ and is initialized at $t=-T-1$ with ${\bf y}_{-T-1} = \Xi$ for some $T \in \N^+$ and some $\Xi \in \R^m$ satisfying $\|\Xi\|\leq M$. Thus, these are also the systems we consider here. Next, throughout this section ${\bf Z}$ is a $(D_d)^{\Z_-}$-valued random variable -- a discrete-time stochastic process -- independent of  ${\bf A},{\bf C},\bm{\zeta}$. As in the previous sections the approximation error is measured conditional on the randomly generated parameters  ${\bf A},{\bf C},\bm{\zeta}$. However, in order to provide an alternative viewpoint we formulate the approximation results in this section in terms of statistical risk. Thus, for some integrable random variable ${\bf Y}_0$ we consider the risk defined by $\mathcal{R}(H):= \mathbb{E}[L(H({\bf Z}),{\bf Y}_0)]$ for a loss function $L \colon \mathbb{R}^m \times \mathbb{R}^m \to [0,\infty)$ satisfying the Lipschitz condition  \begin{equation}\label{eq:lossLipschitz} |L(\mathbf{x},{\bf y}) - L(\overline{\mathbf{x}},\overline{{\bf y}}) | \leq L_L (\|\mathbf{x}-\overline{\mathbf{x}} \|_2 +  \|\mathbf{y}-\overline{\mathbf{y}} \|_2), \enspace \mathbf{x}, \overline{\mathbf{x}}, {\bf y}, \overline{\mathbf{y}}\in \mathbb{R}^m. \end{equation}

In order to state our approximation result for echo state networks with output feedback let us now make precise which kinds of functionals we aim to approximate. Let $N^* \in \N^+$. We consider functions $f \colon \R^{N^*}\times D_d \to \R$ whose restriction to $B_{M+1}\times D_d $ satisfies the following smoothness condition: 
\begin{definition} \label{def:smooth} A function $f \colon \R^{N^*}\times D_d \to \R$ is \textit{sufficiently smooth}, if for $(\mathbf{x}, {\bf z}) \in B_{M+1}\times D_d $ one has $f(\mathbf{x}, {\bf z}) = \int_{\R^{N^*+d}} \hat{f}({\bf w}) e^{i (\mathbf{x}, {\bf z}) \cdot {\bf w}} \d {\bf w} $ where $\hat{f} \colon \R^{N^*+d} \to \C$ is a function satisfying
\begin{equation} \label{eq:Cfdef} C_f = \left(\mathrm{Vol}_{N^*+d}(B_1)\int_{\R^{N^*+d}} \max(1,\|{\bf w}\|^{2(N^*+d+3)}) |\widehat{f}({\bf w})|^2 \d {\bf w} \right)^{1/2} < \infty.\end{equation}
\end{definition} 
\begin{remark}
For instance, if $D_d = \R^d$, $f \in L^1(\R^{N^*+d}) \cap L^2(\R^{N^*+d})$, $\hat{f}$ denotes the Fourier transform of $f$ and $\hat{f}$ is integrable, then condition \eqref{eq:Cfdef} is equivalent to the requirement that $f$ belongs to the Sobolev space $W^{N^*+d+3,2}(\R^{N^*+d})$, see e.g.\ \cite[Theorem~6.1]{Folland1995}.
\end{remark}
\begin{remark} In this section we consider the dimensions $d$ and $N^*$ as fixed. The behaviour of  \eqref{eq:Cfdef} as a function of  $N^*+d$ depends on the function $f$ (or rather the family of functions indexed by $N^*+d$) under consideration. Recalling the estimate for the volume of the unit ball \eqref{eq:VolUnitBall} one observes that the factor $\mathrm{Vol}_{N^*+d}(B_1)$ in \eqref{eq:Cfdef} decreases to $0$ exponentially as $N^*+d \to \infty$.
\end{remark}
With this definition at hand, we now state the regularity assumption imposed on the functionals under consideration. Note that we focus on approximating the state equation here and so we set $m=N^*$ and take $h^*$ the identity in \eqref{eq:Hstarsystem2}. To approximate systems with general $h^*$ one may either combine the results presented here with any static approximation technique or proceed as explained in Remark~\ref{rmk:generalReadout} below. Note that under Assumption~\ref{ass:regularityFctl2} the system \eqref{eq:Hstarsystem2} satisfies the echo state property, see Proposition~\ref{prop:contractionEchoState}. 

\begin{assumption}\label{ass:regularityFctl2}
 Suppose $H^* \colon (D_d)^{\Z_-} \to \R^m$ satisfies that  $H^*({\bf z}) = {\bf x}_0^*$, where $\mathbf{x}^*$ satisfies \eqref{eq:Hstarsystem2}
for some continuous function $F^* \colon \R^{N^*} \times D_d \to B_M \subset \R^{N^*}$ such that
\begin{itemize}
\item for each ${\bf z} \in D_d$,  $F^*(\cdot,{\bf z})$ is an $r$-contraction,
\item for each $j=1,\ldots,N^*$, $F^*_j$ is sufficiently smooth (see Definition~\ref{def:smooth}). 
\end{itemize}
We denote $C_{H^*} = \sum_{j=1}^{N^*} C_{{F^*_j}}$ (with $C_{{F^*_j}}$ as in \eqref{eq:Cfdef}). 
\end{assumption}

\subsection{Approximation results for Echo State Networks with Output Feedback}\label{sec:OFResult}
We now derive bounds on the error arising when echo state networks with output feedback (see \eqref{eq:ESNrandCoeff2}) are employed to approximate functionals induced by sufficiently regular reservoir systems, that is, functionals satisfying Assumption~\ref{ass:regularityFctl2}. 
 When all parameters are trainable the networks \eqref{eq:ESNrandCoeff2} are also called Jordan networks. Here we consider an echo state network \eqref{eq:ESNrandCoeff2} with ${\bf A},{\bf C},\bm{\zeta}$ generated randomly from a generic distribution. The following theorem shows that such echo state networks with ReLU activation function and randomly generated parameters exhibit rather strong universal approximation properties: the same family of systems can be used to approximate any  functional satisfying a mild smoothness condition (expressed in terms of the Fourier transform as in \cite{Barron1993,Klusowski2018}) and the approximation error is of order  $O(1/\sqrt{N})$. In particular, only ${\bf W}$ needs to be tuned.

\begin{theorem}\label{thm:approx} Let $N \in \N^+$ and denote $\bar{N}=N N^*$.  Suppose $\sigma \colon \R \to \R$ is given as $\sigma(x)=\max(x,0)$, the rows of $[{\bf A},{\bf C}]$ are i.i.d.\ random variables distributed uniformly on $B_1 \subset \R^{N^*+d}$ and the entries of $\bm{\zeta}$ are i.i.d.\ random variables distributed uniformly on $[-M-1,M+1]$. Assume that $D_d \subset B_{M+1}$. Then for any functional $H^*$ satisfying Assumption~\ref{ass:regularityFctl2} there exists a readout ${\bf W}$ (a  $\M_{m,\bar{N}}$-valued random variable) such that for any $\delta \in (0,1)$, with probability $\max(1-\delta-\frac{4 C_*(M+1)}{\sqrt{N}},0)$ the system \eqref{eq:ESNrandCoeff2} initialized at $t=-T-1$ from any $\Xi \in \R^m$ with $\|\Xi\|\leq M$  satisfies the echo state property and the associated functional $H_{\bf W}^{{\bf A},{\bf C},\bm{\zeta}} \colon (D_d)^{\Z_-} \to \R^{m}$ satisfies
\begin{equation} \label{eq:ApproxErrorBound} |\mathcal{R}(H_{\bf W}^{{\bf A},{\bf C},\bm{\zeta}}) - \mathcal{R}(H^*)| \leq \frac{L_L}{\delta} \left[ \frac{2(M+1)C_*}{(1-r)\sqrt{N}} + 2 (M+1) r^{T+1} \right],   \end{equation}
where
\begin{equation} \label{eq:cBdef}
C_* = 16 \sqrt{3((M+1)^3+M+3)(M+1)} C_{H^*}.
\end{equation}
\end{theorem}
Theorem~\ref{thm:approx} follows from combining the representation in Proposition~\ref{prop:SmoothRepresentation} with our general reservoir approximation result, Theorem~\ref{thm:approx2} below. Note that in Theorem~~\ref{thm:approx2} below also the boundedness assumption $D_d \subset B_{M+1}$ is not required.

\begin{proof}[Proof of Theorem~\ref{thm:approx}] Firstly, for any $j=1,\ldots,N^*$ the function $F^*_j$ satisfies the hypotheses of Proposition~\ref{prop:SmoothRepresentation}. Therefore, there exists an integrable function $\pi_j^* \colon \R^{N^*+d+1} \to \R$ such that for ${\bf x} \in B_{M+1}$, ${\bf z} \in D_d$, the function $F^*_j$ can be represented as 
\[ F^*_j(\mathbf{x},{\bf z}) = \int_{\R^{N^*+d+1}} \sigma((\mathbf{x},{\bf z},1) \cdot \bm{\omega}) \pi_j^*(\bm{\omega}) \d \bm{\omega},  \]
 and $\pi_j^*(\bm{\omega}) = 0$ for all $\bm{\omega} = ({\bf w},u) \in \R^{N^*+d} \times \R$  satisfying $\|{\bf w}\|>1$ or $|u|>M+1$, and
\begin{equation}\label{eq:auxEq67} \int_{\R^{N^*+d+1}} \|\bm{\omega}\|^2 \pi_j^*(\bm{\omega})^2 \d \bm{\omega} \leq 8 ((M+1)^3+M+3) C_{F_j^*}.
\end{equation}
Recall that the entries of $\bm{\zeta}$ are uniformly distributed on $[-(M+1),M+1]$.
Setting $\pi^k_j (\d \bm{\omega})= \pi_j^*(\bm{\omega})\d \bm{\omega}$ for all $k \in \N^+$, denoting by $\pi_1$ and $\pi_2$ the uniform distribution on $B_1  \subset \R^{N^*+d}$  and $[-(M+1),(M+1)]$, respectively, and setting $\pi = \pi_1 \otimes \pi_2$, one has that $\pi^k_j \ll \pi$ and   
$\frac{\d \pi_j^k}{\d \pi} = 2 \mathrm{Vol}_{N^*+d}(B_1) (M+1)\pi_j^*$. Using \eqref{eq:auxEq67} one therefore obtains 
\[\begin{aligned}  &  4 \sqrt{3} \sum_{j=1}^{N^*} \left(\int_{\R^{N^*+d+1}} \|\bm{\omega}\|^2 \left(\frac{\d \pi_j^k}{\d \pi}(\bm{\omega}) \right)^2 \pi(\d\bm{\omega})\right)^{1/2} \\
&  =  4 \sqrt{6(M+1)\mathrm{Vol}_{N^*+d}(B_1)} \sum_{j=1}^{N^*} \left(\int_{\R^{N^*+d+1}} \|\bm{\omega}\|^2 \pi_j^*(\bm{\omega})^2 \d \bm{\omega} \right)^{1/2}
\\ & \leq 16 \sqrt{3((M+1)^3+M+3)(M+1)} C_{H^*}
\end{aligned} \]
and so the constant $C_*$ in \eqref{eq:cBdef} is larger or equal than the constant $C_*$ in Theorem~\ref{thm:approx2} below. Furthermore, $s_k = 0$ for all $k \in \N^+$ and thus the statement follows from Theorem~\ref{thm:approx2} below.\qed
\end{proof}

\begin{remark}\label{rmk:generalReadout} As pointed out above, here we focus on systems \eqref{eq:Hstarsystem2} in which $h^*$ is the identity. However, Theorem~\ref{thm:approx} could also be extended to more general $h^*$, namely those satisfying that $F_j^* = h_{j-N^*} \circ F^*$ is sufficiently smooth (see Definition~\ref{def:smooth}) for $j=N^*+1,\ldots,N^*+m$.  The matrix ${\bf A}$ in \eqref{eq:ESNrandCoeff2} would then be replaced by ${\bf A} P$ with 
$
P= \left(
\begin{array}{cc}
 \boldsymbol{I}_{N^*}&\boldsymbol{0}_{N^*,m}
\end{array}
\right)$.
\end{remark}
Finally, we prove a more general echo state network approximation result valid for functionals induced by reservoir systems with reservoir function $F^*$ that can be approximated well by functions of the form \eqref{eq:FstarApprox}.

\begin{theorem}\label{thm:approx2} Let $r \in (0,1), L_\sigma >0$, $N \in \N^+$ and denote $\bar{N}=N N^*$.  Suppose $\sigma \colon \R \to \R$ is $L_\sigma$-Lipschitz continuous and the rows of $[{\bf A},{\bf C},\bm{\zeta}]$ are i.i.d.\ random variables with distribution $\pi$. Suppose
 $H^* \colon (D_d)^{\Z_-} \to \R^m$ is the reservoir functional associated to some $F^* \colon \R^{N^*} \times D_d \to B_M \subset \R^{N^*}$, i.e., for any ${\bf z} \in (D_d)^{\Z_-}$ it is given as  $H^*({\bf z}) = {\bf x}_0^*$, where ${\bf x}^*$ satisfies \eqref{eq:Hstarsystem2}. Assume that for each ${\bf v} \in D_d$,  $F^*(\cdot,{\bf v})$ is an $r$-contraction. Furthermore, for any $k \in \N^+$, $j=1,\ldots,N^*$,  let $\pi^k_j$ be a signed Borel-measure on $\R^{N^*+d+1}$ such that $\pi^k_j \ll \pi$, $\int_{\R^{N^*+d+1}} \|\bm{\omega}\|  |\pi_j^k|(\d \bm{\omega})< \infty$ and 
 \[ C_* = 4 \sqrt{3} L_\sigma \sup_{k \in \N^+} \sum_{j=1}^{N^*} \left(\int_{\R^{N^*+d+1}} \|\bm{\omega}\|^2 \left(\frac{\d \pi_j^k}{\d \pi}(\bm{\omega})\right)^2  \pi(\d\bm{\omega})\right)^{1/2} < \infty. \]
Denote for each  $j=1,\ldots,N^*$
 \begin{equation}\label{eq:FstarApprox}
  F^{*,N}_j(\mathbf{x},{\bf v}) = \int_{\R^{N^*+d+1}}  \sigma((\mathbf{x},{\bf v},1) \cdot \bm{\omega})  \pi_j^N(\d \bm{\omega}), \quad \mathbf{x} \in \R^{N^*}, {\bf v} \in D_d
 \end{equation}
and assume $s_N = \E[\max_{t \in \{0, \ldots,-T\}} \sup_{\mathbf{x} \in B_{M+1}} \|F^{*,N}(\mathbf{x},{\bf Z}_t)- F^*(\mathbf{x},{\bf Z}_t)\|] < 1$. Assume that $\E\left[\max_{t \in \{0, \ldots,-T\}} \|{\bf Z}_t\|\right] < \infty$.
 Then there exists a readout ${\bf W}$ (a  $\M_{m,\bar{N}}$-valued random variable) such that for any $\delta \in (0,1)$, with probability at least $\max(1-\delta-\frac{2C_*(M+2+\E\left[\max_{t \in \{0, \ldots,-T\}} \|{\bf Z}_t\|\right])}{\sqrt{N}}-2s_N,0)$
  the system \eqref{eq:ESNrandCoeff2} initialized at $t=-T-1$ from any $\Xi \in \R^m$ with $\|\Xi\|\leq M$  satisfies the echo state property and the associated functional $H_{\bf W}^{{\bf A},{\bf C},\bm{\zeta}} \colon (D_d)^{\Z_-} \to \R^{m}$ satisfies
\[\begin{aligned}   &|\mathcal{R}(H_{\bf W}^{{\bf A},{\bf C},\bm{\zeta}}) - \mathcal{R}(H^*)|  \\ & \leq \frac{L_L}{\delta} \left[ \frac{(M+2+\max_{t \in \{0, \ldots,-T\}}\E[\|{\bf Z}_t \|])C_*}{(1-r)\sqrt{N}} + \frac{s_N}{1-r} + 2 (M+1) r^{T+1} \right]. \end{aligned} \]
\end{theorem}

\begin{remark} Let us discuss the assumption $\E\left[\max_{t \in \{0, \ldots,-T\}} \|{\bf Z}_t\|\right] < \infty$. Firstly, suppose that the input signal satisfies $\|{\bf Z}_t \| \leq B$, $\P$-a.s. for all $t \in \Z_-$. Then clearly also $\E\left[\max_{t \in \{0, \ldots,-T\}} \|{\bf Z}_t\|\right] \leq B$ and so one may initialize the system at any $T \geq \frac{\log(\sqrt{N})}{-\log(r)}-1$ in order to achieve an approximation error bound \eqref{eq:ApproxErrorBound} of order $\frac{1}{\sqrt{N}}$ with high probability $1-O(\frac{1}{\sqrt{N}})$. 
However, our result also covers more general situations. For instance, suppose that $d=1$ and for each $t \in \Z_-$, $Z_t$ is standard normally distributed (not necessarily independent). Then one can show that 
\[ \E\left[\max_{t \in \{0, \ldots,-T\}} \|{\bf Z}_t\|\right] \leq \sqrt{2\log(2T)}, \]
and consequently, choosing $T$ as in the first case, one obtains an error bound of order $\frac{1}{\sqrt{N}}$ with high probability $1-O(\frac{\sqrt{\log(\log(N))}}{\sqrt{N}})$. 
\end{remark}

\begin{proof}[Proof of Theorem~\ref{thm:approx2}]
Recall that $N^* = m$, $h^*({\bf y}) = {\bf y}$ and let us write ${\bf A}$, ${\bf C}$, $\bm{\zeta}$ as block matrices
 \[  {\bf A} = \begin{pmatrix}
    {\bf A}^{(1)} \\
     \vdots    \\
     {\bf A}^{(N)}
   \end{pmatrix} \in \R^{N N^* \times N^*},  {\bf C} =  \begin{pmatrix}
      {\bf C}^{(1)} \\
       \vdots    \\
       {\bf C}^{(N)}
     \end{pmatrix} \in \M_{N N^*,d}, \text{ and }  \bm{\zeta}=  \begin{pmatrix}
   \bm{\zeta}^{(1)} \\
    \vdots    \\
    \bm{\zeta}^{(N)}
  \end{pmatrix} \in \R^{N N^*}, \]
 where ${\bf A}^{(i)}$, ${\bf C}^{(i)}$ and $\bm{\zeta}^{(i)}$ are random matrices (resp. vectors) valued in  $\M_{N^*,N^*}$, $\M_{N^*,d}$ and $\R^{N^*}$, respectively, for each $i=1,\ldots,N$. Define the readout 
\begin{equation*} %\label{eq:Wdef}
{\bf W} = \frac{1}{N} \begin{pmatrix}
  {\bf W}_1 &
   \cdots    &
   {\bf W}_N
 \end{pmatrix}, \quad {\bf W}_i = \begin{pmatrix}
   V_1^{(i)} & & \\
    & \ddots    & \\
    & & V_{N^*}^{(i)}
  \end{pmatrix},
\end{equation*}
where  ${\bf U}_j^{(i)}=({\bf A}_j^{(i)},
{\bf C}_j^{(i)},\zeta_j^{(i)})$ denotes the $j$-th row of $({\bf A}^{(i)},{\bf C}^{(i)},\zeta^{(i)})$ and $V_j^{(i)}$ is given as
\[V_j^{(i)} = \frac{\d \pi_j^N}{\d \pi }({\bf U}_j^{(i)}). \]
By our choice of $V_j^{(i)}$ one calculates for each $i=1,\ldots,N$ and any ${\bf y} \in B_{M+1}$, ${\bf z} \in D_d$, 
\begin{equation}\label{eq:auxEq58}\begin{aligned} \E[({\bf W}_i \sigma({\bf A}^{(i)} {\bf y} + {\bf C}^{(i)} {\bf z} + \bm{\zeta}^{(i)}))_j] & = \E[V_j^{(i)} \sigma(( {\bf y},{\bf z},1)\cdot {\bf U}_{j}^{(i)})] 
\\ & 
= \int_{\R^{N^*+d+1}} \sigma(( {\bf y},{\bf z},1)\cdot \bm{\omega}) \frac{\d \pi_j^N}{\d \pi}(\bm{\omega})  \pi(\d\bm{\omega})  \\ 
& = F_j^{*,N}({\bf y},{\bf z})\end{aligned} \end{equation}
and 
\[ \begin{aligned}
\E[(V_j^{(1)})^2 (\| {\bf A}_j^{(1)}\|^2+ \|{\bf C}_j^{(1)}\|^2 + |\zeta_j^{(1)}|^2)] &  =  \int_{\R^{N^*+d+1}} \|\bm{\omega}\|^2 \left( \frac{\d \pi_j^N}{\d \pi}(\bm{\omega})\right)^2   \pi(\d\bm{\omega}).
\end{aligned} \]
This shows that 
\begin{equation}\label{eq:auxEq59} 4 \sqrt{3} L_\sigma \sum_{j=1}^{N^*} \E[(V_j^{(1)})^2 (\| {\bf A}_j^{(1)}\|^2+ \|{\bf C}_j^{(1)}\|^2 + |\zeta_j^{(1)}|^2)]^{1/2} \leq C_*. \end{equation}

\textit{Measurability and echo state property:}

\noindent Consider
\begin{equation}\label{eq:OmegaESPDef4}\begin{aligned} \Omega_{ESP} & = \left\lbrace \omega \in \Omega \colon \bar{M}(\omega) \leq M +1 \right\rbrace, \quad 
 \bar{M}  = \sup_{\substack{\mathbf{x} \in B_{M+1} \\ t\in \{0,\ldots,-T\}}} \|{\bf W} F^{{\bf A},{\bf C},\bm{\zeta}}(\mathbf{x},{\bf Z}_t) \|. \end{aligned}
\end{equation}
By continuity the supremum in \eqref{eq:OmegaESPDef4} is finite and can also be taken over a countable set. This shows that $\Omega_{ESP} \in \Fc$. Furthermore, consider the system \eqref{eq:ESNrandCoeff2} initialized at $t=-T-1$ from a given $\Xi \in \R^m$ with $\|\Xi\|\leq M$. Clearly, for any ${\bf z} \in (D_d)^{\Z_-}$ there is a unique $({\bf y}_t)_{t=0,\ldots,-T}$ satisfying \eqref{eq:ESNrandCoeff2} and for any $\omega \in \Omega$ the function $H_{\bf W(\omega)}^{{\bf A}(\omega),{\bf C}(\omega),\bm{\zeta}(\omega)} \colon (D_d)^{\Z_-} \to \R^{N^*}$ mapping ${\bf z} \in (D_d)^{\Z_-}$ to ${\bf y}_0(\omega)$ is continuous. On the other hand, for any ${\bf z} \in (D_d)^{\Z_-}$ the mapping $\omega \mapsto H_{\bf W(\omega)}^{{\bf A}(\omega),{\bf C}(\omega),\bm{\zeta}(\omega)}({\bf z})$ is $\Fc$-measurable and thus \cite[Lemma~4.51]{Aliprantis2006} implies that  $H_{{\bf W}}^{{\bf A},{\bf C},\bm{\zeta}}$ is product-measurable, i.e.\ the function $(\omega,{\bf z}) \ni \Omega \times (D_d)^{\Z_-} \mapsto H_{{\bf W}(\omega)}^{{\bf A}(\omega),{\bf C}(\omega),\bm{\zeta}(\omega)}({\bf z}) \in \R^{m}$ is $\Fc \otimes \mathcal{B}((D_d)^{\Z_-})$-measurable.

Writing ${\bf Y}$ for the associated process ${\bf y}$ with input ${\bf z} = {\bf Z}$, we note that for $\omega \in \Omega_{ESP}$ and $t \geq -T$ 
\[ {\bf Y}_t(\omega) = {\bf W}(\omega) F^{{\bf A}(\omega),{\bf C}(\omega),\bm{\zeta}(\omega)}(\mathbf{Y}_{t-1}(\omega),{\bf Z}_t(\omega)) \]
and consequently, by \eqref{eq:OmegaESPDef4}, $\|{\bf Y}_t(\omega)\| \leq M + 1$ for all $t \geq -T-1$. 

\textit{Risk estimation on $\Omega_{ESP}$:}

\noindent Firstly, by \eqref{eq:lossLipschitz} one has for any measurable $H \colon (D_d)^{\Z_-} \to \R^m$
\[ |\mathcal{R}(H) - \mathcal{R}(H^*)|  \leq L_L \E[\|H({\bf Z})- H^*({\bf Z}) \|]= L_L \int_{(D_d)^{\Z_-}} \|H({\bf z})- H^*({\bf z})\| \mu_{{\bf Z}}(\d {\bf z}). \]
Thus 
\begin{equation}\label{eq:auxEq51} \E[|\mathcal{R}(H_{\bf W}^{{\bf A},{\bf C},\bm{\zeta}}) - \mathcal{R}(H^*)|\mathbbm{1}_{\Omega_{ESP}}] \leq L_L \E[ \|H_{\bf W}^{{\bf A},{\bf C},\bm{\zeta}}({\bf Z}) - H^*({\bf Z})\|\mathbbm{1}_{\Omega_{ESP}}].  \end{equation}
For each $t \geq -T$ one estimates
\begin{equation}\label{eq:auxEq47}\begin{aligned}
\E[ &  \|{\bf Y}_t  - \mathbf{X}^*_{t}\|\mathbbm{1}_{\Omega_{ESP}}] \\ &  = \E[ \|\frac{1}{N}\sum_{i=1}^N {\bf W}_i \sigma({\bf A}^{(i)} \mathbf{Y}_{t-1} + {\bf C}^{(i)} {\bf Z}_t + \bm{\zeta}^{(i)}) - F^*(\mathbf{X}^*_{t-1},{\bf Z}_t)\|\mathbbm{1}_{\Omega_{ESP}}] 
\\ & \leq  \E[\sup_{{\bf y} \in B_{M+1}} \|\frac{1}{N}\sum_{i=1}^N {\bf W}_i \sigma({\bf A}^{(i)} {\bf y} + {\bf C}^{(i)} {\bf Z}_t + \bm{\zeta}^{(i)}) - F^{*,N}({\bf y},{\bf Z}_t)\|\mathbbm{1}_{\Omega_{ESP}}]  \\ &  \qquad+ \E[ \| F^{*,N}(\mathbf{Y}_{t-1},{\bf Z}_t)- F^*(\mathbf{Y}_{t-1},{\bf Z}_t)\|\mathbbm{1}_{\Omega_{ESP}}] \\ & \qquad + \E[ \| F^*(\mathbf{Y}_{t-1},{\bf Z}_t)- F^*(\mathbf{X}^*_{t-1},{\bf Z}_t)\|\mathbbm{1}_{\Omega_{ESP}}] 
\\ & \leq \E[\sup_{{\bf y} \in B_{M+1}} \|\frac{1}{N}\sum_{i=1}^N {\bf W}_i \sigma({\bf A}^{(i)} {\bf y} + {\bf C}^{(i)} {\bf Z}_t + \bm{\zeta}^{(i)}) - F^{*,N}({\bf y},{\bf Z}_t)\|] \\ &  \qquad + s_N +  r \E[ \|\mathbf{Y}_{t-1}- \mathbf{X}^*_{t-1}\|\mathbbm{1}_{\Omega_{ESP}}].  
 \end{aligned} \end{equation}
Denoting by $\varepsilon_1, \ldots, \varepsilon_N$ independent Rademacher random variables, we thus obtain by \eqref{eq:auxEq58}, independence and symmetrization that for any ${\bf z} \in D_d$ 
\begin{equation}\label{eq:auxEq46} \begin{aligned} \E[\sup_{{\bf y} \in  B_{M+1}} \|\frac{1}{N}\sum_{i=1}^N &  {\bf W}_i \sigma({\bf A}^{(i)} {\bf y} + {\bf C}^{(i)} {\bf z} + \bm{\zeta}^{(i)}) - F^{*,N}({\bf y},{\bf z})\|] \\ & \leq \sum_{j=1}^{N^*} \E[\sup_{{\bf y} \in  B_{M+1}} \left| \frac{1}{N}\sum_{i=1}^N V_j^{(i)} \sigma(( {\bf y},{\bf z},1)\cdot {\bf U}_{j}^{(i)}) - F_j^{*,N}({\bf y},{\bf z})\right|]
\\ & \leq 2 \sum_{j=1}^{N^*} \E[\sup_{{\bf y} \in  B_{M+1}} \left| \frac{1}{N}\sum_{i=1}^N V_j^{(i)} \varepsilon_i \sigma(( {\bf y},{\bf z},1)\cdot {\bf U}_{j}^{(i)})\right|].
\end{aligned} \end{equation}
Furthermore, for any $v_i \in \R$, ${\bf u}_i=({\bf a}_i,{\bf c}_i,{\zeta}_i) \in S$, $i=1,\ldots,N$ the contraction  principle \cite[Theorem~4.12]{Ledoux2013} (applied to the contractions $\sigma_i(x)=\mathbbm{1}_{\{v_i \neq 0\}} v_i \sigma(x \frac{1}{L_\sigma v_i})$) yields
\begin{equation}\label{eq:auxEq53}\begin{aligned} & \E[\sup_{{\bf y} \in B_{M+1}} \left|\sum_{i=1}^N v_i \varepsilon_i \sigma(( {\bf y},{\bf z},1)\cdot {\bf u}_{i})\right|] 
\\ &  = \E[\sup_{{\bf y} \in B_{M+1}} \left|\sum_{i=1}^N \varepsilon_i \sigma_i(L_\sigma v_i( {\bf y},{\bf z},1)\cdot {\bf u}_{i})\right|] \\ & \leq 2 L_\sigma \E[\sup_{{\bf y} \in B_{M+1}} \left|\sum_{i=1}^N v_i \varepsilon_i (( {\bf y},{\bf z},1)\cdot {\bf u}_{i})\right|] 
\\ & \leq 2 L_\sigma\left((M+1) \E[ \|\sum_{i=1}^N v_i \varepsilon_i {\bf a}_{i} \|] + \E[ |\sum_{i=1}^N v_i \varepsilon_i ({\bf c}_{i} \cdot {\bf z} + \zeta_i) |] \right) 
\\ & \leq 2 L_\sigma\left((M+1)(\sum_{i=1}^N v_i^2 \| {\bf a}_{i}\|^2)^{1/2} + \|{\bf z}\| \left( \sum_{i=1}^N v_i^2 \|{\bf c}_{i}\|^2 \right)^{1/2} +  (\sum_{i=1}^N v_i^2 | \zeta_i|^2 )^{1/2} \right). \end{aligned}
 \end{equation}
By conditioning, using independence and combining this with \eqref{eq:auxEq46} one thus obtains
\begin{equation}\label{eq:auxEq54}\begin{aligned}
& \E[  \sup_{{\bf y} \in B_{M+1}}   \|\frac{1}{N}\sum_{i=1}^N   {\bf W}_i \sigma({\bf A}^{(i)} {\bf y} + {\bf C}^{(i)} {\bf z} + \bm{\zeta}^{(i)}) - F^{*,N}({\bf y},{\bf z})\|] 
\\ &  \leq \frac{4 L_\sigma}{N} \sum_{j=1}^{N^*} \E[(M+1) (\sum_{i=1}^N (V_j^{(i)})^2 \| {\bf A}_j^{(i)}\|^2)^{1/2} + \|{\bf z}\| \left( \sum_{i=1}^N (V_j^{(i)})^2 \|{\bf C}_j^{(i)} \|^2 \right)^{1/2} \\ & \qquad  + (\sum_{i=1}^N (V_j^{(i)})^2| \zeta_j^{(i)}|^2)^{1/2}]
\\ & 
\leq \frac{4 L_\sigma}{\sqrt{N}} \sum_{j=1}^{N^*} \left[ (M+1) \E[(V_j^{(1)})^2 \| {\bf A}_j^{(1)}\|^2]^{1/2} + \|{\bf z}\| \E[(V_j^{(1)})^2 \|{\bf C}_j^{(1)}\|^2]^{1/2} \right. \\ & \qquad \left. + \E[(V_j^{(1)})^2 |\zeta_j^{(1)}|^2 ]^{1/2}\right].
\end{aligned}
\end{equation}
Inserting \eqref{eq:auxEq59} thus yields
\begin{equation}\label{eq:auxEq50}
\E[\sup_{{\bf y} \in B_{M+1}} \|\frac{1}{N}\sum_{i=1}^N {\bf W}_i \sigma({\bf A}^{(i)} {\bf y} + {\bf C}^{(i)} {\bf Z}_t + \bm{\zeta}^{(i)}) - F^{*,N}({\bf y},{\bf Z}_t)\|] \leq \frac{(M+2+\E[\|{\bf Z}_t \|])C_*}{\sqrt{N}}.
\end{equation}
Iterating \eqref{eq:auxEq47} $(T+1)$-times and inserting \eqref{eq:auxEq50}  yields
\begin{equation*} \begin{aligned}
\E[ &  \|{\bf Y}_0  - \mathbf{X}^*_{0}\| \mathbbm{1}_{\Omega_{ESP}}] \\ & \leq \sum_{k=0}^{T} r^k \left[\frac{(M+2+\E[\|{\bf Z}_{-k} \|])C_*}{\sqrt{N}} + s_N \right] +  r^{T+1} \E[ \|\mathbf{Y}_{-T-1}- \mathbf{X}^*_{-T-1}\|\mathbbm{1}_{\Omega_{ESP}}] 
\\ & \leq \frac{(M+2+\max_{t \in \{0, \ldots,-T\}}\E[\|{\bf Z}_{t} \|])C_*}{(1-r)\sqrt{N}} + \frac{s_N}{1-r} + 2 (M+1) r^{T+1}.
\end{aligned} \end{equation*}
Noting that ${\bf Y}_0 = H_{\bf W}^{{\bf A},{\bf C},\bm{\zeta}}({\bf Z})$, \eqref{eq:auxEq50} and \eqref{eq:auxEq51} hence prove that
\begin{equation}\label{eq:auxEq56}\begin{aligned}
\E[& |\mathcal{R}(H_{\bf W}^{{\bf A},{\bf C},\bm{\zeta}}) - \mathcal{R}(H^*)  |\mathbbm{1}_{\Omega_{ESP}}] \\ & \leq L_L \left[ \frac{(M+2+\max_{t \in \{0, \ldots,-T\}}\E[\|{\bf Z}_{t} \|])C_*}{(1-r)\sqrt{N}} + \frac{s_N }{1-r}  + 2 (M+1) r^{T+1} \right]. \end{aligned}
\end{equation}
\textit{Estimating $\P(\Omega \setminus \Omega_{ESP})$:}

It thus remains to prove that  the probability that the random ESN parameters lie in $\Omega_{ESP}$ increases to $1$ at rate $1/\sqrt{N}$. To this end, first note that for any $\mathbf{x} \in \R^{N^*},{\bf z} \in D_d$
\[  \|{\bf W} F^{{\bf A},{\bf C},\bm{\zeta}}(\mathbf{x},{\bf z})\| \leq \|{\bf W} F^{{\bf A},{\bf C},\bm{\zeta}}(\mathbf{x},{\bf z}) - F^{*,N}(\mathbf{x},{\bf z}) \| + \|F^{*,N}(\mathbf{x},{\bf z})-F^{*}(\mathbf{x},{\bf z})\| + M \]
and therefore  
\begin{equation}\label{eq:auxEq55}\begin{aligned}  & \P( \Omega \setminus \Omega_{ESP})  \\ & \leq \P(\bar{M}\geq M+1)\\  &  \leq \P\left(\sup_{\substack{\mathbf{x} \in B_{M+1} \\ t\in \{0,\ldots,-T\}}}\|{\bf W} F^{{\bf A},{\bf C},\bm{\zeta}}(\mathbf{x},{\bf Z}_t) -  F^{*,N}(\mathbf{x},{\bf Z}_t) \| + \|F^{*,N}(\mathbf{x},{\bf Z}_t)-F^{*}(\mathbf{x},{\bf Z}_t)\| \geq 1 \right) 
\\ & \leq \P\left(\sup_{\substack{\mathbf{x} \in B_{M+1} \\ t\in \{0,\ldots,-T\}}}\|{\bf W} F^{{\bf A},{\bf C},\bm{\zeta}}(\mathbf{x},{\bf Z}_t) -  F^{*,N}(\mathbf{x},{\bf Z}_t) \| \geq \frac{1}{2} \right) \\ & \quad \quad  + \P\left(\sup_{\substack{\mathbf{x} \in B_{M+1} \\ t\in \{0,\ldots,-T\}}} \|F^{*,N}(\mathbf{x},{\bf Z}_t)-F^{*}(\mathbf{x},{\bf Z}_t)\| \geq \frac{1}{2} \right)
\\ & \leq 2\E\left[\sup_{\substack{\mathbf{x} \in B_{M+1} \\ t\in \{0,\ldots,-T\}}} \| {\bf W} F^{{\bf A},{\bf C},\bm{\zeta}}(\mathbf{x},{\bf Z}_t) - F^{*,N}(\mathbf{x},{\bf Z}_t) \|  \right] + 2 s_N.
\\ & = 2 \E\left[\E\left[\sup_{\substack{\mathbf{x} \in B_{M+1} \\ {\bf v} \in \{{\bf z}_0, \ldots,{\bf z}_{-T}\}}} \| {\bf W} F^{{\bf A},{\bf C},\bm{\zeta}}(\mathbf{x},{\bf v}) - F^{*,N}(\mathbf{x},{\bf v}) \| \right]_{{\bf z}={\bf Z}} \right] + 2 s_N.
\end{aligned} \end{equation}
The inner expectation can now be estimated using precisely the same arguments as in \eqref{eq:auxEq46}, \eqref{eq:auxEq53}, \eqref{eq:auxEq54} yielding for any ${\bf z} \in (D_d)^{\Z_-}$
\begin{equation*}
%\label{eq:auxEq52} 
\begin{aligned}
\E & \left[\sup_{\substack{\mathbf{x} \in B_{M+1} \\ {\bf v} \in \{{\bf z}_0, \ldots,{\bf z}_{-T}\}}} \| {\bf W} F^{{\bf A},{\bf C},\bm{\zeta}}(\mathbf{x},{\bf v}) - F^{*,N}(\mathbf{x},{\bf v}) \| \right] \\
&  \leq 2 L_\sigma \sum_{j=1}^{N^*} \E\left[\sup_{\substack{\mathbf{y} \in B_{M+1} \\ {\bf v} \in \{{\bf z}_0, \ldots,{\bf z}_{-T}\}}} \left| \frac{1}{N}\sum_{i=1}^N V_j^{(i)} \varepsilon_i \sigma(( {\bf y},{\bf v},1)\cdot {\bf U}_{j}^{(i)})\right|\right]
\\
&  \leq 4 L_\sigma \sum_{j=1}^{N^*} \E\left[\sup_{\substack{\mathbf{y} \in B_{M+1} \\ {\bf v} \in \{{\bf z}_0, \ldots,{\bf z}_{-T}\}}} \left| \frac{1}{N}\sum_{i=1}^N V_j^{(i)} \varepsilon_i (( {\bf y},{\bf v},1)\cdot {\bf U}_{j}^{(i)}) \right|\right]
\\
&  \leq \frac{4 L_\sigma}{\sqrt{N}} \sum_{j=1}^{N^*} (M+1)\E[(V_j^{(1)})^2 \| {\bf A}_j^{(1)}\|^2]^{1/2} + \left(\max_{t \in \{0, \ldots,-T\}} \|{\bf z}_t\|\right) \E[(V_j^{(1)})^2 \| {\bf C}_j^{(1)}\|^2]^{1/2} \\ & \quad \quad \quad + \E[(V_j^{(1)})^2 |\zeta_j^{(1)}|^2]^{1/2}.
\end{aligned}
\end{equation*}
Combining this with \eqref{eq:auxEq55} yields
\begin{equation} \label{eq:auxEq57} \begin{aligned}
\P(\Omega \setminus \Omega_{ESP}) & \leq \frac{2(M+2+\E\left[\max_{t \in \{0, \ldots,-T\}} \|{\bf Z}_t\|\right])C_*}{\sqrt{N}}+2 s_N.
\end{aligned}
\end{equation}
\textit{Putting together the ingredients:}\\
\noindent Altogether, setting 
\[\eta = \frac{L_L}{\delta} \left[ \frac{(M+2 + \max_{t \in \{0, \ldots,-T\}}\E[\|{\bf Z}_t \|]) C_*}{(1-r)\sqrt{N}} + \frac{s_N }{1-r} + 2 (M+1) r^{T+1} \right] \]
and combining \eqref{eq:auxEq56} and \eqref{eq:auxEq57} yields
\[ \begin{aligned}  \P & \left( |\mathcal{R}(H_{\bf W}^{{\bf A},{\bf C},\bm{\zeta}}) - \mathcal{R}(H^*)| > \eta  \right)  \\ & \leq \P \left( |\mathcal{R}(H_{\bf W}^{{\bf A},{\bf C},\bm{\zeta}}) - \mathcal{R}(H^*)|\mathbbm{1}_{\Omega_{ESP}} > \eta  \right) + \P(\Omega \setminus \Omega_{ESP})
\\ & \leq \delta + \frac{2(M+2+\E\left[\max_{t \in \{0, \ldots,-T\}} \|{\bf Z}_t\|\right])C_*}{\sqrt{N}} + 2 s_N. \end{aligned} \]\qed
\end{proof}

\begin{acknowledgements}
We thank Josef Teichmann for fruitful discussions that helped in improving the paper. Lukas G and JPO acknowledge partial financial support  coming from the Research Commission of the Universit\"at Sankt Gallen and the Swiss National Science Foundation (grant number 200021\_175801/1). Lyudmila G acknowledges partial financial support of the Graduate School of Decision Sciences of the Universit\"at Konstanz. JPO acknowledges partial financial support of the French ANR ``BIPHOPROC" project (ANR-14-OHRI-0002-02). The three authors thank the hospitality and the generosity of the FIM at ETH Zurich where a significant portion of the results in this paper were obtained.
\end{acknowledgements}
%\gls{h}
%\glossarystyle{list}
%\printglossary[title=Glossary of Symbols,nonumberlist]
%\section*{Conflict of interest}
%The authors declare that they have no conflict of interest.

%\bibliography{/Users/JPO/Dropbox/Public/Gonon_Ortega}
%\bibliography{/Users/JPO/Dropbox/Public/GOLibrary}
%\bibliography{Gonon_Ortega}

\end{document}